\documentclass[11pt,reqno]{amsart}

\usepackage{amsmath}
\usepackage{amssymb}
\usepackage{amsthm}
\usepackage{mathtools}
\usepackage{url}
\usepackage{enumerate}
\usepackage{color}
\usepackage{hyperref}

\usepackage[margin=3cm]{geometry}

\usepackage{booktabs}

\def\Lam{\Lambda}

\def\R{\mathbb{R}}

\def\Z{\mathbb{Z}}
\def\eps{\epsilon}

\def\cE{\mathcal E}

\def\cH{\mathcal H}
\def\cG{\mathcal G}
\def\cC{\mathcal C}

\def\cL{\mathcal L}

\def\tor{\mathbb T^d_n}

\newcommand{\bb}[1]{\mathbb{#1}}
\newcommand{\N}{\bb N}

\newcommand{\set}[1]{\left\{#1\right\}}
 
\newcommand{\norm}[1]{\|#1\|} 

\newcommand{\Contour}{\Gamma}
\newcommand{\Interface}{\mathcal{S}}
\newcommand{\dis}{\text{dis}}
\newcommand{\ord}{\text{ord}}
\newcommand{\ext}{\text{ext}}

\newcommand{\tunnel}{\text{tunnel}}
\newcommand{\Ext}{\text{Ext}\,}
\newcommand{\Int}{\text{Int}\,}
\newcommand{\ctor}{\V{T}^d_n}

\newcommand{\pc}{\tilde\Omega}

\newtheorem*{theorem*}{Theorem}
\newtheorem{theorem}{Theorem}[section]
\newtheorem{lemma}[theorem]{Lemma}

\newtheorem{defn}{Definition}
\newtheorem{prop}[theorem]{Proposition}
\newtheorem*{prop*}{Proposition}

\newtheorem*{claim*}{Claim}
\newtheorem*{fact*}{Fact}
\newtheorem*{remark*}{Remark}
\newtheorem*{defn*}{Definition}

\theoremstyle{definition}

\newcommand{\abs}[1]{\left|#1\right|}

\newcommand{\bydef}{\coloneqq}

\newcommand{\V}[1]{{\boldsymbol #1 }}

\newcommand{\htor}{\frac{1}{2}\tor}
\newcommand{\htors}{(\htor)^\star}
\newcommand{\rest}{\text{rest}}

\usepackage{tikz}
\usetikzlibrary{patterns}

\tikzset{ boundary/.style={ very thin, gray} }
\tikzset{ lam/.style={very thick,black} }
\tikzset{ lamb/.style={line width=4pt,black} }

\tikzset{ boundaryfillr/.style={ very thin,pattern color=red,
    pattern=north west lines,opacity=.8} }
\tikzset{ boundaryfillb/.style={ very thin,pattern color=blue, pattern=north west lines,
    opacity=0.8} }
\tikzset{ boundaryfillg/.style={ very thin,pattern color=green, pattern=north west lines,
    opacity=0.8} }

\tikzset{ bv/.style={circle,fill=blue!50,draw,thick,minimum size=12pt,
    inner sep=0}}
\tikzset{ rv/.style={circle,fill=red!50,draw,thick,minimum
    size=12pt,inner sep=0}}
\tikzset{ gv/.style={circle,fill=green!50,draw,thick, minimum size=12pt,
    inner sep=0}}
\tikzset{ ov/.style={circle,fill=black!50,draw,thick,minimum size=10pt,
    inner sep=0}}
\begin{document}
\title[Efficient sampling and counting for the Potts model on
$\Z^d$]{Efficient sampling and counting algorithms for the Potts model on $\Z^d$ at all temperatures}

\thanks{An extended abstract of this paper appeared at STOC 2020~\cite{BorgsStoc2020}.}

\author[Borgs]{Christian Borgs}
\author[Chayes]{Jennifer Chayes}
\author[Helmuth]{Tyler Helmuth}
\author[Perkins]{Will Perkins}
\author[Tetali]{Prasad Tetali}

\address{University of California Berkeley}
\email{borgs@berkeley.edu}
\address{University of California Berkeley}
\email{jchayes@berkeley.edu}
\address{University of Durham}
\email{tyler.helmuth@durham.ac.uk}
\address{Georgia Institute of Technology}
\email{math@willperkins.org}
\address{Carnegie Mellon University}
\email{ptetali@cmu.edu}


\begin{abstract}
  For $d \ge 2$ and all $q\geq q_{0}(d)$ we give an efficient
  algorithm to approximately sample from the $q$-state ferromagnetic
  Potts and random cluster models on finite tori $(\Z / n \Z )^d$ for
  any inverse temperature $\beta\geq 0$.  This shows that the physical
  phase transition of the Potts model presents no algorithmic barrier to efficient sampling,
  and stands in contrast to Markov chain mixing time results: the Glauber
  dynamics mix slowly at and below the critical temperature, and the
  Swendsen--Wang dynamics mix slowly at the critical temperature. We
  also provide an efficient algorithm (an FPRAS) for approximating the
  partition functions of these models at all temperatures.

  Our algorithms are based on representing 
  the random cluster model as a contour model using Pirogov--Sinai
  theory, and then computing an accurate approximation of the
  logarithm of the partition function by inductively truncating the
  resulting cluster expansion. The main innovation of our approach is
  an algorithmic treatment of unstable ground states, which  is
  essential for our algorithms to apply to all inverse temperatures $\beta$. By treating
  unstable ground states our work gives a general template for
  converting probabilistic applications of Pirogov--Sinai theory to
  efficient algorithms.
\end{abstract}

\maketitle

\section{Introduction}
\label{sec:intro}

The Potts model is a probability distribution on assignments of $q$ colors to
the vertices of a finite graph $G$. For $\sigma \in [q]^{V(G)}\bydef
\{1,2,\dots, q\}^{V(G)}$ let
\begin{equation}
H_G(\sigma) \bydef \sum_{(i,j)\in E(G)} \delta_{\sigma_i \ne \sigma_j}
\end{equation}
be the the number of bichromatic edges of $G$ under the coloring
$\sigma$. The \emph{$q$-state ferromagnetic Potts model} at
inverse temperature $\beta\geq 0$ is the probability distribution
$\mu^{\text{Potts}}_G$ on $[q]^{V(G)}$ defined by
\begin{equation}
  \label{eq:PottsDef}
  \mu^{\text{Potts}}_G(\sigma) \bydef  \frac{e^{-\beta
      H_G(\sigma)}}{Z^{\text{Potts}}_G(\beta)}, \qquad 
  Z^{\text{Potts}}_G(\beta) \bydef \sum_{\sigma \in [q]^{V(G)}} e^{-\beta H_G(\sigma)}.
\end{equation}
The normalizing constant $Z^{\text{Potts}}_{G}(\beta)$ is the Potts
model partition function.  Since $\beta \geq 0$, monochromatic edges are
preferred. This is often referred to as the \emph{ferromagnetic} Potts
model.

In this paper we are interested in computational aspects of the Potts
model. To this end, we view $Z^{\text{Potts}}_G$ and
$\mu^{\text{Potts}}_G$ as functions and probability
measures indexed by finite graphs $G$, and consider 
two 
computational tasks associated to these objects. The first is the
\emph{approximate counting} problem: for a partition function $Z_G$
and error tolerance $\eps>0$, compute a number $\hat Z$ so that
$e^{-\eps} \hat Z \le Z_G \le e^{\eps} \hat Z$.  We say that such a
$\hat Z$ is an \emph{$\eps$-relative approximation} to $Z_G$.  The
second is the \emph{approximate sampling} problem: for a probability measure
$\mu_G$ and error tolerance $\eps>0$, output a random configuration
$\hat \sigma$ with distribution $\hat \mu$ so that
$\| \hat \mu - \mu_G \|_{TV} < \eps$. We say $\hat \sigma$ is an
\emph{$\eps$-approximate sample} from $\mu_G$.

Approximate counting and sampling algorithms can always be obtained by
brute force in time exponential in the size of the graph, and the
interesting question is if more efficient algorithms exist. To
formalize this, a \emph{fully polynomial-time approximation scheme}
(FPTAS) is an algorithm that given $G$ and $\eps>0$ returns an
$\eps$-relative approximation to $Z_G$ and runs in time polynomial in
$|V(G)|$ and $1/\eps$.  If the algorithm uses randomness it is a
\emph{fully polynomial-time randomized approximation scheme} (FPRAS).
A randomized algorithm that given $G$ and $\eps>0$ outputs an
$\eps$-approximate sample from $\mu_G$ and runs in time polynomial in
both $|V(G)|$ and $1/\eps$ is an \emph{efficient sampling
  scheme}. These notions are standard in the study of the
computational complexity of approximate sampling and counting, see
Section~\ref{sec:survey}
below. 

One of the main result of this paper is the development of an FPRAS and an
efficient sampling scheme for the $q$-state Potts model on the
discrete tori $\tor = (\Z / n \Z )^d$ for \emph{all} inverse
temperatures $\beta\geq 0$, provided $q$ is large enough as a function
of $d$.
\begin{theorem}
  \label{PottsTorusCrit}
  For all $d\ge2$ there exists $q_0=q_{0}(d)$ such that for
  $q \ge q_{0}$ and all inverse temperatures $\beta\geq 0$ there is an
  FPRAS and efficient sampling scheme for the $q$-state Potts model at
  inverse temperature $\beta$ on the torus $\tor$.
\end{theorem}
If $\epsilon$ is not too small, meaning $\eps\geq \exp(-O(n^{d-1}))$,
our approximate counting algorithm is deterministic. We will comment on
this further in what follows, see below Theorem~\ref{PottsZd}. 

In the next section we discuss the context and motivation behind
Theorem~\ref{PottsTorusCrit}. The remainder of the introduction then
turns to some relevant facts about the Potts model (Section~\ref{sec:potts-model-zd}), the closely
related random cluster model (Section~\ref{sec:random-cluster-model}), and a discussion of our proof strategy
and the main challenges in proving
these results (Section~\ref{sec:meth}). 

\subsection{Context and motivation: approximation algorithms and
  computational phase transitions}
\label{sec:survey}

For many statistical mechanics models like the Potts model
\emph{exact} computation of the partition function  has long been known to be \#P-hard,
even for restricted classes of graphs and parameters. In particular,
if P$\neq$NP, this task cannot be performed in polynomial
time.  Current research therefore is focused on approximate counting.  

 For some special models (the ferromagnetic Ising model~\cite{jerrum1993polynomial}, the monomer-dimer model~\cite{jerrum1989approximating}), there is an FPRAS for all graphs and all parameters.  For other models, the computational complexity of approximate counting and sampling depends on the class of graphs and on the parameters of the model.  These models exhibit \emph{computational phase
  transitions}.  We now
briefly introduce a well-known example of such a transition.  Recall
that a subset $I\subset V$ of vertices of a graph $G=(V,E)$ is
independent if no two vertices in $I$ are joined by an
edge. Given $\lambda>0$, the independent set (or hard-core) model with
fugacity $\lambda$ is the probability distribution on
independent sets that chooses $I$ with probability proportional to
$\lambda^{|I|}$. An important series of results in the field of
approximate counting has established the existence of a computational
phase transition for the independent set model. More precisely,
restrict the set of input graphs to be those of maximum degree
$\Delta$, and let
$\lambda_{c} \bydef
\frac{(\Delta-1)^{\Delta-1}}{(\Delta-2)^{\Delta}}$. Then there exists
an FPTAS and an efficient sampling scheme if
$\lambda<\lambda_{c}$~\cite{weitz2006counting}, while there does not
if $\lambda>\lambda_{c}$ unless
NP=RP~\cite{sly2010computational,sly2014counting,galanis2016inapproximability}. The
parameter $\lambda_{c}$ also appears in statistical physics. Namely,
it is the point where the independent set model has a phase transition
on the $\Delta$-regular tree in the sense of uniqueness
($\lambda<\lambda_{c}$) and non-uniqueness ($\lambda>\lambda_{c}$) of
Gibbs measures. The hardness result is obtained by a reduction to
MAX-CUT, an NP-hard problem.

A third class of model lies between the other two: those for which no FPRAS is known in general, but no computational hardness is known either.  An important example of such a model is the
independent set model when one restricts the  inputs  to
\emph{bipartite} graphs. Counting independent sets in bipartite graphs
is called \#BIS, and many approximate counting problems of interest
turn out to be equivalent to the existence of an FPRAS for \#BIS,
see~\cite{dyer2004relative}. 
It has been conjectured that no FPRAS exists for \#BIS.  The
connection to the Potts models, and hence the present work, is as
follows. Fix $q\geq 3$. The existence of an FPRAS for the $q$-state
Potts model on graphs of maximum degree $\Delta$ at large enough inverse temperature $\beta$ would imply the existence of an FPRAS for \#BIS~\cite[Theorem 2]{galanis2016ferromagnetic}.   
The conjecture, therefore, is that no such
FPRAS for the $q$-state Potts model exists.

Until recently the construction of  efficient approximate counting and sampling
schemes for statistical physics models was largely restricted to the
uniqueness regime of the respective models, e.g., via Markov-chain
mixing or correlation decay arguments. Notable exceptions include the Ising and monomer-dimers models and special classes of graphs
with dualities, e.g., planar duality~\cite{ullrich2013comparison,gheissari2018mixing,gheissari2016quasi,blanca2017random}. Recently, efficient
algorithms in non-uniqueness regimes have been developed. These
algorithms are primarily based on 
the observation that classical tools from mathematical physics, the
cluster expansion and Pirogov--Sinai theory, can be used to obtain
efficient algorithms deep inside the non-uniqueness
regime on lattices~\cite{helmuth2018contours}.  Further works~\cite{JenssenAlgorithmsJ,cannon2019counting,PolymerMarkov} extended
the use of the cluster expansion to obtain algorithms for other classes
of graphs and models for parameters, again deep inside non-uniqueness
regimes. 

As will be discussed in Section~\ref{sec:meth} below, our proof of
Theorem~\ref{PottsZd} relies on a significant extension of the
Pirogov--Sinai methodology of~\cite{helmuth2018contours}. Note that
Theorem~\ref{PottsZd} completely rules out the existence of a
computational phase transition for the $q$-state Potts model on tori
$(\Z/n\Z)^{d}$ when $q\gg 1$. The Potts model on tori has a
uniqueness/non-uniqueness phase transition in the infinite volume limit and hence our
result shows that any relation between computational and physical
phase transitions for the Potts model is subtle, in that it sensitive
to the class of graphs being considered. It is important to note that
this potential sensitivity is not new, as it also follows from algorithmic
results~\cite{jerrum1993polynomial} concerning the Ising model. Our
contribution, therefore, is a proof of this subtlety in a less
specialized context, by fairly robust methods, and for a problem
directly related to \#BIS. While the tori we consider are
rather special graphs, we view them as a starting point for
understanding potential barriers to the existence of an FPRAS for the
Potts model, and hence to understanding  the existence or non-existence of an FPRAS for \#BIS.

After the appearance of the extended abstract of this work in~\cite{BorgsStoc2020}, results
concerning all-temperature algorithms for the Potts model on expander
graphs appeared~\cite{HJP}. The methods of~\cite{HJP} are different than those of the present paper, as the
geometry of expander graphs allows one to avoid the use of
Pirogov--Sinai theory and work with simpler polymer models instead.

\subsection{The Potts model on $\Z^d$}
\label{sec:potts-model-zd}

The Potts model is known to exhibit a phase transition on $\Z^{d}$
when $d\geq 2$, and when $q$ is sufficiently large the phase diagram
has been completely understood for some
time~\cite{kotecky1982first,laanait1991interfaces}. For large $q$
there is a critical temperature $\beta_{c} = \beta_c(d,q)$ satisfying
\begin{equation}
  \label{eq:betac}
  \beta_{c} = \frac{\log q}{d} + O(q^{-1/d})
\end{equation}
such that for $\beta<\beta_{c}$ there is a unique infinite-volume
Gibbs measure, while if $\beta>\beta_{c}$ there are $q$ extremal
translation-invariant Gibbs measures. Each of these low-temperature
measures favor one of the $q$ colors. At the transition point
$\beta= \beta_{c}$ there are $q+1$ extremal translation-invariant
Gibbs measures; $q$ of these measures favor one of the $q$ colors, and
the additional measure is the `disordered' measure from
$\beta<\beta_{c}$. We note that the phenomenology of the model is
$q$-dependent~\cite{DuminilCopinLectures}. The preceding results
require $q$ large as they use $q^{-1}$ as a small parameter in proofs.

The existence of multiple measures in the low-temperature phase is
reflected in the dynamical aspects of the model. While Glauber
dynamics for the Potts model mix rapidly at sufficiently high
temperatures, they mix in time $\exp(\Theta(n^{d-1}))$ when
$\beta\geq\beta_{c}$~\cite{borgs2012tight,borgs1991finite}. Even the
global-move Swensden--Wang dynamics take time $\exp(\Theta(n^{d-1}))$
to mix when $\beta=\beta_{c}$~\cite{borgs2012tight}.

\subsection{Random cluster model}
\label{sec:random-cluster-model}

Given a finite graph
$G=(V(G),E(G))$ the \emph{random cluster model} is a probability
distribution on edge sets of $G$ given by
\begin{equation}
\label{eqRCdef}
  \mu^{\text{RC}}_{G}(A) \bydef \frac{p^{|A|} (1-p)^{|E(G)| - |A|} q^{c(G_A)}  }{Z^{\text{RC}}_G(p,q)   } \,, \quad\quad A \subseteq E(G) \,,
\end{equation}
where $c(G_A)$ is the number of connected components of the graph
$G_A = (V(G),A)$ and
\begin{equation}
  Z^{\text{RC}}_{G}(p,q) \bydef \sum_{A \subseteq E(G)}  p^{|A|} (1-p)^{|E(G)| - |A|} q^{c(G_{A})}  
\end{equation}
is the random cluster model partition function.

The Potts model and the random cluster model can be put onto the same
probability space via the Edwards--Sokal coupling (see, e.g.,~\cite{DuminilCopinLectures}). We recall
this coupling in Appendix~\ref{sec:ES}; one consequence is the relation,
for $\beta \ge 0$ and integer $q \ge 2$,
\begin{equation}
\label{eqFKpotts1}
  Z^{\text{Potts}}_{G}(\beta) = e^{\beta |E(G)|} Z^{\text{RC}}_G(1-e^{-\beta},q).
\end{equation}

With the parameterization $p = 1- e^{-\beta}$ the random cluster model
on $\Z^d$, $d \ge 2$, also has a critical inverse temperature
$\beta_c= \beta_c(q,d)$ that satisfies \eqref{eq:betac} and that
coincides with the Potts critical inverse temperature for integer $q$.
For $\beta<\beta_c$ the random cluster model has a unique infinite
volume measure (the \emph{disordered} measure), while for
$\beta>\beta_c$ the \emph{ordered} measure is the unique infinite
volume measure. For $\beta= \beta_c$ the two measures coexist, in the
sense that there are multiple infinite-volume Gibbs measures, with one
corresponding to the ordered and one corresponding to the disordered
measure.

Our counting and sampling algorithms for the Potts model extend to the
random cluster model on finite subgraphs of $\Z^d$ with two different
types of boundary conditions. To make this precise requires a few
definitions.  Let $\Lam$ be a finite set of vertices of $\Z^d$ and let
$G_\Lam$ be the subgraph of $\Z^d$ induced by $\Lam$.  We say $G_{\Lam}$ is
\emph{simply connected} if $G_{\Lam}$ is connected and the subgraph
induced by $\Lam^c =\Z^d \setminus \Lam$ is connected.  The random
cluster model with \emph{free boundary conditions} on $G_{\Lam}$ is
just the random cluster model on the induced subgraph $G_{\Lam}$ as
defined by~\eqref{eqRCdef}.  The random cluster model with \emph{wired
  boundary conditions} on $G_{\Lam}$ is the random cluster model on
the (multi-)graph $G_{\Lam}'$ obtained from $G_{\Lam}$ by identifying
all of the vertices on the boundary of $\Lam$ to be one vertex;
see~\cite[Section~1.2.2]{DuminilCopinLectures} for a formal
definition.  We refer to the Gibbs measures and partition functions
with free and wired boundary conditions as
$\mu^f_\Lam, \mu^w_\Lam, Z^f_{\Lam}, Z^w_{\Lam}$. Explicitly,
\begin{align}
  \label{eq:Zfree}
  Z^f_{\Lam} &\bydef \sum_{A\subset E(G_{\Lam})}
  p^{\abs{A}}(1-p)^{\abs{E(G_{\Lam})}-\abs{A}} q^{c(G_{A})}, \qquad \text{and} \\
  \label{eq:Zwired}
  Z^w_{\Lam}
  &\bydef \sum_{A\subset E(G_{\Lam}')}
  p^{\abs{A}}(1-p)^{\abs{E(G_{\Lam}')}-\abs{A}} q^{c(G'_{A})}, 
\end{align}
where $c(G_{A})$ is the number of connected components of the graph
$(\Lam,A)$ and $c(G'_A)$ is the number of components of the graph
$(\Lam',A)$ in which we identify all vertices on the boundary of $\Lam$.

\begin{theorem}
  \label{PottsZd}
  For $d\ge2$ there exists $q_0 = q_{0}(d)$ so that for $q \ge q_{0}$
  the following is true.

  For $\beta\geq \beta_{c}$ there is an FPTAS and efficient sampling
  scheme for the random cluster model on all finite, simply connected
  induced subgraphs of $\Z^d$ with wired boundary conditions.

  For $\beta \leq \beta_c$ there is an FPTAS and efficient sampling
  scheme for the random cluster model on all finite, simply connected
  induced subgraphs of $\Z^d$ with free boundary conditions.
\end{theorem}

Theorem~\ref{PottsZd} yields an FPTAS, while
Theorem~\ref{PottsTorusCrit} gave an FPRAS for the torus. The reason
for this is that our Pirogov--Sinai based methods become more
difficult to implement on the torus if the error parameter $\epsilon$
is smaller than $\exp(-O(n^{d-1}))$. The algorithm for
Theorem~\ref{PottsTorusCrit} circumvents this by making use of the
Glauber dynamics for this range of $\epsilon$. This is possible
because, despite being slow mixing, the Glauber dynamics are fast
enough when given time $O(\eps^{-1})$ for $\epsilon$ this small
by~\cite{borgs2012tight}. By using Glauber dynamics in a similar
manner we could obtain an FPRAS for the random cluster model on
$\tor$.

We note that our methods are likely capable of handling boundary
conditions other than those described above, but we leave an
investigation of the full scope of their applicability for the future.

\subsection{Proof overview}
\label{sec:meth}

The results of this paper are based on non-trivial extensions of the
recent work~\cite{helmuth2018contours}. To discuss the new
ingredients, we first recall two key ideas
from~\cite{helmuth2018contours}. The first, which has since gone on to
be used in many subsequent
works~\cite{JenssenAlgorithmsJ,cannon2019counting,casel2019zeros,PolymerMarkov,liao2019counting},
is the notion of a \emph{polymer model algorithm}. We discuss this
method in a self-contained way in Section~\ref{secPolymer} below; it
is based on the well-developed ideas of polymer models and cluster
expansion from mathematical
physics~\cite{gruber1971general,kotecky1986cluster}. In~\cite{helmuth2018contours}
this was combined with Barvinok's interpolation
method~\cite{barvinok2017combinatorics} to devise efficient
algorithms. Polymer model algorithms are efficient algorithms for
estimating the partition function of low-density independent set
models. The power of the method is that it can handle independent sets
on very general graphs with vertex-dependent activities. Many problems
of interest can be rephrased in terms of independent set models of
this type.

The second key idea from~\cite{helmuth2018contours} for this work is
the algorithmic use of Pirogov--Sinai theory. An important ingredient
for this is the notion of a \emph{ground state}. Formally, for the
Potts models, the ground states are the colourings $\sigma$ of
$\Z^{d}$ that minimize $H_{\Z^{d}}(\sigma)$. Rigorously, the ground
states are colourings $\sigma$ for which any finite perturbation
$\sigma'$ satisfies $H_{G}(\sigma')-H_{G}(\sigma)>0$ as
$G\uparrow \Z^{d}$; since $\sigma'$ is a finite perturbation this
sequence is constant for large enough volumes $G$. The ground states
of the ferromagnetic Potts model are the $q$ monochromatic
colorings.

This notion of a ground state is meant to capture the intuition that
when $\beta\gg 1$, one expects a typical configuration of the Potts
model to look essentially like one the ground states, with some small
local deviations. Rigorously verifying this picture is non-trivial,
and is part of the subject of Pirogov--Sinai theory. The key output of
the theory is a convergent expansion for the logarithm of the
partition function of the model with monochromatic boundary
conditions, where the terms of the expansion correspond to local
deviations from the given ground state. The expansion has a recursive
flavor: the terms of the expansion are themselves given by ratios of
partition functions with different boundary conditions. This recursion
can be traced back to the fact that local deviations can have internal
structures: there could be a red island inside of a blue lake inside
of a red sea. See Figure~\ref{fig:Potts}.  The algorithms
of~\cite{helmuth2018contours} made use of the symmetry of the ground
states of the Potts model in handling this recursion, the key
point being that symmetry implies (when $\beta \gg 1$) the deviations
are rare enough that their contribution to the relevant partition
functions can be controlled by a convergent cluster expansion. 

\begin{figure}
    \begin{tikzpicture}[scale=.8]
    \draw[boundary] (-1,-1) grid (16,9);
    \draw[lam] (0,0) -- (0,8) -- (15,8) -- (15,0) -- (0,0);

\draw[boundaryfillr,rounded corners=2] (-.5,-.5) --  (15.5,-.5) --
(15.5,8.5) -- (-.5,8.5) -- (-.5,-.5);

\draw[fill=white,color=white,rounded corners=2](.5,.5) -- (.5, 2.5) --
(1.5,2.5) -- (1.5,1.5) -- (2.5,1.5) -- (2.5,.5) -- (.5,.5);
\draw[boundaryfillb,rounded corners=2] (.5,.5) -- (.5, 2.5) --
(1.5,2.5) -- (1.5,1.5) -- (2.5,1.5) -- (2.5,.5) -- (.5,.5);

\draw[fill=white,color=white,rounded corners=2](2,3.5) -- (1.5,3.5) --
(1.5,5.5) -- (2.5, 5.5) -- (2.5,6.5) -- (3.5,6.5) -- (3.5,5.5) --
(4.5,5.5) -- (4.5,6.5) -- (6.5,6.5) -- (6.5,5.5) -- (7.5,5.5) --
(7.5,4.5) -- (9.5, 4.5) -- (9.5, 1.5) -- (4.5, 1.5) -- (4.5,2.5)
--(2.5, 2.5) -- (2.5,3.5) -- (2,3.5);
\draw[boundaryfillb,rounded corners=2] (2,3.5) -- (1.5,3.5) --
(1.5,5.5) -- (2.5, 5.5) -- (2.5,6.5) -- (3.5,6.5) -- (3.5,5.5) --
(4.5,5.5) -- (4.5,6.5) -- (6.5,6.5) -- (6.5,5.5) -- (7.5,5.5) --
(7.5,4.5) -- (9.5, 4.5) -- (9.5, 1.5) -- (4.5, 1.5) -- (4.5,2.5)
--(2.5, 2.5) -- (2.5,3.5) -- (2,3.5);

\draw[fill=white,color=white,rounded corners=2] (4.5,3.5) rectangle (5.5,4.5);
\draw[boundaryfillr,rounded corners=2] (5,3.5) -- (4.5,3.5) --
(4.5,4.5) -- (5.5,4.5) -- (5.5,3.5) -- (5,3.5);

\draw[fill=white,color=white,rounded corners=2] (6.5,2.5) rectangle (7.5,3.5);
\draw[boundaryfillr,rounded corners=2] (7,2.5) -- (6.5,2.5) --
(6.5,3.5) -- (7.5,3.5) -- (7.5,2.5) -- (6.5,2.5);

\draw[fill=white,color=white,rounded corners=2](8,6.5) -- (7.5,6.5) --
(7.5,7.5) -- (9.5,7.5) -- (9.5,6.5) -- (8,6.5);
\draw[boundaryfillg,rounded corners=2] (8,6.5) -- (7.5,6.5) --
(7.5,7.5) -- (9.5,7.5) -- (9.5,6.5) -- (8,6.5);

\draw[fill=white,color=white,rounded corners=2](10,6.5) -- (9.5,6.5) --
(9.5,7.5) -- (12.5,7.5) -- (12.5,5.5) -- (10.5,5.5) -- (10.5,6.5) --
(10,6.5);
\draw[boundaryfillb,rounded corners=2] (10,6.5) -- (9.5,6.5) --
(9.5,7.5) -- (12.5,7.5) -- (12.5,5.5) -- (10.5,5.5) -- (10.5,6.5) --
(10,6.5);

\draw[fill=white,color=white,rounded corners=2](13.5,.5) -- (12.5,.5) --
(12.5,1.5) -- (13.5,1.5) -- (13.5,2.5)-- (14.5,2.5) --
(14.5,.5) -- (13.5,.5);
\draw[boundaryfillg,rounded corners=2] (13.5,.5) -- (12.5,.5) --
(12.5,1.5) -- (13.5,1.5) -- (13.5,2.5)-- (14.5,2.5) --
(14.5,.5) -- (13.5,.5);

\draw[fill=white,color=white,rounded corners=2](12,.5) -- (11.5,.5) --
(11.5,1.5) -- (12.5,1.5) -- (12.5,2.5) -- (13.5,2.5) -- (13.5,3.5) --
(14.5,3.5) -- (14.5,2.5) -- (13.5,2.5) -- (13.5,1.5) -- (12.5,1.5) --
(12.5,.5) -- (12,.5);
\draw[boundaryfillb,rounded corners=2] (12,.5) -- (11.5,.5) --
(11.5,1.5) -- (12.5,1.5) -- (12.5,2.5) -- (13.5,2.5) -- (13.5,3.5) --
(14.5,3.5) -- (14.5,2.5) -- (13.5,2.5) -- (13.5,1.5) -- (12.5,1.5) --
(12.5,.5) -- (12,.5);

    \node[rv] at (0,0) {};
    \node[rv] at (0,1) {};
    \node[rv] at (0,2) {};
    \node[rv] at (0,3) {};
    \node[rv] at (0,4) {};
    \node[rv] at (0,5) {};
    \node[rv] at (0,6) {};
    \node[rv] at (0,7) {};
    \node[rv] at (0,8) {};

    \node[rv] at (1,0) {};
    \node[bv] at (1,1) {};
    \node[bv] at (1,2) {};
    \node[rv] at (1,3) {};
    \node[rv] at (1,4) {};
    \node[rv] at (1,5) {};
    \node[rv] at (1,6) {};
    \node[rv] at (1,7) {};
    \node[rv] at (1,8) {};

    \node[rv] at (2,0) {};
    \node[bv] at (2,1) {};
    \node[rv] at (2,2) {};
    \node[rv] at (2,3) {};
    \node[bv] at (2,4) {};
    \node[bv] at (2,5) {};
    \node[rv] at (2,6) {};
    \node[rv] at (2,7) {};
    \node[rv] at (2,8) {};

    \node[rv] at (3,0) {};
    \node[rv] at (3,1) {};
    \node[rv] at (3,2) {};
    \node[bv] at (3,3) {};
    \node[bv] at (3,4) {};
    \node[bv] at (3,5) {};
    \node[bv] at (3,6) {};
    \node[rv] at (3,7) {};
    \node[rv] at (3,8) {};

    \node[rv] at (4,0) {};
    \node[rv] at (4,1) {};
    \node[rv] at (4,2) {};
    \node[bv] at (4,3) {};
    \node[bv] at (4,4) {};
    \node[bv] at (4,5) {};
    \node[rv] at (4,6) {};
    \node[rv] at (4,7) {};
    \node[rv] at (4,8) {};

    \node[rv] at (5,0) {};
    \node[rv] at (5,1) {};
    \node[bv] at (5,2) {};
    \node[bv] at (5,3) {};
    \node[rv] at (5,4) {};
    \node[bv] at (5,5) {};
    \node[bv] at (5,6) {};
    \node[rv] at (5,7) {};
    \node[rv] at (5,8) {};

    \node[rv] at (6,0) {};
    \node[rv] at (6,1) {};
    \node[bv] at (6,2) {};
    \node[bv] at (6,3) {};
    \node[bv] at (6,4) {};
    \node[bv] at (6,5) {};
    \node[bv] at (6,6) {};
    \node[rv] at (6,7) {};
    \node[rv] at (6,8) {};

    \node[rv] at (7,0) {};
    \node[rv] at (7,1) {};
    \node[bv] at (7,2) {};
    \node[rv] at (7,3) {};
    \node[bv] at (7,4) {};
    \node[bv] at (7,5) {};
    \node[rv] at (7,6) {};
    \node[rv] at (7,7) {};
    \node[rv] at (7,8) {};

    \node[rv] at (8,0) {};
    \node[rv] at (8,1) {};
    \node[bv] at (8,2) {};
    \node[bv] at (8,3) {};
    \node[bv] at (8,4) {};
    \node[rv] at (8,5) {};
    \node[rv] at (8,6) {};
    \node[gv] at (8,7) {};
    \node[rv] at (8,8) {};

    \node[rv] at (9,0) {};
    \node[rv] at (9,1) {};
    \node[bv] at (9,2) {};
    \node[bv] at (9,3) {};
    \node[bv] at (9,4) {};
    \node[rv] at (9,5) {};
    \node[rv] at (9,6) {};
    \node[gv] at (9,7) {};
    \node[rv] at (9,8) {};

    \node[rv] at (10,0) {};
    \node[rv] at (10,1) {};
    \node[rv] at (10,2) {};
    \node[rv] at (10,3) {};
    \node[rv] at (10,4) {};
    \node[rv] at (10,5) {};
    \node[rv] at (10,6) {};
    \node[bv] at (10,7) {};
    \node[rv] at (10,8) {};

    \node[rv] at (11,0) {};
    \node[rv] at (11,1) {};
    \node[rv] at (11,2) {};
    \node[rv] at (11,3) {};
    \node[rv] at (11,4) {};
    \node[rv] at (11,5) {};
    \node[bv] at (11,6) {};
    \node[bv] at (11,7) {};
    \node[rv] at (11,8) {};

    \node[rv] at (12,0) {};
    \node[bv] at (12,1) {};
    \node[rv] at (12,2) {};
    \node[rv] at (12,3) {};
    \node[rv] at (12,4) {};
    \node[rv] at (12,5) {};
    \node[bv] at (12,6) {};
    \node[bv] at (12,7) {};
    \node[rv] at (12,8) {};

    \node[rv] at (13,0) {};
    \node[gv] at (13,1) {};
    \node[bv] at (13,2) {};
    \node[rv] at (13,3) {};
    \node[rv] at (13,4) {};
    \node[rv] at (13,5) {};
    \node[rv] at (13,6) {};
    \node[rv] at (13,7) {};
    \node[rv] at (13,8) {};

    \node[rv] at (14,0) {};
    \node[gv] at (14,1) {};
    \node[gv] at (14,2) {};
    \node[bv] at (14,3) {};
    \node[rv] at (14,4) {};
    \node[rv] at (14,5) {};
    \node[rv] at (14,6) {};
    \node[rv] at (14,7) {};
    \node[rv] at (14,8) {};

    \node[rv] at (15,0) {};
    \node[rv] at (15,1) {};
    \node[rv] at (15,2) {};
    \node[rv] at (15,3) {};
    \node[rv] at (15,4) {};
    \node[rv] at (15,5) {};
    \node[rv] at (15,6) {};
    \node[rv] at (15,7) {};
    \node[rv] at (15,8) {};
\end{tikzpicture}
\caption{A $q=3$ Potts model configuration depicting nested
     regions of constant color.}
\label{fig:Potts}
\end{figure}

Theorems~\ref{PottsTorusCrit} and~\ref{PottsZd} concern not just low
temperatures, but all temperatures. Pirogov--Sinai theory has been
developed for the Potts model at all temperatures when $q\gg 1$, and
for doing this it is very helpful to use the random cluster
representation~\cite{laanait1991interfaces}. Our algorithms rely on
this, and we follow the sophisticated approach
from~\cite{borgs2012tight}. Algorithmically, however, the reliance on
the random cluster model creates a key difficulty.  As
  discussed above, in the Potts model representation, the $q$ ground
states (one for each color) are completely symmetric.  These
  ground states were defined with $\beta\gg 1$ in mind, while applying
  Pirogov--Sinai theory at all temperatures requires having a ground
  state corresponding to the typical behavior when $\beta\ll 1$ as
  well. The random cluster representation achieves this naturally: it has
  two ground states,  the ordered (full set of edges) and disordered (empty set
of edges) ground states. These ground states are not 
symmetric; the former captures  the low-temperature
  behavior and the latter the high-temperature behavior. See
  Figure~\ref{fig:RC}.
\begin{figure}[h]
  \centering
    \begin{tikzpicture}[scale=.7]
    \draw[boundary] (0,0) grid (7,5);
    \draw[lam] (-.5,-.5) rectangle (7.5,5.5);
    
    \draw[lamb] (0,0) -- (0,1);
    \draw[lamb] (0,5) -- (1,5);
    \draw[lamb] (2,5) -- (2,4);
    \draw[lamb] (3,1) -- (4,1);
    \draw[lamb] (3,2) -- (3,3);
    \draw[lamb] (4,5) -- (5,5);
    \draw[lamb] (5,4) -- (6,4);
    \draw[lamb] (7,4) -- (7,3);

    \node[ov] at (0,0) {};
    \node[ov] at (0,1) {};
    \node[ov] at (0,2) {};
    \node[ov] at (0,3) {};
    \node[ov] at (0,4) {};
    \node[ov] at (0,5) {};

    \node[ov] at (1,0) {};
    \node[ov] at (1,1) {};
    \node[ov] at (1,2) {};
    \node[ov] at (1,3) {};
    \node[ov] at (1,4) {};
    \node[ov] at (1,5) {};

    \node[ov] at (2,0) {};
    \node[ov] at (2,1) {};
    \node[ov] at (2,2) {};
    \node[ov] at (2,3) {};
    \node[ov] at (2,4) {};
    \node[ov] at (2,5) {};

    \node[ov] at (3,0) {};
    \node[ov] at (3,1) {};
    \node[ov] at (3,2) {};
    \node[ov] at (3,3) {};
    \node[ov] at (3,4) {};
    \node[ov] at (3,5) {};

    \node[ov] at (4,0) {};
    \node[ov] at (4,1) {};
    \node[ov] at (4,2) {};
    \node[ov] at (4,3) {};
    \node[ov] at (4,4) {};
    \node[ov] at (4,5) {};

    \node[ov] at (5,0) {};
    \node[ov] at (5,1) {};
    \node[ov] at (5,2) {};
    \node[ov] at (5,3) {};
    \node[ov] at (5,4) {};
    \node[ov] at (5,5) {};

    \node[ov] at (6,0) {};
    \node[ov] at (6,1) {};
    \node[ov] at (6,2) {};
    \node[ov] at (6,3) {};
    \node[ov] at (6,4) {};
    \node[ov] at (6,5) {};

    \node[ov] at (7,0) {};
    \node[ov] at (7,1) {};
    \node[ov] at (7,2) {};
    \node[ov] at (7,3) {};
    \node[ov] at (7,4) {};
    \node[ov] at (7,5) {};
\end{tikzpicture}
\qquad
    \begin{tikzpicture}[scale=.7]
    \draw[boundary] (0,0) grid (7,5);
    \draw[lam] (-.5,-.5) rectangle (7.5,5.5);
    
    \draw[lamb] (4,0) grid (7,3);
    \draw[lamb] (0,1) grid (2,4);
    \draw[lamb] (0,0) -- (3,0) -- (3,2) -- (2,2);
    \draw[lamb] (1,1) -- (2,1);
    \draw[lamb] (1,0) -- (1,5);
    \draw[lamb] (0,4) -- (0,5);
    \draw[lamb] (2,0) -- (2,1);
    \draw[lamb] (1,4) -- (1,5) -- (3,5) -- (3,4) -- (2,4);
    \draw[lamb] (2,3) -- (4,3);
    \draw[lamb] (3,2) -- (4,2);
    \draw[lamb] (3,3) -- (3,4) -- (5,4) -- (5,5) -- (7,5) -- (7,4) --
    (6,4) -- (6,3);
    \draw[lamb] (4,3) -- (4,5) -- (3,5);
    \draw[lamb] (5,3) -- (5,4);

    \node[ov] at (0,0) {};
    \node[ov] at (0,1) {};
    \node[ov] at (0,2) {};
    \node[ov] at (0,3) {};
    \node[ov] at (0,4) {};
    \node[ov] at (0,5) {};

    \node[ov] at (1,0) {};
    \node[ov] at (1,1) {};
    \node[ov] at (1,2) {};
    \node[ov] at (1,3) {};
    \node[ov] at (1,4) {};
    \node[ov] at (1,5) {};

    \node[ov] at (2,0) {};
    \node[ov] at (2,1) {};
    \node[ov] at (2,2) {};
    \node[ov] at (2,3) {};
    \node[ov] at (2,4) {};
    \node[ov] at (2,5) {};

    \node[ov] at (3,0) {};
    \node[ov] at (3,1) {};
    \node[ov] at (3,2) {};
    \node[ov] at (3,3) {};
    \node[ov] at (3,4) {};
    \node[ov] at (3,5) {};

    \node[ov] at (4,0) {};
    \node[ov] at (4,1) {};
    \node[ov] at (4,2) {};
    \node[ov] at (4,3) {};
    \node[ov] at (4,4) {};
    \node[ov] at (4,5) {};

    \node[ov] at (5,0) {};
    \node[ov] at (5,1) {};
    \node[ov] at (5,2) {};
    \node[ov] at (5,3) {};
    \node[ov] at (5,4) {};
    \node[ov] at (5,5) {};

    \node[ov] at (6,0) {};
    \node[ov] at (6,1) {};
    \node[ov] at (6,2) {};
    \node[ov] at (6,3) {};
    \node[ov] at (6,4) {};
    \node[ov] at (6,5) {};

    \node[ov] at (7,0) {};
    \node[ov] at (7,1) {};
    \node[ov] at (7,2) {};
    \node[ov] at (7,3) {};
    \node[ov] at (7,4) {};
    \node[ov] at (7,5) {};
\end{tikzpicture}
  \caption{Two random cluster model configurations. The configuration
    on the left is a perturbation of the empty set of edges, and the
    configuration on the right a perturbation of the full set of
    edges.}
  \label{fig:RC}
\end{figure}

In Pirogov-Sinai theory ground states are categorized based on free
energies of truncated models, as is discussed
in~\cite[Section~1.5]{Kotecky}.  For a given choice of parameters,
ground states minimizing the truncated free energy are \emph{stable}
while other ground states are \emph{unstable}.  In the Potts
representation all ground states are stable by symmetry, and this was
exploited in the low temperature algorithms
in~\cite{helmuth2018contours}.  In the random cluster representation,
one of the ground states may well be unstable (in fact only at
$\beta =\beta_c$ are both ground states stable).  Thus while working
with the random cluster representation gives us a convergent cluster
expansion at all temperatures, it also necessitates an algorithmic
approach that accommodates unstable ground states. The next paragraph
  discusses our algorithmic approach. We believe this approach could
  be adapted to other models with unstable ground states, but for the
  sake of  concreteness we restrict our discussion to the setting of the random cluster model.

To see the issue that unstable ground states create
for algorithms, recall that when $\beta\gg 1$ the intuition is that
most configurations look like the ordered ground state, with local
deviations that look like the disordered ground state. Since the
disordered ground state is not stable at low temperatures, it
does not suppress local deviations that flip back to the ordered
ground state. This prevents us from analyzing the recursive structure
of the Pirogov--Sinai expansion by using polymer model methods: the
polymer model expansion in an unstable ground state does not have a
convergent expansion. To circumvent this, we use tools
from~\cite{borgs2012tight} to establish that inside of any unstable
deviation there will be a further deviation back to the stable ground
state. This flip back to the stable ground state happens rapidly enough that
we can use brute-force methods. Since there may be many unstable
deviations, it is also important for us to control their total volume,
and again we use tools from~\cite{borgs2012tight} to do this.

As is clear from this discussion, this papers makes significant use of
the methods developed in~\cite{borgs2012tight}
and~\cite{helmuth2018contours}. For the ease of the reader who wishes
to see the proofs of results we use from~\cite{borgs2012tight} we have
largely stuck to the definitions presented in that paper, and have
made careful note of the situations in which we have chosen
alternative definitions that facilitate our algorithms. To complement
the  discussion above, we conclude this section with an outline of
our arguments along with pointers to the technical content of the
paper.

\begin{enumerate}
\item In Section~\ref{secPolymer} we briefly recall the notion of a
  polymer model and convergence criteria for the cluster expansion,
  and recall from~\cite{helmuth2018contours} how this can be used for
  approximation algorithms. A key improvement
  upon~\cite{helmuth2018contours} is that we work directly with the
  cluster expansion rather than using Barvinok's
  method~\cite{barvinok2017weighted}. This is  essential, as Barvinok's
  method relies on the existence of a zero-free region. In the Potts model there
  cannot be a zero-free region uniformly in the volume near $\beta_{c}$, precisely because
  this is the point at which a phase transition occurs.   In this section we also apply the polymer model algorithm to the
  random cluster model at very high temperatures, meaning 
  $\beta\leq  \beta_h \bydef \frac{3\log q}{4d}$.
\item In Section~\ref{secFKcontour} we first recall the tools from
  Pirogov--Sinai theory developed in~\cite{borgs2012tight} for the
  random cluster model. We then use these tools to establish the
  necessary ingredients for an algorithmic implementation of the
  method.
\item Section~\ref{secEstimates} contains estimates for the contour
  model representation derived in Section~\ref{secFKcontour}.  We
   prove some consequences of estimates from~\cite{borgs2012tight} that are needed for our algorithms. As
  discussed above, the key additional estimates concern how unstable
  contours rapidly `flip' to stable contours, which 
  are essential for our algorithms to be efficient.

  This section focuses on the most interesting case of
  $\beta\geq\beta_{c}$. The case $\beta_h<\beta<\beta_{c}$, which is
  very similar to $\beta>\beta_{c}$ and again uses estimates
  from~\cite{borgs2012tight}, is discussed in Appendix~\ref{sec:HT}.
\item In Section~\ref{sec:count} we present our approximate counting
  algorithms. The broad idea is to use the inductive Pirogov--Sinai
  method of~\cite{helmuth2018contours}, but with significant refinements
   to deal with the presence of an unstable ground
  state. Similar refinements are then used in Section~\ref{sec:sample}
  to develop sampling algorithms.
\end{enumerate}

We remark that it may be possible to combine results and proof techniques
from~\cite{DuminilCopinRaoufiTassion,martinelli19942,alexander2004mixing}
to prove that the Glauber dynamics mix rapidly on the torus and
sufficiently regular subsets of $\Z^d$ for all $\beta<\beta_{c}$,
which would yield a much faster sampling algorithm than the one we
have given here. We are not aware, however, of any existing statement
in the literature which would directly imply rapid mixing in the whole
range $\beta<\beta_{c}$, and leave this as an open problem. Further
open problems can be found in the conclusion of this paper, Section~\ref{sec:Conc}.

\section{Polymer models, cluster expansions, and algorithms}
\label{secPolymer}

This section describes how two related tools from statistical physics,
abstract polymer models and the cluster expansion, can be used to
design efficient algorithms to approximate partition functions.

An \emph{abstract polymer
  model}~\cite{gruber1971general,kotecky1986cluster} consists of a set
$\cC$ of \emph{polymers}, with each polymer $\gamma \in \cC$ equipped with a complex-valued
\emph{weight} $w_{\gamma}$ and a non-negative \emph{size}
$\| \gamma \|$. The set $\cC$ also comes equipped with a symmetric
compatibility relation $\sim$ such that each polymer is
incompatible with itself, denoted $\gamma \nsim \gamma$.  Let $\cG$
denote the collection of all sets of pairwise compatible polymers from
$\cC$, including the empty set of polymers. The polymer
model partition function  is defined to be
\begin{equation}
\label{eq:PolyZ}
Z( \cC,w) \bydef \sum_{\Gamma \in \cG} \prod_{\gamma \in \Gamma}
w_{\gamma}. 
\end{equation}
In~\eqref{eq:PolyZ} $w$ is shorthand for the collection of polymer
weights.

Let $\Gamma$ be a non-empty tuple of polymers. The
\emph{incompatibility graph $H_{\Gamma}$} of $\Gamma$ has vertex set
$\Gamma$ and edges linking any two incompatible polymers, i.e.,
$\{\gamma,\gamma'\}$ is an edge if and only if $\gamma\nsim\gamma'$. A
non-empty ordered tuple $\Gamma$ of polymers is a \emph{cluster} if
its incompatibility graph $H_\Gamma$ is connected.  Let $\cG^c$ be the
set of all clusters of polymers from $\cC$. The cluster expansion is
the following formal power series for $\log Z(\cC, w)$ in the
variables $w_\gamma$:
\begin{equation}
  \label{eq:clusterexp}
  \log Z(\cC,w) 
  = 
  \sum_{\Gamma \in \cG^c} \phi(H_\Gamma) \prod_{\gamma
    \in \Gamma} w_\gamma . 
\end{equation}
In~\eqref{eq:clusterexp} $\phi(H)$ denotes the \emph{Ursell function}
of the graph $H=(V(H),E(H))$, i.e., 
\begin{equation*}
  \phi(H) 
  \bydef 
  \frac{1}{|V(H)|!} \sum_{\substack{ A\subseteq E(H)
      \\(V(H), A) \text{ connected} }} (-1)^{|A|}. 
\end{equation*}

For a proof of~\eqref{eq:clusterexp}
see, e.g., \cite{kotecky1986cluster,friedli2017statistical}. Define
$\| \Gamma \| \bydef \sum_{\gamma \in \Gamma} \| \gamma \|$, and
define the truncated cluster expansion by
\begin{equation*}
  T_{m}(\cC,w) 
  \bydef 
  \sum_{\substack{\Gamma \in \cG^c \\ \| \Gamma \|< m}} \phi(H_\Gamma)
  \prod_{\gamma \in \Gamma} w_\gamma  \,.
\end{equation*}

Henceforth we will restrict our attention to a special class of
polymer models defined in terms of a graph $G$ with maximum degree
$\Delta$ on $N$ vertices. Namely, we will assume that each polymer is
a connected subgraph $\gamma=(V(\gamma),E(\gamma))$ of $G$.  The
compatibility relation is defined by disjointness in $G$:
$\gamma \sim \gamma'$ iff $V(\gamma) \cap V(\gamma') = \emptyset$. We
write $\abs{\gamma}$ for $\abs{V(\gamma)}$, the number of vertices in
the polymer $\gamma$.

A useful criteria for convergence of the formal power series
in~\eqref{eq:clusterexp} 
is given by the following adaptation of a theorem of Koteck\'{y} and
Preiss~\cite{kotecky1986cluster}.

\begin{lemma}
  \label{KPthm}
  {Let
    $G$ be a graph} of
  maximum degree $\Delta{\geq 2}$ on $N$
  vertices. {Suppose that polymers are connected subgraphs of
    $G$ that contain at least two vertices.} 
  Suppose further that for some
  $b>0$ and all $\gamma \in \cC$,
  \begin{align}
    \label{eqPeierls}
    \| \gamma \| &\ge b | E(\gamma)|, \\
    \label{eqPolymerKP}
    |w_{\gamma}| &\le {\exp\left(-   \left(\frac{4 + \log \Delta}{b}  + 3 \right) \| \gamma \|\right)}.
  \end{align}
  Then the cluster expansion~\eqref{eq:clusterexp} converges
  absolutely, and for $m\in\N$,
  \begin{equation}
    \label{eqTruncBound}
    \left|  T_m (\cC,w)  - \log Z(\cC,w) \right|  \le N e^{- 3m  } \,.
  \end{equation}
  
Moreover, if instead all polymers are connected, induced subgraphs of $G$, and 
  for some
  $b>0$ and all $\gamma \in \cC$,
  \begin{align}
    \label{eqPeierls2}
    \| \gamma \| &\ge b |\gamma|, \\
    \label{eqPolymerKP2}
    |w_{\gamma}| &\le {\exp\left(-   \left(\frac{3 + \log \Delta}{b}  + 3 \right) \| \gamma \|\right)},
  \end{align}
  then the same conclusion holds. 
\end{lemma}

This lemma implies that if conditions~\eqref{eqPeierls}
and~\eqref{eqPolymerKP} hold, then $\exp( T_{m}(\cC,w))$ is an
$\eps$-relative approximation to $Z(\cC,w)$ for
$m \ge \log (N/\eps)/3$.

\begin{proof}
  We append to $\cC$ a polymer $\gamma_v$ for each $v \in V(G)$
  consisting only of that vertex, with
  size $\| \gamma_v \|=1$ and $w_{\gamma_v}=0$. By definition, 
  $\gamma_{v}$ is incompatible with every other polymer that contains
  $v$.  Then
  \begin{align*}
  \sum_{\gamma \nsim \gamma_v} |w_{\gamma}| e^{ |E(\gamma)| +  3\| \gamma
    \|} 
    &\le  \sum_{\gamma \nsim \gamma_v} e^{ |E(\gamma)|} e^{ - (\frac{4 +
      \log \Delta}{b})\| \gamma \|}     \\
       &\le  \sum_{\gamma \nsim \gamma_v} e^{ |E(\gamma)|} e^{ -(4 + \log
      \Delta) | E(\gamma )|}     \\
      & \le \sum_{k \ge 1} (e \Delta)^{k}  
       e^{-(3 +  \log \Delta) k} 
  \end{align*}
  where the first inequality is by~\eqref{eqPolymerKP}, the second
  by~\eqref{eqPeierls}, and the third is by bounding the number of  connected subgraphs of $G$ with $k$ edges that contain $v$ by $(e\Delta)^{k}$~\cite{BorgsChayesKahnLovasz}. This yields
  \begin{equation}
    \label{eq:bd}
    \sum_{\gamma \nsim \gamma_v} |w_{\gamma}| e^{ |E(\gamma)| +  3\| \gamma
    \|} \leq  \sum_{k \ge 1} e^{ -  2k } < 1/2.
  \end{equation}
  
  Moreover, under the second assumption, that all polymers are connected induced subgraphs, we have a similar bound, with $| \gamma|$ in place of $|E(\gamma)|$:
  \begin{equation}
    \label{eq:bd2}
    \sum_{\gamma \nsim \gamma_v} |w_{\gamma}| e^{ |\gamma| +  3\| \gamma
    \|} {\leq \sum_{k\geq 1}e^{-k}}< 1,
\end{equation}
{where we have used that the number of connected induced
  subgraphs on $k$ vertices that contain $v$ is at most $(e\Delta)^{k}$~\cite{BorgsChayesKahnLovasz}.}

  Now fix a polymer $\gamma$. By summing~\eqref{eq:bd} over all
  $v \in \gamma$ we obtain
  \begin{equation}
    \sum_{\gamma' \nsim \gamma} |w_{\gamma'}| e^{|E(\gamma')| +3 \| \gamma ' \|} < |\gamma|/2  \le |E(\gamma)| \,.
  \end{equation}
  By applying the main theorem of~\cite{kotecky1986cluster} with
  $a(\gamma) = |E(\gamma)|$, $d(\gamma) = 3 \| \gamma \|$ we obtain 
  that the cluster expansion converges absolutely. Moreover, we also
  obtain that
  \begin{equation}
    \sum_{\substack{\Gamma \in \cG^c \\ \Gamma \ni v}}  \left|
      \phi(H_{\Gamma}) \prod_{\gamma \in \Gamma} w_{\gamma}  \right|
    e^{3 \| \Gamma \| } \le 1  \,, 
  \end{equation}
  where the sum is over all clusters that contain a polymer containing
  the vertex $v$. By using this estimate, {restricting to
    $\|\Gamma\|\ge m$}, and summing 
  over all $v \in V(G)$ one obtains
  \begin{equation}
    \label{eq:tailbound}
     \sum_{\substack{\Gamma \in \cG^c \\ \norm{\Gamma}\geq m}}  \left|
      \phi(H_{\Gamma}) \prod_{\gamma \in \Gamma} w_{\gamma}  \right|
    \leq N e^{-3m}
  \end{equation}
 which is~\eqref{eqTruncBound}.
 
 The same argument works under the second assumption by taking $a(\gamma) = |\gamma|$.  
\end{proof}

Because clusters are connected objects arising from a bounded-degree
graph, the truncated cluster expansion can be computed
efficiently. Recall that $N=\abs{V(G)}$.
\begin{lemma}[]
  \label{lemPolyModelCount}
  Suppose the conditions of Lemma~\ref{KPthm} hold.  Then given a list
  of all polymers $\gamma$ of size at most $m$ along with the weights
  $w_{\gamma}$ of these polymers, the truncated cluster expansion
  $T_m(\cC,w)$ can be computed in time $O (N\exp( O(m)))$.
\end{lemma}
\begin{proof}
  This is~\cite[Theorem 6]{helmuth2018contours}.
\end{proof}

The next lemma says that, for the purposes of approximating a polymer
partition function, it is sufficient to have approximate
evaluations $\tilde w_{\gamma}$ of the weights $w_{\gamma}$.

\begin{lemma}
  \label{lemPolymerApprox}
  Let $v\colon \cC\to [0,\infty)$ be a non-negative function on
  polymers such that 
  $v(\gamma) \le \| \gamma \|^2$. 
  Suppose $0<\eps < N^{-1}$, and let $m= \log
  (8/\eps)/3$. Suppose the conditions of Lemma~\ref{KPthm} hold and that for all
  $\gamma \in \cC$
  with $\| \gamma \| \le m$, $\tilde w_\gamma$ is an
  $\eps v(\gamma)$-relative approximation to $w_\gamma$. Then
  $\exp( T_{m}(\cC, \tilde w))$ is an $N\eps /4$-relative approximation
  to $Z(\cC,w)$.
\end{lemma} 

\begin{proof}
  Using the definition of $m$ and applying Lemma~\ref{KPthm}, we have
  \begin{equation*}
    |\log Z_G(\cC,w) - T_{m}(\cC,w) | \le N\eps/8, 
  \end{equation*}
  so by the triangle inequality it is enough to show that
  \begin{equation}
    \label{eq:tri2}
    \left | T_{m}(\cC,\tilde w) -  T_{m}( \cC,w) \right | \le N\eps/8.
  \end{equation}
  Define $r_{\gamma}$ by
  $\log \tilde w_\gamma = \log w_\gamma + r_\gamma$. To
  prove~\eqref{eq:tri2}, note the identity
  \begin{equation*}
    T_{m}(\cC,\tilde w) -  T_{m}( \cC,w)
    = \sum_{\substack{\Gamma \in \cG^c(G) \\ \norm{\Gamma}< m}}
      \phi(H_\Gamma) \prod_{\gamma \in \Gamma} w_\gamma \cdot \left [
        \exp\left( \sum_{\gamma \in \Gamma}  r_\gamma \right)  - 1
      \right ].
  \end{equation*}
  Our hypotheses imply $|r_{\gamma}| \le \eps v(\gamma)$, and hence by
  the triangle inequality we obtain
  \begin{equation*}
    \abs{T_{m}(\cC,\tilde w) -  T_{m}( \cC,w)} \leq 
    \sum_{\substack{\Gamma \in \cG^c(G) \\ \norm{\Gamma}< m}}
    \left(\exp \left (\sum_{\gamma\in\Gamma} \epsilon v(\gamma) \right)-1 \right)
    \abs{ \phi(H_{\Gamma}) \prod_{\gamma\in\Gamma}w_{\gamma}},
  \end{equation*}
  where we have used the elementary inequality
  $\abs{e^{a}-1}\leq e^{b}-1$ when $\abs{a}\leq b$ to bound the term
  in square brackets. Since $v(\gamma)\leq \norm{\gamma}^{2}$ this
  yields, after ordering the sum over clusters according to their size $k$,
  \begin{equation*}
    \abs{T_{m}(\cC,\tilde w) -  T_{m}( \cC,w)}
    \leq
    \sum_{k=1}^{m-1}(\exp(\eps k^{2})-1)
    \sum_{\substack{\Gamma \in \cG^c(G) \\ \norm{\Gamma}=k}}
    \abs{ \phi(H_{\Gamma}) \prod_{\gamma\in\Gamma}w_{\gamma}}
    \leq \sum_{k=1}^{m-1}(\exp(\eps k^{2})-1) Ne^{-3k}.
  \end{equation*}
  The last inequality follows from the convergence of the cluster
  expansion (see \eqref{eq:tailbound} in the proof of
  Lemma~\ref{KPthm}). Since $\eps<N^{-1}$ we can bound
  $e^{\eps k^{2}}-1$ by $2\eps k^{2}$, and~\eqref{eq:tri2} follows
  since $\sum_{k\geq 1}k^{2}e^{-3k}<1/16$.
\end{proof}

Putting Lemmas~\ref{KPthm},~\ref{lemPolyModelCount},
and~\ref{lemPolymerApprox} together we see that the partition function
$Z(\cC, w)$ can be approximated efficiently if
\begin{enumerate}
\item conditions~\eqref{eqPeierls} and~\eqref{eqPolymerKP} hold
\item polymers of size at most $m{=O(\log N/\eps)}$ can be enumerated efficiently,
  i.e., in time polynomial in $N$ and exponential in $m$, and
\item the polymer weights $w_{\gamma}$ can be approximated efficiently, i.e.,
  in time polynomial in the size of $\gamma$.
\end{enumerate}

\subsection{High temperature expansion}
\label{secHighTemp}
This section explains how the polymer model algorithm of the previous
section yields efficient counting and sampling algorithms for the
random cluster model when $q$ is sufficiently large and
$\beta \le \beta_{h}=\frac{3 \log q}{4d}$. This use of the polymer model
algorithm also serves as a warm-up for the more sophisticated
contour-based algorithms we will use in later sections when
$\beta>\beta_{h}$. 

In fact, the simpler setting of $\beta\le \beta_{h}$ allows for
greater generality: 
we will derive an algorithm that applies to the random cluster model
on \emph{any} graph $G$ of maximum degree at most $2d$.
\begin{theorem}
\label{thmHighTempExpansion}
Suppose $d \ge 2$ and $q=q(d)$ is sufficiently large.  Then for
$\beta \le \beta_h$ there is an FPTAS and efficient sampling scheme
for the Potts model and the random cluster model with
$p=1- e^{-\beta}$ on all graphs of maximum degree at most $2d$.
\end{theorem}
\begin{proof}
  Let $G=(V(G),E(G))$ be such a graph. We define polymers to be
  connected subgraphs of $G$ with at least two vertices. As per our
  convention, polymers are compatible if they are vertex
  disjoint, and $|\gamma| = |V(\gamma)|$. We set $\norm{\gamma} =
  |E(\gamma)|$, and define the weight of a polymer $\gamma$ to be
  \begin{equation}
    \label{eq:HTpolyweight}
    w_\gamma 
    \bydef \left( \frac{p}{1-p}  \right)^{\norm{\gamma}} q^{1-|\gamma|} =
    (e^\beta -1)^{\norm{\gamma}} q^{1-|\gamma|} \,.
  \end{equation}

  Let $\cC(G)$ be the set of all polymers on $G$, $\cG(G)$ be the
  collection of all sets of pairwise compatible polymers from
  $\cC(G)$, and let
  \begin{equation}
    \Xi(G) \bydef \sum_{\Gamma \in \cG(G)} \prod_{\gamma \in \Gamma} w_{\gamma}
  \end{equation}
 be the corresponding polymer model partition function.  Then we have the identity
 \begin{equation}
   \label{eq:HT-Poly}
   Z^{\text{RC}}_G(p,q) = (1-p)^{|E(G)|} q^{|V(G)|}\;  \Xi(G) .
 \end{equation}
 The relation~\eqref{eq:HT-Poly} follows by extracting a common
 prefactor of $(1-p)^{|E(G)|} q^{|V(G)|}$ from the random cluster
 partition function. {To see this relation it may help to
   temporarily allow polymers that consist of a single vertex; since
   these receive weight one by~\eqref{eq:HTpolyweight} it is equivalent to remove
   these from the set of polymers.}

Condition~\eqref{eqPeierls} holds with $b=1$ since $\| \gamma \| = |E(\gamma)|$. 
 We will show that condition~\eqref{eqPolymerKP} holds if $\beta\le \beta_{h}$ and
 $q$ is sufficiently large as a function of $d$. 
Suppose there is a $q_{0}$ such that for
 all $\gamma$, all $\beta\leq \beta_{h}$, and all $q\geq q_{0}$
 \begin{equation}
   \label{eq:HT-KP}
   w_{\gamma}\leq C^{-\norm{\gamma}}.
 \end{equation}
 Then if $C=C(d){>0}$ is {large} 
 enough, \eqref{eqPolymerKP} holds. Since
 $b=1$, $C= \exp({7+\log 2d})$ 
 suffices, and we fix $C$ to be this
 value hereon. We now verify~\eqref{eq:HT-KP} in three steps, by
 considering polymers grouped according to the value of
 $k=\norm{\gamma}$.

 \begin{enumerate}
 \item For $k>5d$ we will use the fact that
   $|\gamma| \ge \norm{\gamma}/d$ since every edge is incident to two
   vertices and every vertex is incident to at most $2d$ edges.  Then using the fact that $\beta \le \beta_h$
   we have
   \begin{equation}
     w_\gamma \le q(e^{\beta} -1)^k q^{-k/d} \le q^{1 - \frac{k}{4d}}  \le q^{ - \frac{k}{20d}}\,,
   \end{equation}
   which is at most $C^{-\norm{\gamma}}$ if $q\geq C^{20d}$.

\item For $d < k \le 5d$, we will use the fact that
   $|\gamma| \ge \frac{1}{2}+ \sqrt{2 \norm{\gamma}}$ since the number of
   edges in a graph on $r$ vertices is at most $\binom {r}{2}$. Then
   we have
   \begin{equation}
     w_\gamma \le q q^{\frac{3k}{4d}} q^{-\frac{1}{2}-\sqrt{2k}} \leq q^{\frac{1}{2} + \frac{3c}{4} - 2 \sqrt{c}}\,,
   \end{equation}
   where $c = k/d$ and where we use the fact that $d \ge 2$ and
   $q\geq 1$. Then since
   $\frac{1}{2} + \frac{3c}{4} - 2 \sqrt{c} \le - \frac{1}{5}$ for
   $c \in [1,5]$, we have
   \begin{equation}
     w_\gamma \le q^{-1/5}\,,
   \end{equation}
   which is at most $C^{-\norm{\gamma}}$ if $q\geq C^{25d}$.

 \item For $1 \le k \le d$, since $|\gamma| \ge 2$, we have
   \begin{equation}
     w_\gamma \le q^{-1} (e^{\beta}-1)^k \le q^{-1} e^{\beta k} \le q^{-1/4} \,,
   \end{equation}
   which is at most $C^{-\norm{\gamma}}$ provided $q\geq C^{4d}$. 
 \end{enumerate}

 Thus taking $q_{0} {\geq} \exp( 25d (7+\log 2d))$
 suffices. Lemmas~\ref{KPthm} and~\ref{lemPolyModelCount} then give an
 FPTAS for computing the random cluster partition function
 $Z^{\text{RC}}_G(1-e^{-\beta},q)$ for all graphs of maximum degree
 $2d$, as enumerating subgraphs of size $m$ in a bounded degree graph
 takes time $\exp(O(m))$, and computing the weight functions only
 requires counting the number of edges and vertices in each subgraph.

 The efficient sampling scheme follows from~\cite[Theorem
 10]{helmuth2018contours}.    Counting and sampling algorithms for the random cluster
 model can be converted into algorithms for the Potts model via
 the Edwards--Sokal coupling described in Appendix~\ref{sec:ES}.
\end{proof}

\begin{proof}[Proof of Theorems~\ref{PottsTorusCrit} and~\ref{PottsZd} for
  $\beta\leq \beta_{h}$]
  Theorem~\ref{PottsTorusCrit} follows immediately from
  Theorem~\ref{thmHighTempExpansion} since $\tor$ is $2d$-regular. 

  By~\eqref{eq:betac}, $\beta_{h}<\beta_{c}$ when $q$ is large enough. Thus Theorem~\ref{PottsZd} requires we provide approximate
  counting and sampling algorithms for free boundary conditions. Since
  induced subgraphs of $\Z^{d}$ have degree bounded by $2d$, the
  result follows by Theorem~\ref{thmHighTempExpansion}.
\end{proof}

\section{Contour model representations}
\label{secFKcontour}

\emph{Contour models} refer to a class of polymer models that arise in
Pirogov--Sinai theory~\cite{pirogov1975phase}.  For a given spin
configuration, contours represent geometric boundaries between regions
dominated by different ground states; the precise definition for the
purposes of this paper will be given below. This section describes an
important contour model representation for the random cluster model on
the torus $\tor$ that is the basic combinatorial object in our
algorithms. This contour representation was originally developed for
obtaining optimal lower bounds on the mixing time for Glauber and
Swensden--Wang dynamics~\cite{borgs2012tight}. In addition to
recalling the construction from~\cite{borgs2012tight} this section
also develops the additional ingredients necessary for algorithmic
applications of the representation.

\subsection{Continuum embedding}
\label{sec:cont}

The contour model representation from~\cite{borgs2012tight} is based
on the natural embedding of the discrete torus $\tor = (\Z/n\Z)^{d}$
of side-length $n\in\N$ into the continuum torus
$\ctor \bydef (\R/n\R)^{d}$. This subsection recalls the basic
definitions, and explains how they can be rephrased in terms of
discrete graph-theoretic notions.\footnote{This continuum construction
  allows for tools from algebraic topology to be used. We have chosen
  to follow the continuum terminology to allow the interested reader
  to easily consult~\cite{borgs2012tight}.}

In what follows we abuse notation slightly and write $\tor$ for the
graph $(\tor, E)$, where $E$ is the edge set of the discrete torus.
We will follow the convention that bold symbols, e.g., $\V{V}$, denote
subsets of $\ctor$, while objects denoted by non-bold symbols like $V$
reside in $\tor$. Thus each vertex $v\in \tor$ is identified with a
point $\V{v}\in \ctor$, and we will identify each edge
$e=\{u,v\}\in E$ with the unit line segment $\V{e}\subset \ctor$
that joins $\V{u}$ to $\V{v}$.  We will also drop $\tor$ from the
notation when possible, e.g., $E$ for $E(\tor)$. 

{Let} 
$\Omega = 2^{E}$ {denote} 
the set of configurations of the
random cluster model on $\tor$. Let $\V{c}\subset \ctor$ denote a
closed $k$-dimensional hypercube with vertices in $\tor$ for some
$k=1,\dots, d$.  We say a hypercube $\V{c}$ is \emph{occupied} with
respect to $A\in \Omega$ if for all edges $e$ with
$\V{e}\subset \V{c}$, $e$ is in $A$. Define
\begin{equation}
  \label{eq:fattening}
  \V{A} \bydef  \left \{\V{x}\in \ctor \mid \text{there exists $\V{c}$ occupied
    s.t.\ $d_{\infty}(\V{x},\V{c})\leq \frac{1}{4}$} \right\},
\end{equation}
where $d_{\infty}$ is the $\ell_{\infty}$-distance, and the distance
from a point to a set is defined in the standard way:
$d_{\infty}(\V{x},\V{c}) =
\inf_{\V{y}\in\V{c}}d_{\infty}(\V{x},\V{y})$. Thus $\V{A}$ is the
closed $1/4$-neighborhood of the occupied hypercubes of $A$. The
connected components of the (topological) boundary $\partial \V{A}$ of
the set $\V{A}$ are the crucial objects in what follows.  Since each
connected component arises from an edge configuration in $\Omega$, it
is clear that the set of possible connected components is a finite
set.  As the connected components of $\partial \V{A}$ are continuum
objects, it may not be immediately apparent how to represent them in a
discrete manner. We briefly describe how to do this now.

Let $\htor$ denote the graph $(\frac{1}{2}\Z/n\Z)^{d}$; as a graph
this is equivalent to the discrete torus $(\Z/(2n)\Z)^{d}$. The
notation $\htor$ is better because we will embed $\htor$ in
$\ctor$ such that (i) $\V{0}$ coincides in $\tor$ and $\htor$, and
(ii) the nearest neighbors of $0$ in $\htor$ are the midpoints of the
edges $e$ containing $0$ in $\tor$.\footnote{More formally, since $\Z^{d}\subset
\frac{1}{2}\Z^{d}\subset \R^{d}$, we obtain a common embedding of
$\htor$ and $\tor$ in $\ctor$.} 

An important observation is that $\V{A}$ can be written as a union of
collections of adjacent closed $d$-dimensional hypercubes of
side-length $1/2$ centered at vertices in $\htor$, where two hypercubes
are called \emph{adjacent} if they share a $(d-1)$-dimensional face.
Adjacency of a set of hypercubes means the set of hypercubes is
connected under the binary relation of being adjacent. By construction
the connected components of $\V{A}$ correspond to the
connected components of the edge configuration $A$.

The boundary $\partial \V{A}$ of $\V{A}$ is just the sum, modulo two,
of the boundaries of the hypercubes whose union gives $A$. 
These boundaries are $(d-1)$-dimensional hypercubes dual to edges in
$\htor$; here dual means that the barycenter of the
$(d-1)$-dimensional hypercube is the same as barycenter of the edge in
$\htor$. The $(d-1)$-dimensional hypercubes that arise from this
duality are the vertices in $\htors$, the graph dual to $\htor$; two
vertices in $\htors$ are connected by an edge if and only if the
corresponding $(d-1)$-dimensional hypercubes intersect in one
$(d-2)$-dimensional hypercube. The preceding discussion implies
$\partial \V{A}$ can be identified with a  subgraph of
$\htors$.  

In the sequel we will discuss components of $\partial \V{A}$ as
continuum objects; by the preceding discussion this could be
reformulated in terms of subgraphs of $\htors$. In
Appendix~\ref{app:subcomp} we show that the computations we perform
involving components of $\partial \V{A}$ can be efficiently computed
using their representations as subgraphs of $\htors$. 

\subsection{Contours and Interfaces}
\label{sec:contours-interfaces}

An important aspect of the analysis in~\cite{borgs2012tight} is that
it distinguishes topologically trivial and non-trivial components of
$\partial \V{A}$. To make this precise, for $i=1,\dots, d$ we define
the \emph{$i$th fundamental loop} $\V{L}_{i}$ to be the set
$\{\V{y}\in \ctor \mid \text{$\V{y}_{j}=1$ for all $j\neq i$}\}$. The
\emph{winding vector} $N(\V{\gamma})\in \{0,1\}^{d}$ of a connected
component $\V{\gamma}\in \partial\V{A}$ is the vector whose $i$th
component is the number of intersections (mod 2) of $\V{\gamma}$ with
$\V{L}_{i}$. 

\begin{defn}
  Let $A\in\Omega$ be an edge configuration.
  \begin{enumerate}
  \item The set of \emph{contours $\Contour(A)$ associated to $A$} is the
    set of connected components of $\partial \V{A}$ with winding vector $0$.
  \item The \emph{interface network $\Interface(A)$ associated to
    $A$} is the set of connected components of $\partial \V{A}$ with
    non-zero winding vector.  Each connected component of an interface network is an \emph{interface}.
  \end{enumerate}
  Without reference to any particular edge configuration, 
  a subset $\V{\gamma}\subset\ctor$ is a \emph{contour} if there is
  an $A\in\Omega$ such that $\V{\gamma}\in \Contour(A)$.
  Interfaces and interface networks are defined analogously.
\end{defn}

Since each fundamental loop intersects each $(d-1)$-dimensional face of
a hypercube centered on $\htor$ exactly zero or one times, we have
the following lemma, which ensures contours can be efficiently
distinguished from interfaces. 
\begin{lemma}
  \label{lem:wind-comp}
  Suppose $\V{\gamma}\in \partial \V{A}$ is comprised of $K$
  $(d-1)$-dimensional faces. Then the winding vector of $\V{\gamma}$
  can be computed in time $O(nK)$.
\end{lemma}
\begin{proof}
  Fix $i\in \{1,2,\dots, d\}$. Each fundamental loop $L_{i}$ has
  length $O(n)$, and hence the set $F_{i}$ of faces that have
  non-trivial intersection with $L_{i}$ has cardinality
  $\abs{F_{i}}= O(n)$. Given the list of faces in $\V{\gamma}$ we can
  compute the $i$th component of the winding vector by (i) iterating
  through the list of faces of $\V{\gamma}$ and adding one each time
  we find a face in $F_{i}$, and (ii) taking the result modulo two.
\end{proof}

The connected components of $\ctor\setminus\partial\V{A}$
are subsets of either $\V{A}$ or $\ctor\setminus\V{A}$. In the former
case we call a component \emph{ordered} and in the latter case
\emph{disordered}.
We write $\V{A}_{\ord}$
(resp.\ $\V{A}_{\dis}$) for the union of the ordered (resp.\ disordered)
components associated to $A$. 
\begin{defn}
  \label{def:labl}
  The \emph{labelling $\ell_{A}$ associated to $A$} is the map from
  the connected components of $\ctor\setminus\partial \V{A}$ to the
  set $\{\dis,\ord\}$ that assigns $\ord$ to components in $\V{A}_{\ord}$ and
  $\dis$ to components in $\V{A}_{\dis}$.
\end{defn}

\begin{defn}
  Two contours $\V{\gamma}_{i}$, $i=1,2$ are \emph{compatible} if
  $d_{\infty}(\V{\gamma}_{1},\V{\gamma}_{2})\geq \frac{1}{2}$. We
  extend this definition analogously to two interfaces, or one
  interface and one contour.
\end{defn}

\begin{defn}
  \label{def:MCI}
  A \emph{matching collection of contours $\Contour$ and interfaces
    $\Interface$} is a triple $(\Contour,\Interface,\ell)$ such that
  $\Interface$ is an interface network and
  \begin{enumerate}
  \item The contours and interfaces in $\Contour\cup\Interface$ are
    pairwise compatible, and
  \item $\ell$ is a map from the set of connected components of
    $\ctor\setminus \cup_{\V{\gamma}\in\Contour\cup\Interface}\V{\gamma}$ to
    the set $\{\dis,\ord\}$ such that for every $\V{\gamma}\in
    \Contour\cup\Interface$, distinct components adjacent to $\V{\gamma}$
    are assigned different labels. 
  \end{enumerate}
\end{defn}

\begin{lemma}
  \label{lem:rep}
  The map from edge configurations $A\in\Omega$ to triples
  $(\Contour,\Interface,\ell)$ of matching contours and interfaces is
  a bijection.
\end{lemma}
\begin{proof}
  See~\cite[p.15]{borgs2012tight}.
\end{proof}

\subsection{Contour and interface formulation of $Z$}
\label{sec:cont-interf-form}

By Lemma~\ref{lem:rep} we can rewrite the partition function in terms
of matching collections of contours and interfaces by re-writing the
weight $w(A)$ of a configuration $A$ in terms of its contours and
interfaces. By weight $w(A)$ we mean the numerator of~\eqref{eqRCdef},
i.e., $w(A) = p^{\abs{A}}(1-p)^{\abs{E\setminus A}}q^{c(V,A)}$. To
this end, define
\begin{equation}
  \label{eq:eord-etc}
  e_{\ord} \bydef -d\log(1-e^{-\beta}),\;\;\;
  e_{\dis} \bydef d\beta - \log q, \;\;\;
  \kappa \bydef \frac{1}{2}\log (e^{\beta}-1).
\end{equation}
Further, define the \emph{size
  $\norm{\V{\gamma}}$} of a contour $\V{\gamma}$ (resp.\ \emph{size}
$\norm{\V{S}}$ of an interface $\V{S}$) by
\begin{equation}
  \label{eq:size}
  \norm{\V{\gamma}} \bydef \abs{ \V{\gamma}\cap \bigcup_{e\in E}\V{e}},
  \quad \norm{\V{S}} \bydef \abs{ \V{S} \cap \bigcup_{e\in E}\V{e}}.
\end{equation}
This is the number of intersections of $\V{\gamma}$ (resp.\ $\V{S}$)
with $\bigcup_{e\in E}\V{e}$. For a continuum set $\V{\Lam}$ we write
$| \V{\Lam}|$ for $| \V{\Lam} \cap \tor|$, that is, the number of
vertices of $\tor$ in $\V{\Lam}$ in the embedding of $\tor$ into
$\ctor$.  This will cause no confusion as we never need to measure the
volume of a continuum set.

Using these definitions, $w(A)$ can be written as
\begin{equation}
  \label{eq:eweight}
  w(A) = q^{c(\V{A}_{\ord})} e^{-e_{\dis}\abs{\V{A}_{\dis} }}
  e^{-e_{\ord}\abs{\V{A}_{\ord} }}
  \prod_{\V{S}\in\Interface}e^{-\kappa\norm{\V{S}}}
  \prod_{\V{\gamma}\in\Gamma} e^{-\kappa\norm{\V{\gamma}}},
\end{equation}
where $c(\V{A}_{\ord})$ is the number of connected components of
$\V{A}_{\ord}$.  The products run over the sets of interfaces and
contours associated to the edge configuration $A$, respectively. We
indicate the derivation of~\eqref{eq:eweight} in
Section~\ref{sec:deriv-cont-repr} below; see
also~\cite[p.13-15]{borgs2012tight}. Since {(recall $\Omega=2^{E}$)}
\begin{equation}
  \label{eq:Znew}
  Z = Z^{\text{RC}}_{\tor}(1-e^{-\beta},q) =  \sum_{A\in\Omega}w(A) \, ,
\end{equation}
it follows from \eqref{eq:eweight} and Lemma~\ref{lem:rep} that
\begin{equation}
  \label{eq:Z-mci}
  Z = \sum_{(\Interface,\Contour)} q^{c(\V{A}_{\ord})}
  e^{-e_{\dis}\abs{\V{A}_{\dis} }}
  e^{-e_{\ord}\abs{\V{A}_{\ord} }}
  \prod_{\V{S}\in\Interface}e^{-\kappa\norm{\V{S}}}
  \prod_{\V{\gamma}\in\Contour}e^{-\kappa\norm{\V{\gamma}}},
\end{equation}
where the sum runs over matching collections of contours and interfaces. This is the
contour and interface network representation of the random cluster model partition
function. 

In what follows it will be necessary to separate different contributions to
$Z$. To this end, let
\begin{equation}
  \label{eq:newsplit}
  \Omega_{\tunnel} \bydef \{A\in\Omega \mid \Interface(A)\neq
  \emptyset\}, \quad \Omega_{\rest} \bydef \Omega\setminus\Omega_{\tunnel},
\end{equation}
and define the corresponding partition functions
\begin{equation}
  \label{eq:newsplit-1}
  Z_{\tunnel} \bydef \sum_{A\in\Omega_{\tunnel}}w(A), \quad Z_{\rest}
  \bydef \sum_{A\in\Omega_{\rest}} w(A).
\end{equation}
By~\eqref{eq:Z-mci} $Z_{\rest}$ can be expressed in terms of contours
alone. We will see later that $Z_{\tunnel}$ is very small compared to
$Z_{\rest}$, and so the task of approximating $Z$ is essentially the
task of approximating $Z_{\rest}$.

\subsubsection{Derivation of contour representation}
\label{sec:deriv-cont-repr}

We briefly indicate how to obtain~\eqref{eq:eweight}. Recall that
$G_{A}$ denotes the graph $(V(G),A)$. Let
$\norm{\delta A} = |\delta_{1}A| + 2|\delta_{2}A|$, where $\delta_{k}A$
is the set of edges in $E\setminus A$ that contain $k$ vertices in
$V(A)$. Observe
\begin{align}
  \label{eq:c}
  c(V,A) &= c(G_{A}) + |V\setminus V(A)| \\
  \label{eq:bdry}
  2|A| &= 2d|V(A)| - \norm{\delta A}.
\end{align}
The first of these relations follows since every vertex not contained
in an edge of $A$ belongs to a singleton connected component, and the
second is a counting argument. Using these relations one obtains
\begin{equation}
  \label{eq:eweight2}
  w(A) = q^{c(G_{A})} e^{-e_{\dis}|V\setminus V(A)|}e^{-e_{\ord}|V(A)|}e^{-\kappa\norm{\delta A}}.
\end{equation}

To pass from~\eqref{eq:eweight2} to~\eqref{eq:eweight} requires just a
few observations. First, $c(G_{A})$ equals the number of components of
$\V{A}$, which is the number of connected components of
$\V{A}_{\ord}$. Second, $|V(A)| = |\V{A}_{\ord}|$, and similarly
$|V\setminus V(A)| = |\V{A}_{\dis}|$.  Lastly, $\norm{\delta A}$ is
precisely the sum of sizes of the contours and interfaces, as each
contribution to $\norm{\delta A}$ is given by a transverse
intersection of an edge $\V{e}$ with the boundary of $\V{A}$.

\subsection{External contour representations} 
\label{sec:Dis-Ord-Rep}
Next we will take the first steps to construct a representation of
$Z_\rest$ as the sum of polymer model partition functions.  We begin
with some basic results and definitions. Fix an arbitrary point
$\V{x}_{0}\in\ctor$ that cannot be contained in any contour, and let
$\sqcup$ denote disjoint union.

\begin{lemma}[{\cite[Lemma~4.3]{borgs2012tight}}]
  \label{lem:contour-split}
  For any contour $\V{\gamma}$, $\ctor\setminus\V{\gamma}$ has exactly
  two components.
\end{lemma}

\begin{defn}
  \label{def:ext}
  Let $\V{\gamma}$ be a contour, and suppose
  $\ctor\setminus\V{\gamma} = \V{C}\sqcup \V{D}$. Then the
  \emph{exterior} $\Ext \V{\gamma}$ of $\V{\gamma}$ is $\V{C}$ if
  $\abs{\V{C}}>\abs{\V{D}}$, and is $\V{D}$ if the inequality is
  reversed. In the case of equality the exterior is the component
  containing $\V{x}_{0}$. The \emph{interior} $\Int \V{\gamma}$ of
  $\V{\gamma}$ is the component of $\ctor\setminus\V{\gamma}$ that is
  not $\Ext \V{\gamma}$.
\end{defn}

Note that the notion of exterior is defined relative to $\ctor$,
though we omit this from the notation.

\begin{remark*}
  This is a different definition of exterior than is used
  in~\cite{borgs2012tight}; our definition is more convenient for
  algorithmic purposes. Most of the results of~\cite{borgs2012tight}
  concerning the interiors/exteriors of contours apply verbatim with
  this change, and whenever we use these results we will remark on why
  they apply.
\end{remark*}

If two contours $\V{\gamma}$ and $\V{\gamma}'$ are compatible, then we
write (i) $\V{\gamma}<\V{\gamma}'$ if
$\Int \V{\gamma} \subset \Int \V{\gamma}'$ and (ii)
$\V{\gamma} \bot \V{\gamma}'$ if
$\Int \V{\gamma} \cap \Int \V{\gamma}' = \emptyset$.  Given a matching collection of contours $\Contour$, $\V{\gamma}\in\Contour$
is an \emph{external contour} if there does not exist
$\V{\gamma}'\in\Contour$ such that $\V{\gamma}'<\V{\gamma}$. 
The \emph{exterior} of a matching collection of contours $\Contour$ is
\begin{equation}
  \label{eq:ext-pcc}
  \Ext \Contour \bydef \bigcap_{\V{\gamma}\in\Contour}\Ext\V{\gamma}.
\end{equation}
If $\Contour$ is matching, then $\Ext\Contour$ is a
connected subset of $\tor$. This follows by noting
that~\cite[Lemma~5.5]{borgs2012tight} holds with
Definition~\ref{def:ext} of the interior and exterior, and given this,
the connectedness of $\Ext\Contour$ follows by the argument
in~\cite[Lemma~5.6]{borgs2012tight}. Note that since $\Ext\Contour$ is contained
in $\tor\setminus\bigcup_{\V{\gamma}\in\Contour}\V{\gamma}$, this
implies that $\Ext\Contour$ is labelled either $\ord$ or $\dis$. 

As usual in Pirogov--Sinai theory, see,
e.g.~\cite[Section~6.2]{borgs2012tight}, it is useful to resum the
matching compatible contours that contribute to~\eqref{eq:Z-mci}
according to the external contours of the configuration. To make this
precise, we require several definitions. A matching collection of
contours $\Contour$ is \emph{mutually external} if
$\V{\gamma}\bot\V{\gamma}'$ for all
$\V{\gamma}\neq\V{\gamma}'\in\Contour$. For a continuum set
$\V \Lam \subseteq \ctor$, we say a contour $\V{\gamma}$ is \emph{a
  contour in $\V{\Lam}$} if
$d_{\infty}(\V{\gamma}, \ctor\setminus\V{\Lam})\geq 1/2$. The distance
to the empty set is infinite by convention.

Write $\cC(\V{\Lam})$ for the set of contours in $\V{\Lam}$, and $\cC
= \cC(\ctor)$ for the set of all contours.  For
$\V{\Lam}\subseteq\ctor$ define
$\cG^{\ext}(\V{\Lam})$ to be the set of matching mutually external
contours in $\V{\Lam}$, and then define
\begin{align}
  \label{eq:matchextord}
  Z_{\ord}(\V{\Lam}) 
  &\bydef \sum_{\Contour \in \cG^{\ext}_{\ord} (\V{\Lam})} 
    e^{- e_{\ord} | \V{\Lam} \cap \Ext \Contour  |} \prod_{\V{\gamma} \in
    \Contour} e^{-\kappa \|\V{\gamma}\|} Z_{\dis} (\Int \V{\gamma})  \\
  \label{eq:matchextdis}
  Z_{\dis} (\V{\Lam}) 
  &\bydef \sum_{\Contour \in \cG^{\ext}_{\dis}(\V{\Lam})} 
    e^{- e_{\dis} | \V{\Lam} \cap \Ext \Contour|}
    \prod_{\V{\gamma} \in \Contour} e^{-\kappa \|\V{\gamma}\|}
    q Z_{\ord} (\Int \V{\gamma}),
\end{align}
where the sums in \eqref{eq:matchextord} and \eqref{eq:matchextdis}
run over sets of matching mutually external contours in which
$\Ext\Gamma$ is labelled $\ord$ and $\dis$, respectively. This is the
desired resummation. In the special case $\V{\Lam}=\ctor$ these
partition functions represent the sums of $w(A)$ over
\begin{align}
\label{eq:Z-split-ord}
  \Omega_{\ord} 
  &\bydef \{A\in\Omega\setminus\Omega_{\tunnel} \mid
    \text{$\Ext\Contour(A)$ is labelled $\ord$}\}, 
  \\
  \label{eq:Z-split-dis}
  \Omega_{\dis} 
  &\bydef \{A\in\Omega\setminus\Omega_{\tunnel} \mid
    \text{$\Ext\Contour(A)$ is labelled $\dis$}\}.
\end{align}
That is, we get a decomposition $Z_{\rest} = qZ_{\ord}+Z_{\dis}$, where
\begin{equation}
  \label{eq:Zmatch}
  Z_{\ord} = q^{-1}\sum_{A\in\Omega_{\ord}}w(A), \quad Z_{\dis} = \sum_{A\in\Omega_{\dis}}w(A).
\end{equation}

Subsection~\ref{sec:RCM-form} will give interpretations of these
quantities in terms of random cluster model partition functions for
many other choices of $\V{\Lam}$.

\subsection{Labelled contours}
\label{sec:labell-cont-rcm}
This subsection introduces labelled contours and establishes some
basic properties of these objects. These properties will ensure that
we can efficient enumerate labelled contours.

In Definition~\ref{def:labl} we associated a labelling to an entire
collection of matching and compatible contours and interfaces. For
collections of contours, since each contour splits $\tor$ into two
pieces, it is more convenient to associate the labelling to individual
contours. We do this by assigning a label to $\Int \V{\gamma}$ (resp.\
$\Ext \V{\gamma}$) according to the label of the region of
$\tor\setminus \cup_{\V{\gamma}\in\Contour}\V{\gamma}$ adjacent to
$\V{\gamma}$ contained in $\Int \V{\gamma}$ (resp.\
$\Ext \V{\gamma}$).

A \emph{compatible set of labelled contours $\Contour$} is a set of
compatible contours $\Contour$ such that the connected components of
$\tor\setminus \cup_{\V{\gamma}\in\Contour}\V{\gamma}$ are assigned
the same labels by the labelled contours. More precisely, for a
component $\V{B}$ of
$\tor\setminus \cup_{\V{\gamma}\in\Contour}\V{\gamma}$,
$\partial \V{B}$ is a union of compatible contours
$\V{\gamma}_{0}, \dots, \V{\gamma}_{k}$ for some $k\geq 0$, and (up to
relabelling) either (i) $\V{\gamma}_{i}<\V{\gamma}_{0}$ for
$i=1,\dots, k$ or (ii) $\V{\gamma}_{i}\perp \V{\gamma}_{j}$ for
  $i\neq j$. The condition of compatibility of the labels in
  the first case is that the interior label of $\V{\gamma}_{0}$ is
the same as the exterior label of $\V{\gamma}_{i}$ for all
$i=1,\dots k$, and in the second case is that all exterior
  labels agree.

By construction, the set of collections of matching and compatible
contours is the same as the set of collections of compatible labelled
contours. The advantage of the latter is that it enables us to define
a labelled contour $\V{\gamma}$ to be \emph{ordered} if its exterior
label is $\ord$, and \emph{disordered} if its exterior label is
$\dis$. We let $\cC_{\ord}(\V{\Lambda})$ and $\cC_{\dis}(\V{\Lambda})$
denote the sets of labelled contours in $\V{\Lambda}$ with external
labels $\ord$ and $\dis$, respectively, with
$\cC_\ord = \cC_{\ord}(\ctor)$ and $\cC_{\dis} = \cC_{\dis}(\ctor)$.
The next lemma gives a way to construct a labelled contour
$\V{\gamma}$ from an edge configuration.

\begin{lemma}
  \label{lem:construct}
  Let $\ell\in\{\ord,\dis\}$, let $\V{\gamma}\in \cC_{\ell}$,
  and $\V{\Lambda}=\Int\V{\gamma}$. Then
  \begin{itemize}
  \item If $\ell=\dis$, let $E'(\V\Lambda)$ be set of edges contained in
    $\V{\Lambda}$. Then $\V{\gamma}$ is the unique component of
    $\partial \V{A}$ where $A = E'(\V\Lambda)\subset E$.
  \item If $\ell=\ord$, let $E'(\V\Lambda)$ be the set of edges whose
    midpoints are contained in $\V{\Lambda}$. Then $\V{\gamma}$ is the
    unique component of $\partial \V{A}$ where
    $A=E\setminus E'(\V\Lambda)$.
  \end{itemize}
\end{lemma}
\begin{proof}
  These claims follows from~\cite[Lemma~5.1]{borgs2012tight}; see the
  proof of~\cite[Lemma~5.11]{borgs2012tight}.\footnote{These results
    rely only on the geometry of hypercubes and not on the definitions
    of interior/exterior.} 
\end{proof}

Lemma~\ref{lem:construct} gives a way to construct a given contour
from some set of edges $A$. For our algorithms it will be important to
be able to generate contours from a relatively small set of edges. We
first explain how to do this for disordered contours.  

Suppose $\V{\gamma}\in \cC_{\dis}$ and let $\Lambda = \Int
\V{\gamma}\cap\tor$. Define
\begin{equation}
  \label{eq:active2s}
  \cE_{\V{\gamma}} \bydef  \{ e=\{i,j\} \mid i,j\in\Lambda, \,\,
  d_{\infty}(\text{mid}(\V{e}),\V{\gamma})\geq 3/4 \},
\end{equation}
where $\text{mid}(\V{e})$ denotes the midpoint of the edge $\V{e}$;
this is the vertex of $\htor$ on the two-step path from $i$ to $j$ in
$\htor$. 

\begin{lemma}
  \label{lem:edge-dis}
  Suppose $\V{\gamma}\in \cC_{\dis}$ and let
  $\V{\Lambda} = \Int \V{\gamma}$. Suppose
  $F\subseteq \cE_{\V \gamma}$ and let $A = E'\setminus F$, where
  $E'=E'(\Lambda)$ is defined as in Lemma~\ref{lem:construct}. Let
  $\Contour$ be the set of contours in $\partial \V{A}$. Then
  $\V{\gamma}\in\Contour$, and for all $\V{\gamma}'\in\Contour$ with
  $\V{\gamma}'\neq\V{\gamma}$ we have $\V{\gamma}' <
  \V{\gamma}$. Moreover, all sets of matching contours consisting of
  $\gamma$ and contours in $\Int \gamma$ arise from such $F$.
  \end{lemma}
\begin{proof}
  We begin by recalling an alternate construction of
  $\V{A}$ from~\cite{borgs2012tight}. Let $E\subset
  E(\tor)$, and let $D\subset E$. Set
  $D^{\star}$ to be the set of
  $(d-1)$-dimensional unit hypercubes dual to the edges of
  $D$, and set 
  \begin{equation*}
    V_{-}(D) = \set{x\in V(\tor) \mid \{x,y\} \in D \text{ if } \{x,y\} \in E}.
  \end{equation*}
  Set $\V{D}_{\dis}$ to be the union of the open
  $3/4$-neighborhood of $V_{-}(D)$ and the open
  $1/4$-neighborhood of
  $D^{\star}$. Then by~\cite[Lemma~5.1, (iv)]{borgs2012tight}, if $D =
  E\setminus A$, $\V{E}\setminus \V{A} =
  \V{D}_{\dis}$. I.e., $\V{D}_{\dis}$ is the disordered region associated to
  $A$ (relative to the region
  $\V{E}$).

  To prove the lemma, we apply this construction with
  $E=E'(\Lam)$ and $D=F$. The definition of
  $\cE_{\V{\gamma}}$ ensures that both the open
  $3/4$-neighborhoods of the included vertices and the open
  $1/4$-neighborhoods of the included dual facets are at distance at
  least $1/2$ from $\V{\gamma}$. This implies that
  $\V{\gamma}$ is a boundary component of $\V{E\setminus
    F}$, and the first claim follows as all other boundary components
  are adjacent to
  $\V{D}_{\dis}$. The second claim follows from the bijection of
  Lemma~\ref{lem:rep}, which restricts to a bijection in this setting.
\end{proof}

\begin{lemma}
  \label{lem:dis-edge-con}
  Suppose $\V{\gamma}\in \cC_{\dis}$. Then there is a connected graph
  with edge set $A$ such that (i) $\abs{A}\leq 2d\norm{\V{\gamma}}$
  and (ii) 
  $\V{\gamma}$ is the outermost contour in $\partial \V{A}$.
\end{lemma}
\begin{proof}
  Choose $F=\cE_{\V{\gamma}}$ in Lemma~\ref{lem:edge-dis}. Then the
  subgraph of $\tor$ induced by $E''=E'(\Lambda)\setminus F$ is
  connected: if not $\partial \V{{E''}}$ would contain two compatible
  exterior contours as the boundaries of the thickenings of the
  connected components of $E''$ are compatible. This would contradict
  the conclusion of Lemma~\ref{lem:edge-dis} that there is a unique
  exterior contour.

  The bound on the size of $A$ is crude; it can be obtained by noting
  that the included edges all contain a vertex from which there is an
  edge outgoing from $\Lambda$, and the count of these vertices is a
  lower bound for $\norm{\V{\gamma}}$. Each of the vertices is
  contained in at most $2d$ edges.
\end{proof}

We now establish a similar way to construct an ordered contour from a
small edge set. The situation is slightly different due to the
differences between ordered and disordered contours in
Lemma~\ref{lem:construct}. Define, for $\V{\gamma}\in \cC_{\ord}$,
$\Lambda = \Int \V{\gamma} \cap \tor$,
\begin{equation}
  \label{eq:active1s}
  \cE_{\V{\gamma}} \bydef \{ \{i,j\} \mid i,j\in\Lambda\}.
\end{equation}

\begin{lemma}
  \label{lem:edge-ord}
  Suppose $\V{\gamma}\in \cC_{\ord}$ and $F\subseteq \cE_{\V{\gamma}}$. 
  Let $A = (E\setminus E'(\Lambda))\cup F$,
  where $E'(\Lambda)$ is defined as in Lemma~\ref{lem:construct}. Let
  $\Contour$ be the set of contours in $\partial \V{A}$. Then
  $\V{\gamma}\in\Contour$, and for all $\V{\gamma}'\in\Contour$ with
  $\V{\gamma}'\neq\V{\gamma}$ we have
  $\V{\gamma}'<\V{\gamma}$. Moreover, all sets of matching contours consisting of
  $\gamma$ and contours in $\Int \gamma$ arise from such $F$.
\end{lemma}
\begin{proof}
  The proof is essentially the same as for
  Lemma~\ref{lem:edge-dis}. Let $A' = E\setminus E'(\Lam)$. The
  set $\V{F}$ is disjoint from $\V{{A'}}$ as every vertex $i$ interior
  to $\V{\gamma}$ is at distance at least $3/4$ from
  $\V{\gamma}$. This implies $\partial \V{A}$ is the union of
  $\partial \V{{A'}}$ and $\partial \V{F}$, which implies the first
  claim. The second claim follows from the bijection of
  Lemma~\ref{lem:rep}, which restricts to a bijection in this
  setting.
\end{proof}

Two edges $e,f\in E$ are called \emph{$1$-adjacent} if
$d_{\infty}(\V{e},\V{f})\leq 1$. A set of edges $A$ is
\emph{$1$-connected} if for any $e,f\in A$, there is a sequence of
$1$-adjacent edges in $A$ from $e$ to $f$. In the next lemma,
$\partial \V{{A^{c}}}$ is the boundary of the thickening of the edge
set $A^{c} =  E \setminus A$.
\begin{lemma}
  \label{lem:ord-edge-con}
  Suppose $\V{\gamma}\in \cC_{\ord}$. Then there is a
  $1$-connected set of edges $A$ of size at most
  $\norm{\V{\gamma}}$ such that $\V{\gamma}$ is the outermost contour
  in $\partial \V{{A^{c}}}$.
\end{lemma}
\begin{proof}
  Let $A$ be the set of all edges that intersect $\V{\gamma}$. By the
  definition of $\norm{\cdot}$, $\abs{A}\leq \norm{\V{\gamma}}$. By
  Lemma~\ref{lem:edge-ord} $\V{\gamma}$ is the outermost contour in
  $\V{{A^{c}}}$, as $A^{c} = E'(\Lam)\cup \cE_{\dis}(\Lam)$. The
  $1$-connectedness of $A$ follows from the connectedness of
  $\V{\gamma}$ and the observation that every point of $\V{\gamma}$ is
  at most $d_{\infty}$ distance $1/2$ from an edge in $A$.
\end{proof}

\subsection{Contour Enumeration}
\label{sec:contour-enumeration}

This section uses the results of the previous subsection to guarantee the
existence of an efficient algorithm for enumerating contours. This
requires a few additional lemmas.
\begin{lemma}
  \label{lem:iso}
  For all $\V{\gamma} \in \cC$,
  $\abs{\Int\V{\gamma}}\leq \norm{\V{\gamma}}^{2}$, and
  $\abs{\Int\V{\gamma}}\leq (n/2) \norm{\V{\gamma}}$.
\end{lemma}
\begin{proof}
  This follows by~\cite[Lemma~5.7]{borgs2012tight}, as
  the interior of a contour as defined by Definition~\ref{def:ext} is
  always smaller than the definition of the interior of a contour
  in~\cite{borgs2012tight}.
\end{proof}

\begin{lemma}
  \label{lem:findext}
  There is an algorithm that determines the vertex set
  $\Int \V{\gamma}\cap \tor$ in time $O(\norm{\V{\gamma}}^{3})$.
\end{lemma}
\begin{proof}
  Let $m =\norm{\V{\gamma}}$.  Let $G$ be the subgraph of $\tor$ that
  arises after removing all edges that intersect
    some $(d-1)$-dimensional face 
  in $\V{\gamma}$. Consider the following greedy algorithm to
  determine the {two} connected components of $G$. The
  algorithm starts at
  $C_{0}=x$, where $x$ is chosen such that it is contained in an edge
  not present in $G$. The algorithm determines the connected component
  containing $x$ in $G$ by adding at step $i+1$ the first vertex (with
  respect to lexicographic order) in $\tor\setminus C_{i}$ that
  neighbors $C_{i}$; if no neighbors exist the component has been
  determined. The $k$th step takes time at most $(2d)k$, so performing
  $N$ steps of this algorithm takes time $O(N^{3})$.

  Since $\abs{\Int \V{\gamma}}\leq \norm{\V{\gamma}}^{2}$ by
  Lemma~\ref{lem:iso}, we can stop the greedy procedure after
  $m^{2}+1$ steps. If the algorithm terminates due to this condition,
  the component being explored is the exterior component. The
  interior component can then
  be determined in at most $O(m^{3})$ additional steps by running the
  greedy algorithm from the neighbor of $x$ that is in the interior component.
  Otherwise the algorithm will have already terminated and determined
  the interior component.
\end{proof}

\begin{lemma}
  \label{lem:enum}
  Fix an edge $e\in E$. There is an algorithm to construct all
  contours $\V{\gamma}\in \cC_{\ord}$ that (i) can arise from a
  connected edge set $A$ that contains $e$ and (ii) have
  $\norm{\V{\gamma}}\leq m$. The algorithm runs in time $\exp(O(m))$.

  Similarly, there is an $\exp(O(m))$-time algorithm to construct all contours
  $\V{\gamma}\in \cC_{\dis}$ that (i) can arise from an edge set $A$
  such that $A^{c}$ is $1$-connected and contains $e$ and (ii) have
  $\norm{\V{\gamma}}\leq m$.
\end{lemma}
\begin{proof}
  We first consider disordered contours, and begin by enumerating all
  connected sets $A$ of edges that contain $e$ that are of size at
  most $2dm$. This can be done in time $\exp(O(m))$. If $2m\leq n$
  then we consider the enumerated edge sets as subsets of
  $E(\mathbb{T}_{2m}^{d})$; otherwise we consider them as subsets of
  $E(\tor)$. 
  
  For each edge set $A$, construct $\partial \V{A}$ and take the
  outermost contour (if there is not a single outermost contour,
  discard the result).  By Lemma~\ref{lem:dis-edge-con} this generates
  all disordered contours of size at most $m$ that arise from
  connected edge sets containing $e$. We obtain the desired list of
  contours by removing any duplicates, which takes time at most
  $\exp(O(m))$. The remainder of the proof shows that the operations
  in this paragraph can be done in time polynomial in $m$.

  The constructions of $\partial \V{A}$ takes time at most $O(m)$ as
  it is a $\Z_{2}$ sum of $(d-1)$-dimensional facets, and determining
  these facets takes a constant amount of time (depending only on the
  dimension $d$) for each edge. Determining if a
  component of $\partial \V{A}$ is a contour can be done by computing
  the winding number of the component; this takes time
  $O((2dm\wedge n)K)$ for a component with $K$ facets by
  Lemma~\ref{lem:wind-comp}.  Determining the interior of a given
  contour takes time at most $O(m^{3})$ by Lemma~\ref{lem:findext},
  and hence we can check if $\Int \V{\gamma}' \subset \Int \V{\gamma}$
  for all pairs in time $O(m^{4})$ since there are at most $m^{2}$
  contours. This completes the proof for disordered contours.

  For ordered contours the argument applies nearly verbatim. The
  changes are as follows. First, enumerate $1$-connected sets $A^{c}$
  that contain $e$. Secondly, to see that we get the desired contours,
  appeal to Lemma~\ref{lem:ord-edge-con}. Lastly, computing
  $\partial \V{A}$ takes time $O( (m\wedge n)^{d})$ which is
  polynomial in $m$; this is by our choice of torus in the first
  paragraph of the proof.
\end{proof}

The next definition is useful for inductive arguments involving
contours. 
\begin{defn}
\label{defLevel}
  The \emph{level} $\cL(\V{\gamma})$ of a contour $\V{\gamma}$ is
  defined inductively as follows. If $\V{\gamma}$ is \emph{thin},
  meaning $\cC(\Int\V{\gamma})=\emptyset$, then
  $\cL(\V{\gamma})=0$. Otherwise,
     $ \cL(\V{\gamma})=1+\max\{\cL(\V{\gamma}') \mid 
      \V{\gamma}'<  \V{\gamma}\}.$
\end{defn}

Call a set $\V{\Lambda} \subseteq \ctor$ a \emph{region} if
$\V{\Lambda} = \ctor $ or if $\V \Lam$ is a connected component of
$\ctor \setminus \partial \V{A}$ for some $A\subset E$.  In the former
case set $\partial \V{\Lam}= \emptyset$, and in the latter case set
$\partial \V{\Lam}$ to be the union of all connected components of
$\partial \V{A}$ incident to $\V{\Lam}$. In particular if
$\V{\Lam} = \Int \V{\gamma}$ for some contour $\V \gamma$, then
$\V{\Lam}$ is a region and $\partial \V{\Lam} = \V{\gamma}$.  Finally,
for compatible contours $\V{\gamma}_1, \dots, \V{\gamma}_t$, define
$\| \V\gamma_1 \cup \cdots \cup \V{\gamma}_t\| = \| \V \gamma_1 \| +
\cdots + \| \V{\gamma}_t\|$. We conclude this subsection by stating
our main algorithmic result on efficiently computing sets of contours.

\begin{prop}
  \label{prop:enum}
  There is an
  $O((\abs{\V{\Lambda}}+\norm{\partial \V{\Lambda}})\exp(O(m)))$-time
  algorithm that, for all regions $\V{\Lam}$, (i) enumerates all contours in
  $\cC_{\ord}(\V{\Lambda})\cup\cC_{\dis}(\V{\Lambda})$ with size at
  most $m$ and (ii) sorts this list consistent with the level
  assignments.
\end{prop}
\begin{proof}
  We begin by proving the first item. Apply Lemma~\ref{lem:enum} for
  each edge contained in $\V{\Lambda}$. This takes time
  $O(\abs{\V{\Lambda}}\exp(O(m)))$ as there are at most $2d$
  edges in $\V{\Lambda}$ for each vertex of $\tor$ in
  $\V{\Lambda}$. The output is a (multi-)set of contours of size at
  most $m$ contained in $\tor$. Trim the resulting list of contours to
  remove duplicates.

  By Lemma~\ref{lem:findext} in time $\exp(O(m))$ we can determine
  $\Int \V{\gamma}$ for every $\V{\gamma}$ from the list obtained in
  the first paragraph. We determine the list of level zero contours by
  iterating through the list, checking for each $\V{\gamma}$ if
  $\V{\gamma}'<\V{\gamma}$ for some other $\V{\gamma}'\neq \V{\gamma}$
  in the list. If not, assign $\V{\gamma}$ level $0$. This takes time
  at most $\exp(O(m))$. We continue by running the same operation on
  the sublist of all contours of level at least one, i.e., the sublist
  of contours not assigned level $0$. If $\V{\gamma}$ has level at
  least one and there is no $\V{\gamma}'<\V{\gamma}$, $\V{\gamma}'$
  also of level at least one, then $\V{\gamma}$ is assigned level
  one. By repeating this we assign a level to every contour. The
  maximal level of a contour is $m^{2}$, the maximal size of the
  interior of a contour of size $m$, and hence the total running time
  is at most $m^{2}\exp(O(m)) = \exp(O(m))$.
  
  To conclude, trim the list to retain only contours $\V{\gamma}'$
  contained in $\V{\Lambda}$. This can be done by removing contours at
  distance less than $1/2$ from $\V{\gamma}$. Computing this distance
  takes time $O(\norm{\V{\gamma}} \norm{\V{\gamma}'})$, which is at
  most $O(\norm{\V{\gamma}}m)$.
\end{proof}

\subsection{Polymer representations for $Z_{\ord}$ and
  $Z_{\dis}$}
\label{sec:extern-cont-repr-1}

To obtain polymer model representations of $Z_{\ord}$ and $Z_{\dis}$,
define $\pc_{\ord}(\V{\Lambda})$ and $\pc_{\dis}(\V{\Lam})$ to be the
sets of compatible collections of contours in $\V{\Lam}$ that are
labelled $\ord$ and $\dis$, respectively. Define
\begin{equation}
  \label{eq:Kweight}
  K_{\ord}(\V{\gamma})  
  = e^{- \kappa \| \V{\gamma} \|} \frac{Z_{\dis} (\Int \V{\gamma})
  }{  Z_{\ord} (\Int \V{\gamma})   }, \qquad
  K_{\dis}(\V{\gamma})  
  = e^{- \kappa \| \V{\gamma} \|} \frac{q Z_{\ord} (\Int \V{\gamma})
  }{  Z_{\dis} (\Int \V{\gamma})   }.
\end{equation}
By following a well trodden path in Pirogov--Sinai theory (see,
e.g.,~\cite[p.28]{borgs2012tight}
or~\cite[p.28]{helmuth2018contours}), these definitions give the
following representations for $Z_{\ord}$ and $Z_{\dis}$ as partition
functions of abstract polymer models:
\begin{align}
  \label{eqZordPoly}
  Z_{\ord}(\V{\Lam}) 
  &= e^{- e_{\ord} | \V\Lam|} \sum_{\Contour\in\pc_{\ord}(\V{\Lambda})} \prod_{\V{\gamma} \in \Contour}
    K_{\ord}(\V{\gamma}) \\ 
  \label{eqZdisPoly}
  Z_{\dis}(\V{\Lam}) 
  &=  e^{- e_{\dis} | \V\Lam|} \sum_{\Contour\in\pc_{\dis}(\V{\Lambda})} \prod_{\V{\gamma} \in
    \Contour} 
    K_{\dis}(\V{\gamma}) \,,
\end{align}
where the sums run over collections of compatible labelled contours in $\V{\Lam}$
with external label $\ord$ and $\dis$, respectively.

In fact, for $\ell\in \{\ord,\dis\}$, the above formulas represent
$Z_{\ell}(\V{\Lam})$ as the partition function of a polymer model in
the form discussed in Section~\ref{secPolymer}, i.e., where polymers
are subgraphs of a fixed graph $G$ with bounded degree. In detail,
recalling the discussion in Section~\ref{sec:cont}, we consider
contours as induced subgraphs of (a subgraph of) the bounded-degree graph
$\htors$. Thus $|\V{\gamma}|$ is the number of vertices in a contour
when represented as a subgraph. Condition~\eqref{eqPeierls2} holds with
$b=1$ since $\norm{\V{\gamma}}\geq |\gamma|$ by~\eqref{eq:size}. The
more substantial hypothesis \eqref{eqPolymerKP2} will be verified in
later sections for appropriate choices of the label and of $\beta$.

In the sequel we will write $|\V{\Lam}|_{\htors}$ for the size of set
of vertices of $\htors$ that are part of some contour $\V{\gamma}$ in
$\cC_{\ell}(\V{\Lam})$ for some $\ell$. The next technical lemma shows
it is enough to find algorithms that are polynomial time in
$|\V{\Lam}|_{\htors}$.
\begin{lemma}
  \label{lem:polygraphsize}
  For $\V{\Lam}$ a continuum set,  
  $|\V{\Lam}|_{\htors}$ is polynomial in $|\Lam|$.
\end{lemma}
\begin{proof}
  By construction, contours inside
  $\V{\Lam}$ arise from edge configurations of edges inside
  $\V{\Lam}$. The number of such edges is at most $2d$ times the
  number of vertices inside. Since contours are boundaries of unions of
  $(d-1)$-dimensional hypercubes centered at vertices in $\htors$ that
  lie on edges, this proves the claim, since there are a bounded
  number of such hypercubes associated to each
  edge.
\end{proof}

\subsection{Random cluster model formulations of contour partition
  functions}
\label{sec:RCM-form}

The definitions of the partition functions $Z_{\ord}(\V{\Lambda})$ and
$Z_{\dis}(\V{\Lambda})$ 
in~\eqref{eq:matchextord} and~\eqref{eq:matchextdis} only involve
contours. In general, these contour partition functions do not
correspond to random cluster model partition functions due to the
exclusion of interfaces. However, we will show that when
$\Lambda = \Int \V{\gamma} \cap \tor$ can be embedded as a subgraph of
$\Z^{d}$, there is such an interpretation.

To make this precise, recall the definitions~\eqref{eq:Zfree} and~\eqref{eq:Zwired} of
$Z^f_{\Lam}$ and $Z^w_{\Lam}$ for $\Lam\subset \Z^{d}$ such that the
subgraph $G_{\Lam}$ induced by $\Lam$ is simply
connected. Recall that $p=1-e^{-\beta}$.
\begin{prop}
  \label{prop:scregion}
  Suppose $\Lambda\subset\Z^{d}$ is simply connected, and let
  $n = 3 |\Lam|$.  Then there are contours
  $\V{\gamma}_{\dis}\in \cC_{\dis}(\tor)$ and
  $\V{\gamma}_{\ord}\in \cC_{\ord}(\tor)$ determined by $\Lam$ such that
  \begin{equation*}
    Z_{\dis} (\Int \V{\gamma}_{\ord}) =
    (1-p)^{-\frac{1}{2}\norm{\V{\gamma}_{\ord}}}Z^f_{\Lam}, \quad 
    Z_\ord(\Int \V{\gamma}_{\dis}) =
    q^{-1}p^{d\abs{\Int\V{\gamma}_{\ord}}-\abs{E(\Lam)}} Z^w_{\Lam} .
\end{equation*} 
\end{prop}
\begin{proof}
 
  Since $n=3|\Lam|$, we can embed $\Lam\subset \tor$. Moreover, the
  set of boundary vertices
  $\partial \Lam \bydef \{ i \in \Lam: \exists j \in \Z^{d}\setminus
  \Lam, (i,j) \in E(\Z^{d}) \}$ can be identified with
  $ \{i\in \Lam: \exists j\in \Lam^{c}, (i,j)\in E\}$. Thus the graphs
  $G_{\Lam}$ and $G_{\Lam}'$ used in the definitions of $Z^{f}_{\Lam}$
  and $Z^{w}_{\Lam}$ are the same whether defined by considering
  $\Lam$ as a subset of $\Z^{d}$ or $\tor$. Note that by our choice of
  $n$ we know that any component of $\partial \V{A}$
  will be a contour if $A$ is a subset of edges that are at graph
  distance at most two from $\Lam$. To see this in an elementary way, note
  that we can further consider $\Lam$ as a subset of $\tor$ such that
  the fundamental loops of $\tor$ are at distance at least (say) ten
  from $\V{\Lam}$. 

  We first consider the case of $Z^{f}_{\Lam}$. To do this, let
  $A_{0}\subset E$ be the set of edges with both endpoints
  in $\Lam^c$.  Let $\V{\gamma}_\ord$ be the unique contour in
  $\partial \V{{A_{0}}}$; the fact that there is a unique contour follows from the
  fact that $\Lam$ is simply connected.  
  By Lemma~\ref{lem:edge-ord}, for any subset $A$ of edges in
  $E(G_{\Lambda})=\cE_{\V{\gamma}_{\ord}}$, the contours of
  $\partial \V{A}$ are contained in $\Int
  \V{\gamma}_{\ord}$. Moreover, this Lemma ensures that by carrying
  out the contour construction of Section~\ref{sec:deriv-cont-repr}
  for subsets of edges $A' = A_{0}\cup A$ where all edges of $A$ are
  from $E(G_{\Lambda})$, we obtain all contour configurations
  $\Contour = \{\gamma_{\ord}\}\cup \Contour'$ where the contours of
  $\Contour'$ are contained in $\Int \V{\gamma}_{\ord}$. 

  To obtain the conclusion, note that (i) $\sum_{A'}w(A')$ is
  proportional to $Z^{f}_{\Lam}$, where the sum runs over these
  $A' = A_{0}\cup A$ described above, and (ii) $\sum_{A'}w(A')$ is
  proportional to $Z_{\dis}( \Int \V{\gamma}_{\ord})$. To obtain the
  proportionality constant we compare the contributions of the empty
  edge configuration (empty contour configuration), see~\eqref{eq:matchextdis}. These are,
  respectively, $q^{\abs{\Lam}}(1-p)^{\abs{E(\Lam)}}$ and
  $q^{\abs{\Lam}}(1-p)^{d\abs{\Lam}}$. The ratio of these terms is
  $(1-p)^{{-\frac{1}{2}}\norm{\V{\gamma}_{\ord}}}$ since $\norm{\V{\gamma}_{\ord}}$
  is exactly the number of edges between $\Lambda$ and
  $\Lambda^{c}$.

We now consider the case of $Z^{w}_{\Lam}$. Let $A=E(G_{\Lam})$, and
consider the ordered contour $\V{\gamma}'$ that arises from the edge
set $E\setminus A$. Define
\begin{equation}
  \tilde A \bydef A\cup \{e\in E \mid d_{\infty}(\text{mid}(\V{e}),\V{\gamma}')\leq
  1/2\},
\end{equation}
the set of edges whose midpoints are either in the interior of
$\V{\gamma}'$ or within distance $1/2$ of $\V{\gamma}'$. Then set
$\V{\gamma}_{\dis}$ to be the single contour in
$\partial \V{{\tilde A}}$; there is only one contour in this set by
the assumption $\Lam$ is simply connected. Note that $A$ is precisely
$\cE_{\V{\gamma}_{\dis}}$ as defined above Lemma~\ref{lem:edge-dis},
and hence there is a bijection between contour configurations in
$\Int \V{\gamma}_{\dis}$ and subsets of $\tilde A$ in which each edge
not in $A$ is occupied. As for the case of $Z^{f}_{\Lam}$ we can now
conclude, as summing over such edge sets is proportional to both
$Z_{\ord}(\Int \V{\gamma}_{\dis})$ (recall~\eqref{eq:Zmatch})
and $Z^{w}_{\Lam}$. To compute the proportionality constant, we
compare the all occupied configuration to the empty contour
configuration, see~\eqref{eq:matchextord}. This gives, respectively, $qp^{\abs{E(\Lam)}}$ and
$e^{-e_{\ord}\abs{\Int \V{\gamma}_{\ord}}}$, and hence
\begin{equation}
  Z_{\ord}(\Int \V{\gamma}_{\dis}) =
  q^{-1}p^{d\abs{\Int\V{\gamma}_{\ord}}-\abs{E(\Lam)}} Z^{w}_{\Lam}.
\end{equation}
\end{proof}

\section{Contour model estimates}
\label{secEstimates}

In this section we state several estimates related to the contour
representations from the previous section. Recall the definition~\eqref{eq:Znew} of
$Z$.

\begin{lemma}[{\cite[Lemma~6.1(a)]{borgs2012tight}}]
  \label{lem:Z-split}
 There are constants $c>0$, $q_{0}=q_{0}(d)<\infty$, and $n_{0}<\infty$ such
  that if $q\geq q_{0}$, $n\geq n_{0}$, and $\beta\geq \beta_{c}$,
  \begin{equation}
    \label{eq:tunnel-small}
    \frac{Z_{\tunnel}}{Z} \leq \exp(-c\beta n^{d-1}).
  \end{equation}
\end{lemma}
In what follows $c$ will always denote the constant from
Lemma~\ref{lem:Z-split}, and $q_0$ and $n_0$ will always be at least
as large as the constants in the lemma. More precisely, several lemmas will
  require $q_{0}$ to be chosen large enough, and we implicitly take
  $q_{0}$ to be the maximum of these requirements. We also choose $n_0$ large enough so that via~\eqref{eq:betac} we have $\beta_c > \beta_h$.  Lemma~\ref{lem:Z-split}
ensures that $Z_{\tunnel}$ is neglectable when approximating $Z$ up to
relative errors $\epsilon \gg \exp (  -c\beta n^{d-1})$.  We will also need to
know that $Z_{\dis}$ is neglectable when $\beta>\beta_{c}$. This
requires two lemmas.

\begin{lemma}
  \label{lemBCTsup}
  If $q \ge q_0$, $n\geq n_{0}$, and $\beta > \beta_c$ there exist constants
  $a_{\dis}>0$ and  $f>0$ so that if $\eps_{n}\bydef 2\exp(-c\beta n)$, then
  \begin{equation}
    Z_{\ord} \ge \exp(-(f +\eps_n ) n^d), \quad    
               Z_{\dis} \le \exp((- f + \eps_n) n^d) \max_{\Gamma \in \cG^{\ext}_{\dis}} e^{- \frac{a_{\dis}}{2} |\Ext \Gamma|} \prod_{\V{\gamma} \in \Gamma} e^{-\frac{c}{2} \beta \| \V{\gamma} \| } .
  \end{equation}
\end{lemma}
\begin{proof}
  With $a_{\dis}\geq 0$ this follows
  from~\cite[Lemma~6.3]{borgs2012tight} provided $f=f_{\ord}$ for
  $\beta\geq \beta_{c}$, and that $f=f_{\ord}$ follows
  from~\cite[Lemma~A.3]{borgs2012tight}. What remains is to prove $a_{\dis}>0$ when
  $\beta>\beta_{c}$. The results of~\cite{laanait1991interfaces} imply that there
  is a unique Gibbs measure for the random cluster model when
  $\beta>\beta_{c}$. If $a_{\dis}$ was $0$ for some $\beta>\beta_{c}$,
  then the argument establishing~\cite[Lemma~6.1 (b)]{borgs2012tight}
  implies the existence of multiple Gibbs measures, a contradiction.
\end{proof}

\begin{lemma}
  \label{lemSuperEstimates2a}
  If $q \ge q_0$, $n\ge n_0$, and $\beta > \beta_c$, then there exists
  a constant $b_{\dis}>0$ so that
  \begin{equation}
    \label{eq:SuperEstimates2a}
    \frac{ Z_{\dis}}{ Z } \le 2\exp(-b_{\dis} n^{d-1}) \,.
  \end{equation}
\end{lemma}
\begin{proof}
 Suppose $\Gamma \in \cG_{\dis}^{\ext}$. Then we claim that 
  \begin{equation}
    \label{eq:minsize}
    |\Ext\Gamma| + \sum_{\V{\gamma}\in\Gamma}\norm{\V{\gamma}} \geq  2n^{d-1}.
  \end{equation}
  To see this, note that 
  \begin{equation}
    \label{eq:minsize1}
    |\Ext\Gamma| + \sum_{\V{\gamma}\in\Gamma}|\Int \V{\gamma}|= n^{d},
  \end{equation}
  which combined with Lemma~\ref{lem:iso} implies
  \begin{equation}
    |\Ext\Gamma| + \frac{n}{2} \sum_{\V{\gamma}\in\Gamma}\norm{\V{\gamma}} \ge n^d
  \end{equation}
  which implies~\eqref{eq:minsize} when $n\geq 2$.

  By Lemma~\ref{lemBCTsup}, if 
  $n$ is large enough,
  \begin{equation}
    \label{eq:disneg}
    \frac{Z_{\dis}}{Z_{\ord}} \le 2 \max_{\Gamma \in \cG_{\dis}^{\ext}} e^{- \frac{a_{\dis}}{2} |\Ext \Gamma|} \prod_{\V{\gamma} \in \Gamma} e^{-\frac{c}{2} \beta \| \V{\gamma} \| } .
  \end{equation}
  Set $b_{\dis} \bydef \min \{ a_{\dis}, c \beta \}>0$. By~\eqref{eq:minsize}, 
  \begin{equation}
    e^{- \frac{a_{\dis}}{2} |\Ext \Gamma|} \prod_{\V{\gamma} \in \Gamma}
    e^{-\frac{c}{2} \beta \| \V{\gamma} \| }  \le \exp(- b_{\dis} n^{d-1}) 
  \end{equation}
  for all $\Gamma \in \cG_{\dis}^{\ext}$. The lemma now follows from~\eqref{eq:disneg}.
\end{proof}

The next two lemmas will allow us to verify the Koteck\'{y}--Preiss
condition for the contour models defining $Z_{\dis}$ and $Z_{\ord}$ from the previous section.
\begin{lemma}[{\cite[Lemma~6.3]{borgs2012tight}}]
  \label{lemKestimates}
  If $q \ge q_0$ and $\beta=\beta_c$, then
  \begin{equation*}
    K_{\ord} (\V{\gamma}) \le e^{-c \beta \|\V{\gamma}\|}, \quad \text{and}\quad 
    K_{\dis} (\V{\gamma}) \le e^{-c \beta \| \V{\gamma} \|} \,,
  \end{equation*}
  for all $\V{\gamma}$ in $\cC_\ord$ and $\cC_{\dis}$, respectively.
\end{lemma}

\begin{lemma}[{\cite[Lemma~6.3]{borgs2012tight}}]
  \label{lemSuperEstimates}
  If $q \ge q_0$ and $\beta > \beta_c$, then
  \begin{equation*}
    K_{\ord} (\V{\gamma}) \le e^{-c \beta \| \V{\gamma} \|} \quad \text{for all} \quad \V{\gamma} \in \cC_{\ord} \,.
  \end{equation*}
\end{lemma}

In particular, since {$\beta > \beta_h= \frac{3 \log q}{4d}$},
then for sufficiently large $q$ the contour weights
$w_{\gamma}=K_{\ord}(\gamma)$ (for $\beta \ge \beta_c$) and
${w_{\gamma}=}K_{\dis}(\gamma)$ (for $\beta = \beta_c$) will
satisfy {condition~\eqref{eqPolymerKP2}. Condition~\eqref{eqPeierls2}
is satisfied with $b=1$ by the discussion in Section~\ref{sec:extern-cont-repr-1}.}

Next we will show that when $\beta > \beta_c$ and the disordered ground state is unstable, that regions with disordered boundary conditions `flip' quickly to ordered regions by way of a large contour; more precisely, the dominant contribution to $Z_{\dis}(\V \Lam)$ comes from collections of contours with small external volume.    

\newcommand{\extbs}{\cH}

For a region $\V \Lam$ and
$M >0$ we define
\begin{equation*}
  \extbs^{\text{flip}}_\dis(\V\Lambda,M) \bydef \{\Gamma\in \cG_{\dis}^{\ext}(\V\Lam)
  \mid \abs{\Ext \Gamma \cap \V\Lam} \leq M\},
\end{equation*}
and 
\begin{equation}
  \label{eq:Zflip}
  Z_{\dis}^{\text{flip}}(\V\Lam,M) \bydef 
  \sum_{\Gamma\in\extbs^{\text{flip}}_\dis(\V\Lambda,M)}e^{-e_{\dis}|\Ext\Gamma \cap \V\Lam|}
  \prod_{\V{\gamma}\in\Gamma}e^{-\kappa \norm{\V{\gamma}}}qZ_{\ord}(\Int\V{\gamma}). 
\end{equation}
Thus, c.f.~\eqref{eq:matchextdis}, $Z_{\dis}^{\text{flip}}(\V\Lam,M)$
is the contribution to $Z_{\dis}(\V\Lam)$ from contour configurations
with small exterior volume.

\begin{lemma}
  \label{lemSuperEstimates2}
  Suppose $q \ge q_0$ and $\beta > \beta_c$.  Then there exists
  $a_\dis >0$ so that the following holds for all $n \ge n_0$.
  Suppose $\V{\gamma} \in \cC_{\ord}$.  For any $\eps>0$, if
  \begin{equation}
    \label{eq:Mlb}
    M\geq 
    \frac{2}{a_{\dis}}(\kappa+3)\norm{\V{\gamma}}
  \end{equation}
  then $Z_{\dis}^{\text{flip}}(\Int \V{\gamma} ,M)$ is an $\eps$-relative
  approximation to $Z_{\dis}(\Int \V{\gamma} )$.
\end{lemma}
\begin{proof}
  Let $\V\Lam = \Int \V{\gamma}$. Note that the lemma is immediate if
  $\Int \V{\gamma}$ does not contain any contours. Let
  \begin{equation*}
    Z_{\dis}^{\text{err}}(\V\Lam) \bydef
    Z_{\dis}(\Lam) - Z_{\dis}^{\text{flip}}(\V\Lam,M) \,.
  \end{equation*} 
  To prove the lemma it suffices to show that
  \begin{equation}
    \label{eq:fliplem}
    0 \le Z_{\dis}^{\text{err}}(\V\Lam ) / Z_{\dis}^{\text{flip}}(\V\Lam ,M) \le
    \eps/2.
  \end{equation}
  The lower bound is immediate since $Z_{\dis}$ is a sum of
  non-negative terms and $Z^{\text{flip}}_{\dis}(\V\Lam,M)$ is at least
  one. Thus the proof
  of~\eqref{eq:fliplem} 
  has two parts: lower bounding $Z_{\dis}^{\text{flip}}(\V\Lam ,M)$ and
  upper bounding $Z_{\dis}^{\text{err}}(\V\Lam)$.  The combination of
  these bounds will prove~\eqref{eq:fliplem}.

  We begin with the lower bound on
  $Z_{\dis}^{\text{flip}}(\V{\Lam},M)$. Recall the
  definition~\eqref{eq:active1s} of $\cE_{\V{\gamma}}$. Let
  $\V{\gamma}' \in \cC_\dis (\V\Lam)$ be the contour obtained by
  thickening $\cE_{\V{\gamma}}$ and taking the boundary, i.e.,
  $\partial \V{{\cE_{\V{\gamma}}}}$.  Let $\Gamma =
  \{\V{\gamma}'\}$. Note that $\Ext \Gamma$ contains no vertices,
  because $\V\Lam$ is connected and all edges inside $\V\Lam$ are in
  $\cE_{\V{\gamma}}$.
  
  Next observe that $\norm{\V{\gamma}'}\leq \norm{\V{\gamma}}$. This is
  because by construction any edge contributing to $\norm{\V{\gamma}'}$
  must have one vertex outside of $\V\Lam$, and such an edge also
  contributes to $\norm{\V{\gamma}}$.  In particular,
  $\Gamma \in \extbs^{\text{flip}}_\dis(\Lam,M)$, and hence 
  \begin{align*}
    Z_{\dis}^{\text{flip}}(\V\Lambda,M) 
    &\geq
      e^{-e_{\dis}\abs{\Ext\Gamma \cap \V\Lam}} e^{-\kappa\norm{\V{\gamma} '}} q
      Z_{\ord}(\Int\V{\gamma}') \\ 
    &\geq e^{-\kappa\norm{\V{\gamma}}} q Z_{\ord}(\Int\V{\gamma}') \\
    &\geq
      e^{-(\kappa+1)\norm{\V{\gamma}}} qe^{-(f + \eps_n)\abs{\Int \V{\gamma}'}} \\
    &\ge \frac{1}{2}e^{-(\kappa+1)\norm{\V{\gamma}}} qe^{-f \abs{\V\Lam}}\, ,
  \end{align*}
  where $\eps_n = 2 e^{-c \beta n}$ as above and $f$ is the constant
  from Lemma~\ref{lemBCTsup}. The second inequality used that $\Ext
  \Gamma$ contains no vertices. The second-to-last
  inequality follows from Lemma~\ref{lemBCTsup}, and
  the last inequality follows since (i) $|\Int \V{\gamma} | = | \Int \V{\gamma}'|$
  and (ii) for $n$ large enough we have $e^{\eps _n | \Int \V{\gamma}|} \le 2$
  for all $\V{\gamma} \in \cC$.

  Next we prove an upper bound on
  $Z_{\dis}^{\text{flip}}(\V{\Lam},M)$. 
  In fact, the upper bound is essentially
  contained in~\cite[Appendices~A.2 and~A.3]{borgs2012tight}, and we
  explain it here.  Some further notation will be helpful. Let
  $a_{\dis}>0$ be the constant from Lemma~\ref{lemBCTsup}.  We call a
  contour $\V{\gamma} \in \cC_{\dis}$ `small' if
  $\text{diam}(\V{\gamma}) \le \frac{c \beta}{a_{\dis}}$ and `large'
  otherwise. Here $\text{diam}(\V{\gamma})$ denotes the diameter of
  $\V{\gamma}$, the maximum over $i=1,\dots, n$ of $\abs{I_{i}(\V{\gamma})}$,
  where $I_{i}(\V{\gamma}) = \{ k\in \Z/n\Z \mid \V{S}^{(i)}_{k}\cap \V{\gamma}
  \neq \emptyset\}$, where $\V{S}^{i}_{k}$ is the set
  $\{\V{x}\in \ctor \mid \V{x}_{i}=k\}$. See~\cite[p.22]{borgs2012tight}.

 For a region $\V\Lam'$, let
 \begin{align*}
   \cG_{\dis}^{\ext,\text{small}}(\V\Lam') 
   &\bydef \{ \Gamma \in \cG_{\dis}^\ext(\V\Lam') | \V{\gamma}' \text { is
     small } \forall \V{\gamma} ' \in \Gamma \}, \\ 
   \cG_{\dis}^{\ext,\text{large}}(\V\Lam') 
   &\bydef \{ \Gamma \in \cG_{\dis}^\ext(\V\Lam') | \V{\gamma}' \text { is
     large } \forall \V{\gamma} ' \in \Gamma \} ,  
 \end{align*}
 and
\begin{align*}
  Z_{\dis}^{\text{small}} (\V\Lam ') 
  &\bydef  \sum_{\Gamma \in \cG_{\dis,\text{small}}^\ext(\V\Lam')}  
    e^{-e_{\dis}|\Ext\Gamma \cap \V\Lam'|} \prod_{\V{\gamma}\in\Gamma} e^{-\kappa
    \norm{\V{\gamma}}}qZ_{\ord}(\Int\V{\gamma}) \\                       
  &=e^{- e_{\dis} |\V\Lam'|} \sum_{\Gamma \in
    \cG_{\dis}^{\ext, \text{small}}(\V\Lam')} 
    \prod_{\V{\gamma} ' \in \Gamma} K_{\dis} (\V{\gamma}') .
\end{align*}

Moreover, let
\begin{align*}
  \extbs^{\text{err}}_{\dis}(\V\Lam) 
  &\bydef  \{\Gamma\in \cG_{\dis}^{\ext}(\V\Lam) \mid \abs{\Ext \Gamma \cap
    \V\Lam} > M\}, \qquad \text{and} \\ 
  \extbs^{\text{err}, \text{large}}_{\dis}(\V\Lam) 
  &\bydef 
    \{\Gamma\in \cG_{\dis}^{\ext,\text{large}}(\V\Lam) \mid
    \abs{\Ext \Gamma \cap \V\Lam} > M\} . 
\end{align*}

Following the  proof of \cite[Lemma A.1]{borgs2012tight}, we have that 
\begin{align*}
  Z_{\dis}^{\text{err}}(\V\Lam,M)
  &=  \sum_{\Gamma \in   \extbs^{\text{err}}_{\dis}(\V\Lam)}
    e^{-e_{\dis} |\Ext \Gamma \cap \V\Lam|}   \prod_{\V{\gamma}' \in \Gamma}
    e^{-\kappa \|\V{\gamma}'\|} q Z_{\ord} (\Int \V{\gamma} ') 
   \\                             
  &\le \sum_{\Gamma \in  \extbs^{\text{err},
    \text{large}}_{\dis}(\V\Lam) }  Z_{\dis}^{\text{small}}(\Ext \Gamma \cap \V\Lam)
    \prod_{\V{\gamma}' \in \Gamma} q e^{-\kappa \| \V{\gamma}'\| }
    Z_{\ord}(\Int \V{\gamma}')  
  \\
  &\le e^{(\eps_n-f
    ) |\V\Lam| + \|\V{\gamma}\|} e^{-\frac{a_{\dis}}{2} M} \sum_{\Gamma \in
    \extbs^{\text{err}, \text{large}}_{\dis}(\V{\Lam})}
    e^{-\frac{a_{\dis}}{2} |\Ext \Gamma \cap \V\Lam|}  \prod_{\V{\gamma} ' \in \Gamma}
    e^{-(\frac{\beta}{8} -3) \|\V{\gamma} ' \|} 
  \\
  &\le  2e^{-f
    |\Lam| + 2\|\V{\gamma}\|} e^{-\frac{a_{\dis}}{2} M} \, .
\end{align*}
The first inequality follows since for each
$\Gamma \in \extbs^{\text{err}}_{\dis}(\V\Lam)$, the set of large
contours in $\Gamma$ appear in
$\extbs^{\text{err}, \text{large}}_{\dis}(\V\Lam) $.  The second
inequality follows from the proof of~\cite[Lemma A.1]{borgs2012tight};
as above we are using that $f=f_{\ord}$ when $\beta>\beta_{c}$.  The
last inequality follows from~\cite[(A.12)]{borgs2012tight} and the
fact that $e^{\eps_n |\V\Lam|} \le 2$ for large enough $n$.

We can now conclude and prove~\eqref{eq:fliplem}: putting the bounds together and
using~\eqref{eq:Mlb} we get
\begin{equation*}
\frac{Z_{\dis}^{\text{err}}(\V\Lam )}{ Z_{\dis}^{\text{flip}}(\V\Lam ,M) }\le  4 q^{{-1}} e^{(\kappa +3) \| \V{\gamma} \| -\frac{a_{\dis}}{2} M} \le  \eps/2.\qedhere
\end{equation*}
\end{proof}

We conclude this section with an enumerative lemma concerning
$\cH^{\text{flip}}_{\dis}$.

\begin{prop}
  \label{prop:vacant}
  There is an algorithm that given $\V{\gamma}\in \cC_{\ord}$ and
  $M\in\N$ outputs $\cH^{\text{flip}}_{\dis}(\Int \V{\gamma},M)$ in time
  $\norm{\V{\gamma}}e^{O(\norm{\V{\gamma}}+M)}$.
\end{prop}
\begin{proof}
  This follows from a variation on the proof of
  Proposition~\ref{prop:enum}.  To determine
  $\cH^{\text{flip}}_{\dis}(\Int \V{\gamma})$ we will consider $\V{\gamma}$
  to be a contour in a torus of side-length
  $\norm{\V{\gamma}}\wedge n$; this torus has volume polynomial in
  $\norm{\V{\gamma}}$. 

  $\cH^{\text{flip}}_{\dis}(\Int\V{\gamma})$ is the set of mutually external
  contour configurations $\Gamma\setminus \V{\gamma}$ obtained as $F$
  ranges over the possibilities listed in Lemma~\ref{lem:edge-ord}. As
  in Lemma~\ref{lem:ord-edge-con} we can determine $E'\cup F$ by
  considering it as the complement of $1$-connected set of edges
  $A=A'\sqcup B$, where $A'$ is the set of edges that intersect
  $\V{\gamma}$. For any choice of such an $A$,
  $\Ext\Contour\cap\tor$ is of size at least $O(\abs{B})$, so to
  determine $\cH^{\text{flip}}_{\dis}(\Int \V{\gamma}, M)$ it is enough to
  consider all possible sets $B$ of size at most $M$. The claim now
  follows by arguing as in the proof of Proposition~\ref{prop:enum}.
\end{proof}

\section{Approximate counting algorithms}
\label{sec:count}

This section describes our approximate counting algorithms for $\beta>\beta_{h}$.  The algorithms differ depending on whether $\beta=\beta_c$, $\beta>\beta_c$, or $\beta_{h}<\beta<\beta_{c}$.  Recall that for $\ell \in \{ \dis, \ord\}$, $Z_{\ell}(\V{\Lam})$ was defined for all regions  $\V{\Lam}$
in~\eqref{eq:matchextord}--\eqref{eq:matchextdis}. The heart of this section is the following lemma.

\begin{lemma}
\label{lemOrdDiscompute}
For $d\geq 2$ and $q \ge q_0$ the following hold.
\begin{enumerate}
\item If $\beta = \beta_c$ there is an FPTAS to approximate $Z_{\ord}(\V{\Lam})$ and $Z_{\dis}(\V{\Lam})$. 
\item If $\beta> \beta_c$ there is an FPTAS to approximate  $Z_{\ord}(\V{\Lam})$.  
\item If $\beta_h < \beta < \beta_c$ there is an FPTAS to approximate $Z_{\dis}(\V{\Lam})$. 
\end{enumerate}
In each case the FPTAS applies to any region  
$\V{\Lam}$, with running time polynomial in $| \V \Lam|$, the number of vertices of $\tor$ in $\V{\Lam}$.
\end{lemma}

Sections~\ref{sec:lbc} and~\ref{sec:lbgc} prove the first two cases of Lemma~\ref{lemOrdDiscompute}. The case $\beta_{h}<\beta<\beta_{c}$ is very similar to $\beta>\beta_{c}$, and we defer the details to Appendix~\ref{sec:HT}.  In Section~\ref{secZTogether} we show how these results, together with a result 
from~\cite{borgs2012tight}, suffice to give an FPRAS for $Z$ on the torus.

\subsection{Proof of Lemma~\ref{lemOrdDiscompute} when $\beta = \beta_c$}
\label{sec:lbc} 

We begin by defining a useful variant of the truncated cluster expansion for $Z_{\ord}(\V{\Lam})$ and $Z_{\dis}(\V{\Lam})$. Let $K$ be a function from contours to positive real numbers. For $\ell\in\{\ord,\dis\}$ define

\begin{equation*}
  T_{\ell,m}(\V{\Lam},K) \bydef \sum_{\substack{\Gamma \in \cG^c_{\ell}(\V{\Lam}) \\ \|\Gamma \|  <  m }} \phi(\Gamma) \prod_{\V{\gamma} \in \Gamma} K(\V{\gamma}).
\end{equation*}
so that by~\eqref{eqZordPoly} and~\eqref{eqZdisPoly} $Z_{\ell}(\Lam) = \exp(-e_{\ell}|\Lam|) T_{\ell,\infty}(\V{\Lam},K_{\ell})$ provided the cluster expansion for the polymer models converge.

Recall that the level of a contour was defined in Definition~\ref{defLevel}, and that $|\V{\Lam}|_{\htors}$ was defined immediately prior to Lemma~\ref{lem:polygraphsize}.
\begin{lemma}
\label{lemBcInductive}
Suppose $d\geq 2$, $q \ge q_0$ and $\beta = \beta_c$.  Given 
$\V{\Lam}$ with $|\V{\Lam}|_{\htors} = N$, and an error parameter $\eps > 0$, let $m=\log(8N^2/\eps)/3$.  Inductively (by level) define weights $\tilde K_{\ord}(\V{\gamma})$ and $\tilde K_{\dis}(\V{\gamma})$ for all contours $\V{\gamma}$  in $\cC_{\ord}(\V{\Lam})$ and  $\cC_{\dis}(\V{\Lam})$ with size $\norm{\V{\gamma}}\leq m$ by:
\begin{enumerate}
\item If $\V{\gamma}$ is thin, then set
\begin{equation*}
\tilde K_\ord(\V{\gamma}) =  e^{- \kappa \| \V{\gamma} \| - (e_{\dis} - e_{\ord}) |\Int \V{\gamma}|}, \quad 
\tilde K_\dis(\V{\gamma}) =  q e^{- \kappa \| \V{\gamma} \| - (e_{\ord} - e_{\dis}) |\Int \V{\gamma}|}. 
\end{equation*}
\item If $\V{\gamma}$ is not thin, then set
\begin{align*}
\tilde K_{\ord} (\V{\gamma}) &=   e^{-\kappa \| \V{\gamma} \| - (e_{\dis} - e_{\ord}) |\Int \V{\gamma}|} \exp \left[ T_{m, \dis}(\Int \V{\gamma}, \tilde K) - T_{m, \ord}(\Int \V{\gamma}, \tilde K)  \right ],  \\
\tilde K_{\dis} (\V{\gamma}) &=qe^{-\kappa \| \V{\gamma} \| - (e_{\ord} - e_{\dis}) |\Int \V{\gamma}|} \exp \left[ T_{m, \ord}(\Int \V{\gamma}, \tilde K) - T_{m, \dis}(\Int \V{\gamma}, \tilde K)  \right ] \,.
\end{align*}
\end{enumerate}

Then for $N$ sufficiently large $e^{-e_{\ell}|\Lam|}\exp ( T_{\ell,m}(\V{\Lam},\tilde K_{\ell}))$ is an $\eps$-relative approximation to $Z_{\ell}(\V{\Lam})$ for $\ell\in\{\ord,\dis\}$.
\end{lemma}
\begin{proof}
Suppose $\ell \in  \{ \text{dis}, \text{ord} \}$.  First note that the inductive definition of the weights $\tilde K_{\ell}(\V{\gamma})$ makes sense:  to compute $\tilde K_{\ell}(\V{\gamma})$  for a contour $\V{\gamma}$ of level $t+1$ only requires knowing $\tilde K_{\ell}(\V{\gamma}')$ for contours $\V{\gamma}'$ of level $t$ and smaller. 

Since $\beta = \beta_c$ and $q \ge q_0$, Lemma~\ref{lemKestimates} tells us that 
\begin{equation}
\label{eqTaubound}
K_\ell(\V{\gamma}) \le e^{-c \beta \| \V{\gamma} \|}
\end{equation}
for $\ell \in \{ \text{dis}, \text{ord} \}$ and for all $\V{\gamma} \in \cC_{\ell}(\Lam)$.  If $q_0$ is large enough then \eqref{eqTaubound} implies condition~\eqref{eqPolymerKP2} holds since $\beta_c$ grows like $\log q$ by~\eqref{eq:betac}. Thus by Section~\ref{sec:extern-cont-repr-1} the hypotheses of Lemma~\ref{KPthm} are satisfied and the cluster expansion for $Z_{\ell}(\V{\Lam})$ converges for $\ell\in \{\ord,\dis\}$. 

Now let $\eps' = \eps/N$, so that $m = \log (8N/\eps')/3$. We will apply Lemma~\ref{lemPolymerApprox} with $v(\V{\gamma}) = | \Int \V{\gamma}|$. This is a valid choice of $v(\V{\gamma})$ by Lemma~\ref{lem:iso}.  Lemma~\ref{lemPolymerApprox} says that 
\begin{equation*} 
e^{- e_{\ord} |\Lam|} \exp \left( T_{\ord,m}(\V{\Lam}, \tilde K _{\ord}) \right) 
\quad \text{and} \quad
e^{- e_{\dis} |\Lam|} \exp \left( T_{\dis,m}(\V{\Lam}, \tilde K_{\dis})  \right)
\end{equation*}
are $\eps$-relative approximations to $Z_{\ord}(\V{\Lam})$ and $Z_{\dis}(\V{\Lam})$ if for all $\V{\gamma} \in \cC_{\ell}(\V{\Lam})$ of size at most $m$,  $\tilde K_{\ell} (\V{\gamma})$ is an $\eps' | \Int \V{\gamma}|$-relative approximation to $K_{\ell}(\V{\gamma})$.  We will prove this by induction on the level of $\V{\gamma}$.  

For a thin contour,  
$\tilde K_{\ell} (\V{\gamma}) = K_{\ell}(\V{\gamma})$. Now suppose that for all contours $\V{\gamma}$ of level at most $t$ and size at most $m$, $\tilde K_{\ell} (\V{\gamma})$ is an $\eps' | \Int \V{\gamma}|$-relative approximation of $K_{\ell}(\V{\gamma})$.  Consider a contour $\V{\gamma}$ of level $t+1$ and size at most $m$.  Then all contours $\V{\gamma}'$ that appear in the expansions
\begin{align*}
T_{m, \dis}(\Int \V{\gamma}, \tilde K_{\dis})  \quad \text{and} \quad T_{m, \ord}(\Int \V{\gamma}, \tilde K_{\ord}) 
\end{align*}
 are of level at most $t$ and size at most $m$, and so for each such $\V{\gamma}'$, by the inductive hypothesis  $\tilde K_\ell(\V{\gamma}')$ is an $\eps' |\Int \V{\gamma} '|$-relative approximation to $K_\ell(\V{\gamma}')$.  Then by Lemma~\ref{lemPolymerApprox}, we have that
 \begin{align*}
e^{ - (e_{\dis} - e_{\ord}) |\Int \V{\gamma}|}    \exp \left[ T_{m, \dis}(\Int \V{\gamma}, \tilde K_{\dis}) - T_{m, \ord}(\Int \V{\gamma}, \tilde K_{\ord})  \right ] 
\end{align*}
is an $| \Int \V{\gamma}| \eps'$-relative approximation to   $\frac{ Z_{\dis} (\Int \V{\gamma}) }{Z_{\ord}(\Int \V{\gamma})   }  $ (and likewise for $\dis$ and $\ord$ swapped).  Multiplying by the prefactor $e^{-\kappa \|\V{\gamma} \|}$ for $\ord$ and by $q e^{-\kappa \|\V{\gamma} \|}$ for $\dis$  shows that $\tilde K_{\ell}(\V{\gamma})$ is an $\eps' | \Int \V{\gamma}|$-relative approximation to $K_{\ell} (\V{\gamma})$ as desired.
\end{proof}

With this, we can prove the $\beta = \beta_{c}$ case of Lemma~\ref{lemOrdDiscompute}.
\begin{proof}[Proof of Lemma~\ref{lemOrdDiscompute} when $\beta=\beta_{c}$]

  Let $N = |\V{\Lam}|_{\htors}$ and let $m = \log ( 8N^2/\eps)/3$. We need to show that the expansion $T_{\ell,m}(\V{\Lam}, \tilde K_{\ell})$ and the weights $\tilde K_{\ell}(\V{\gamma})$ for all $\V{\gamma}$ of size at most $m$ in $\cC_{\ell}(\V{\Lam})$ can be computed in time polynomial in $N$ and $1/\eps$ for $\ell \in \{ \dis, \ord \}$.  We can list the sets of contours in $\cC_{\ord} (\V{\Lam})$ and $\cC_{\dis} (\V{\Lam})$ of size at most $m$, together with their labels and levels, in time $O(N \exp(O(m))$ by Proposition~\ref{prop:enum}.  Since $m = \log ( 8N^2/\eps)/3$, $O(N\exp(O(m))$ is polynomial in $N$ and $1/\eps$. The number $N$ itself is polynomial in $|\Lam|$ by Lemma~\ref{lem:polygraphsize}.

  To prove the lemma we must compute the weights $\tilde K_{\ell}(\V{\gamma})$ and the truncated cluster expansions $T_{m, \ell}(\Int \V{\gamma}, \tilde K_{\ell})$ 
  for each contour in the list.  We do this inductively by level. For level zero contours $\tilde K_{\ell}(\V{\gamma})=K_{\ell}(\V{\gamma})$ only depends on $\norm{\V{\gamma}}$ and $\abs{\Int\V{\gamma}}$, so $\tilde K_{\ell}(\V{\gamma})$ can be computed in time $O(\norm{\V{\gamma}}^{3})$ by computing these quantities by using Lemma~\ref{lem:findext}. We then continue inductively; each $\tilde K_{\ell}(\V{\gamma})$ can be computed efficiently since the truncated cluster expansions can be computed in time polynomial in $N$ and $1/\eps$ using Lemma~\ref{lemPolyModelCount}.
\end{proof}

\subsection{Proof of Lemma~\ref{lemOrdDiscompute} when $\beta > \beta_c$}
\label{sec:lbgc}

When $\beta > \beta_c(q,d)$ the ordered ground state is stable, but the disordered state is unstable. For a definition of stability of ground states, see, e.g.,~\cite{borgs1989unified}; the upshot for this paper is that we cannot use the cluster expansion to approximate $Z_{\dis}(\V{\Lam})$ for a region $\V{\Lam}$.

To deal with this complication we will appeal to Lemma~\ref{lemSuperEstimates2}.  In words, this lemma says that for $\beta> \beta_c$, a typical contour configuration in a region with disordered boundary conditions will have very few external vertices.   
We will exploit this fact to enumerate all sets of typical external contours in the region. This is possible since the number of external vertices is small. Once we have fixed a set of external contours  
we are back to the task of approximating partition functions with ordered boundary conditions.

We now make the preceding discussion precise. Given $K \colon \cC_{\ord}(\V{\Lam}) \to [0,\infty)$ and $M>0$, define
\begin{equation*}
\Xi^M_{\dis}(\V{\Lam}, K) \bydef e^{e_{\dis}|\Lam|}\sum_ {\Gamma \in   \extbs^{\text{flip}}_{\dis}(\V{\Lambda},M)}  e^{-e_{\dis} | \Ext \Gamma |}  \prod_{\V{\gamma} \in \Gamma} e^{-\kappa \| \V{\gamma} \| }  q \exp \left[ T_{m, \ord}(\Int \V{\gamma},  K)    \right]. 
\end{equation*}

\begin{lemma}
\label{lemSuperCritApprox}
Suppose $d\geq 2$, $q \ge q_0$ and $\beta > \beta_c$.  Let $\V{\Lam}$ be a region with $|\V{\Lam}|_{\htors}=N$, fix $\eps > 0$, and let $m=\log(8N^2/\eps)/3$.    Inductively (by level) define $\tilde K_{\ord}(\V{\gamma})$ for $\V{\gamma}\in\cC_{\ord}(\V{\Lam})$ with size $\norm{\V{\gamma}}$ at most $m$ by
\begin{enumerate}
\item If $\V{\gamma}$ is thin, then 
\begin{align*}
\tilde K_\ord(\V{\gamma}) &=  e^{- \kappa \| \V{\gamma} \| - (e_{\dis} - e_{\ord}) |\Int \V{\gamma}|} \, .
\end{align*}
\item If $\V{\gamma}$ is not thin, define 
\begin{align*}
\tilde K_{\ord} (\V{\gamma}) &=   e^{-\kappa \| \V{\gamma} \| - (e_{\dis} - e_{\ord}) |\Int \V{\gamma}|} \exp \left[  - T_{m, \ord}(\Int \V{\gamma}, \tilde K)  \right ]  \Xi^M_{\dis}(\Int \V{\gamma}, \tilde K_{\ord})    \, ,
\end{align*}
with $M= \frac{2}{a_{\dis}} \left( \log ( \frac{32q}{\eps'}) + (\kappa+3) m\right)$.  
\end{enumerate}

Then for all $N$ large enough, $e^{- e_{\ord} |\V{\Lam}|} \exp \left( T_{\ord,m}(\V{\Lam}, \tilde K _{\ord})  \right)$ is an $\eps$-relative approximation to $Z_{\ord}(\V{\Lam})$\,.
\end{lemma}
\begin{proof}
Let $\eps ' = \eps /N$ so that $m = \log (8N /\eps')/3$. 

If $q_{0}$ is large enough then we have $K_{\ord} (\V{\gamma}) \le e^{-c \beta \| \V{\gamma} \|} $ by Lemma~\ref{lemSuperEstimates} since $\beta > \beta_c$.  
This {along with~\eqref{eq:betac}} implies condition~\eqref{eqPolymerKP2} holds for ordered contours, and thus by Section~\ref{sec:extern-cont-repr-1} the hypotheses of Lemma~\ref{KPthm} are satisfied and the cluster expansion for $Z_{\ord}(\V{\Lam})$ converges. Applying Lemma~\ref{lemPolymerApprox} with $v(\V{\gamma}) = |\Int \V{\gamma}|$ then tells us that 
\begin{equation*}
  e^{- e_{\ord} |\V{\Lam}|} \exp \left( T_{\ord,m}(\V{\Lam}, \tilde K _{\ord})  \right) 
\end{equation*}
is an $\eps$-relative approximation to $Z_{\ord}(\V{\Lam})$ if for all $\V{\gamma} \in \cC_{\ord}(\Lam)$ of size at most $m$, $\tilde K_{\ord}(\V{\gamma})$ is an $\eps' | \Int \V{\gamma}|$-relative approximation to $K_{\ord}(\V{\gamma})$.  We will prove this is the case by induction. The base case of the induction (thin contours) holds since $\tilde K_{\ord}(\V{\gamma}) = K_{\ord}(\V{\gamma})$. Now suppose that the statement holds for all contours of level at most $t$ and size at most $m$, and consider a contour $\V{\gamma}$ of level $t+1$ and size at most $m$.

The inductive hypothesis and Lemma~\ref{lemPolymerApprox} imply that 
\begin{equation*}
 e^{- e_{\ord} |\V{\Lam}|}\exp \left[  T_{m, \ord}(\Int \V{\gamma}, \tilde K)  \right ]
\end{equation*}
is an $\eps' | \Int \V{\gamma}|/2$-relative approximation to  $Z_{\ord}(\Int \V{\gamma})$, and so it suffices to show that $e^{-e_{dis}|\V{\Lam}|}\Xi^M_{\dis}(\Int \V{\gamma}, \tilde K_{\ord}) $ is  an $\eps' |\Int \V{\gamma}|/2$-relative approximation to  $Z_{\dis}(\Int \V{\gamma})$.

By Lemma~\ref{lemSuperEstimates2}, $Z_{\dis}^{\text{flip}}(\Int \V{\gamma},M) $ is an $\eps'/4$-relative approximation to $Z_{\dis}(\Int \V{\gamma}) $ for $M= \frac{2}{a_{\dis}} \left( \log ( \frac{32q}{\eps'}) + (\kappa+3) m\right)$, and so  it suffices to show that $e^{-e_{dis}|\Lam|}\Xi^M_{\dis}(\Int \V{\gamma}, \tilde K_{\ord}) $ is an $\eps' |\Int \V{\gamma}| /4$-relative approximation to $Z_{\dis}^{\text{flip}}(\Int \V{\gamma} ,M) $.  We will accomplish this by showing, for each $\Gamma \in   \extbs^{\text{flip}}(\Int \V{\gamma},M)$, that 
\begin{equation*}
e^{-e_{\dis} | \Ext \Gamma |}  \prod_{\V{\gamma}' \in \Gamma} e^{-\kappa \| \V{\gamma}' \| }  q \exp \left[ T_{m, \ord}(\Int \V{\gamma}',  \tilde K)    \right]
\end{equation*}
is an $\eps' |\Int \V{\gamma}| /4$-relative approximation to 
\begin{equation*}
e^{- e_{\dis} | \Ext \Gamma |} \prod_{\V{\gamma} ' \in \Gamma} e^{-\kappa \| \V{\gamma} ' \|} q Z_{\ord} (\Int \V{\gamma} ')
\end{equation*}
and then summing over $\Contour$.  The prefactors are identical, and so it comes down to comparing $ \prod_{\V{\gamma}' \in \Gamma}  \exp \left[ T_{m, \ord}(\Int \V{\gamma}',  \tilde K)    \right] $ to  $\prod_{\V{\gamma}' \in \Gamma}Z_{\ord} (\Int \V{\gamma} ')$.  Since the contours in $\Contour$ are mutually external,
\begin{equation*}
\sum _{\V{\gamma}' \in \Gamma} | \Int \V{\gamma}'| \le |\Int \V{\gamma}| \, ,
\end{equation*}
and hence it suffices to show that for each $\V{\gamma}'$, $ \exp \left[ T_{m, \ord}(\Int \V{\gamma}',  \tilde K)    \right]$ is an $\eps' | \Int \V{\gamma}'|/4$-relative approximation to $Z_{\ord} (\Int \V{\gamma} ')$.  This follows from Lemma~\ref{lemPolymerApprox} since $m =\log (8N /\eps')/3 $ and by induction we have that  $\tilde K_{\ord} (\V{\gamma}'')$ is an $\eps'|\Int \V{\gamma}''|$-relative approximation to $ K_{\ord} (\V{\gamma}'')$ for all contours $\V{\gamma}''$ that contribute to  $T_{m, \ord}(\Int \V{\gamma}',  \tilde K) $.
\end{proof}

With this, we can prove the $\beta>\beta_c$ case of Lemma~\ref{lemOrdDiscompute}.
\begin{proof}[Proof of Lemma~\ref{lemOrdDiscompute} when $\beta>\beta_c$]

Given Lemma~\ref{lemSuperCritApprox}, we need to show that we can compute $\tilde K _{\ord}(\V{\gamma})$ for all $\V{\gamma}$ of size at most $m = \log(8 N^2/\eps)/3  $ in time polynomial in $N$ and $1/\eps$.  The proof of this is the same as the proof of the $\beta=\beta_c$ case of the lemma except that now we have to account for the computation of  $\Xi^M_{\dis}(\Int \V{\gamma}, \tilde K)$ for 
all $\V{\gamma} \in \cC_{\dis}(\V{\Lam})$ of size at most $m$, with $M= \frac{2}{a_{\dis}} \left( \log ( \frac{32q}{\eps'}) + (\kappa+3) m\right)$.

For a given $\Gamma \in  \extbs^{\text{flip}}_{\dis}(\Int \V{\gamma},M)$, the computation of  
\begin{equation}
e^{-e_{\dis} | \Ext \Gamma |}  \prod_{\V{\gamma}' \in \Gamma} e^{-\kappa \| \V{\gamma}' \| }  q \exp \left[ T_{m, \ord}(\Int \V{\gamma}',  \tilde K) \right ]
\end{equation}
can be done in time polynomial in $N$ and $1/\eps$ since it just involves computing the truncated cluster expansions $T_{m,\ord} (\Int \V{\gamma}', \tilde K)$ for at most $m^2$ contours $\V{\gamma}'$, and since we compute $\tilde K_{\ord}(\V{\gamma}')$ in order of the level of $\V{\gamma}'$, we will have already computed all the weight functions needed in the expansion.

To conclude, note the set $\extbs^{\text{flip}}_{\dis}(\Int \V{\gamma},M)$ can be enumerated in polynomial time by Proposition~\ref{prop:vacant} since both $\|\V{\gamma} \|$ and $M$ are $O(\log(N^2/\eps))$.  Since $N$ is polynomial in $|\V{\Lam}|$ by Lemma~\ref{lem:polygraphsize}, the proof is complete.
\end{proof}

Note that Lemma~\ref{lemSuperCritApprox} used the value of $a_{\dis}>0$ to determine the value of $M$ in the definitions of the weights $\tilde K$. It is desirable to avoid using $a_{\dis}$ as an input of the algorithm, and hence we close this section with a lemma that shows how to bound $M$ without knowing $a_{\dis}$ precisely.

\begin{lemma}
  \label{lem:adis}
  Suppose $d\geq 2$, $q\geq q_{0}$, and $\beta>\beta_{c}$. There is an $O(1)$-time algorithm to determine a constant $a^{\star}_{\dis}>0$ such that $a_{\dis}>a^{\star}_{\dis}$. The constants in the $O(1)$ term may depend on $q,\beta,d$.
\end{lemma}
\begin{proof}
  We follow the notation from~\cite[Appendix~A.1]{borgs2012tight}. As discussed below~\cite[(A.7)]{borgs2012tight}, we have $\abs{f_{\ell}-f^{(n)}_{\ell}}\leq\eps_{n}$ for $\ell\in \{\ord,\dis\}$, where $\eps_{n}=2e^{-c\beta n}$, where $n$ is the side-length of the torus $\tor$, and $f_{\ell}=\lim_{n\to\infty}f_{\ell}^{(n)}$. 

 Compute $f^{(n)}_{\ell}$ for $\ell\in\{\ord,\dis\}$ until $\abs{f^{(n)}_{\ord}-f^{(n)}_{\dis}}$ is at least $3\eps_{n}$. Let $n_{0}$ be the first such $n$ that is found. Then by the triangle inequality, $a_{\dis}$ is at least $a^{\star}_{\dis}=\eps_{n_{0}}$.

 Note that $n_{0}$ can be bounded above in terms of the value of $a_{\dis}=a_{\dis}(\beta,d,q)$ and $\eps_{n}$, so the above procedure terminates in a finite time (depending on $\beta,d,q$).
\end{proof}

\subsection{Proof of Theorem~\ref{PottsTorusCrit}}
\label{secZTogether}

To prove Theorem~\ref{PottsTorusCrit} we will need the following result from~\cite{borgs2012tight} about the mixing time of the Glauber dynamics. 
\begin{theorem}[{\cite[Theorem~1.1]{borgs2012tight}}]
\label{thmBCTglauber}
The mixing time of the Glauber dynamics for the $q$-state ferromagnetic Potts model satisfies
\begin{equation}
\tau_{q,\beta}(\tor) = e^{O( n^{d-1})} , 
\end{equation}
where the $O(\cdot)$ in the exponent hides constants that depend on $q, \beta$. 
\end{theorem}

We will use this result to give an approximation algorithm when the approximation parameter $\eps$ is extremely small.  The reason we are able to combine the Glauber dynamics with our contour-based algorithm to give an FPRAS is that~\cite{borgs2012tight} proves \emph{optimal} slow mixing results for the Glauber and Swendsen--Wang dynamics.  That is, up to a constant in the exponent, the upper bound of the mixing time of the Glauber dynamics (or Swendsen--Wang dynamics) is the inverse of the bound on $Z_\tunnel/Z$ from Lemma~\ref{lem:Z-split}.  Thus when $\eps$ is too small for the contour algorithms to work, the Glauber dynamics can take over.

\begin{proof}[Proof of Theorem~\ref{PottsTorusCrit}]
Let $N = n^d$ be the number of vertices of $\tor$. 
We will use a simple fact several times below:  if $\eps \in (0,1)$, $Z, Z^* > 0$, and $Z^*/Z < \eps/2$, then  $(Z-Z^*)$ is an $\eps$-relative approximation to $Z$.  

We first consider the case $\beta = \beta_c$.  To give an FPRAS for $Z=Z_{\tor}$ we consider two subcases. Let $c$ be the constant from Lemma~\ref{lem:Z-split}.

Suppose $\eps < 4e^{-c \beta n^{d-1}}$. Since $e^{O( n^{d-1})}$ is polynomial in $N$ and $1/\eps$, we can use Glauber dynamics to obtain an $\eps$-approximate sample in polynomial time.  By using simulated annealing (e.g.~\cite{vstefankovivc2009adaptive}) we can also approximate the partition function in time polynomial in $N$ and $1/\eps$. 

If $\eps \ge   4e^{-c \beta n^{d-1}}$, then by Lemma~\ref{lem:Z-split}, $Z_\rest = Z_{\dis} + Z_{\ord}$ is an $\eps/2$-relative approximation to $Z$, so it suffices to find an $\eps/4$-relative approximation to both $Z_{\dis} $ and $Z_{\ord}$.  This can be done in time polynomial in $N$ and $1/\eps$ by Lemma~\ref{lemOrdDiscompute}.

Next we consider the case $\beta > \beta_c$. Again there are two subcases. Let $c$ be the constant from Lemma~\ref{lem:Z-split} as before, and let $b_{\dis}$ be the constant from Lemma~\ref{lemSuperEstimates2a}.  If $\eps < 4e^{-c \beta n^{d-1}} + 4 e^{-b_{\dis} n^{d-1}}$, then again $e^{O( n^{d-1})}$ is polynomial in $N$ and $1/\eps$ and we can approximately count and sample by using the Glauber dynamics.

If $\eps \ge   4e^{-c \beta n^{d-1}} + 4 e^{-b_{\dis} n^{d-1}}$, then by Lemma~\ref{lem:Z-split} and Lemma~\ref{lemSuperEstimates2a}, $Z_{\ord}$ is an $\eps/2$-relative approximation to $Z$ and so it suffices to give an $\eps/2$-relative approximation to $Z_{\ord}$.   This can be done in time polynomial in $N$ and $1/\eps$ by Lemma~\ref{lemOrdDiscompute}. 

Lastly, consider $\beta<\beta_{c}$. The case $\beta\leq \beta_{h}$ was completed in Section~\ref{secPolymer}. The case $\beta_{h}<\beta<\beta_{c}$ is done exactly as the case $\beta>\beta_{c}$ with the roles of $\ord$ and $\dis$ reversed; see Appendix~\ref{sec:HT} for details.
\end{proof}

\begin{proof}[Proof of Theorem~\ref{PottsZd} for counting]
Let $\Lam \subset \Z^d$ be such that the induced subgraph  $G_{\Lam}$ is finite and  simply connected. By Proposition~\ref{prop:scregion}, we can construct an ordered contour $\V{\gamma}_{\ord}$ and a disordered contour $\V{\gamma}_{\dis}$ so that 
\begin{align*}
Z_{\dis} (\Int \V{\gamma}_{\ord}) = (1-p)^{-\frac{1}{2}\norm{\V{\gamma}_{\ord}}}Z^f_{\Lam}, 
  \qquad 
Z_\ord(\Int \V{\gamma}_{\dis}) =  q^{-1}p^{d\abs{\Int \V{\gamma}_{\ord}}-|E(\Lam)|}Z^w_{\Lam}  \,.
\end{align*}
The FPTAS for $Z^w_{\Lam}$ for $\beta \ge \beta_c$ then follows from Lemma~\ref{lemOrdDiscompute}, as does the FPTAS for $Z^f_{\Lam}$ for $\beta_h < \beta \le \beta_c$.  The case $\beta\leq \beta_{h}$ was covered in Section~\ref{secPolymer}. 
\end{proof}

\section{Sampling}
\label{sec:sample}

In this section we present efficient approximate sampling algorithms for the random cluster and Potts models when $\beta>\beta_{h}$. By the Edwards--Sokal coupling, see Appendix~\ref{sec:ES}, it suffices to obtain algorithms for the random cluster model. Describing the strategy, which is based on that of~\cite[Sections~5 and~6]{helmuth2018contours}, requires a few definitions.

Recall the definition~\eqref{eqRCdef} of the random cluster measure $\mu^{\text{RC}}$ on $\tor$. Thus $\mu^{\text{RC}}$ is a measure on subsets of edges $A\in\Omega$. Recalling the definitions~\eqref{eq:Z-split-ord} and~\eqref{eq:Z-split-dis} of the sets $\Omega_{\ord}$ and $\Omega_{\dis}$ of ordered and disordered edge configurations, we analogously define 
\begin{equation*}
  \mu_{\ell}(A) \bydef \frac{w(A)}{Z_{\ell}}, \quad A\in \Omega_{\ell}\, \text{ with } \ell\in \{\ord,\dis\}.
\end{equation*}

For a region $\V{\Lam}$, define measures $\nu_{\ell}^{\V{\Lam}}$ on the sets of external contours $\cG_{\ell}^{\ext}(\V{\Lam})$ as follows. 
\begin{align}
  \label{eq:nudis}
\nu_\dis^{\V{\Lam}}(\Gamma) &\bydef \frac{  e^{ - e_{\dis} |\Lam \cap \Ext \Gamma| }  \prod_{\V{\gamma} \in \Gamma} e^{-\kappa \| \V{\gamma} \|} qZ_{\ord}(\Int \V{\gamma})   }{  Z_{\dis} (\V{\Lam})  }  ,  \qquad  \Gamma \in \cG_\dis^\ext(\V{\Lam}) , \\
\nu_\ord^{\V{\Lam}}(\Gamma) &\bydef \frac{  e^{ - e_{\ord} |\Lam \cap \Ext \Gamma| }  \prod_{\V{\gamma} \in \Gamma} e^{-\kappa \| \V{\gamma} \|} Z_{\dis}(\Int \V{\gamma})   }{  Z_{\ord} (\V{\Lam})  }  ,  \qquad  \Gamma \in \cG_\ord^\ext(\V{\Lam})  , \label{eq:nuord}
\end{align}
where $|\Lam\cap \Ext\Gamma|$ is the number of vertices contained in the continuum set $\V{\Lam} \cap \Ext\Gamma$.

We now outline our strategy for approximately sampling from $\mu_{\ord}$ and $\mu_{\dis}$; a small modification will also apply to sampling from $\mu^{\text{RC}}$ on the torus. The key idea is that the inductive representations of the partition functions in~\eqref{eq:matchextord} and~\eqref{eq:matchextdis} yield a procedure to sample from $\mu_\dis$ and $\mu_\ord$ if we can sample from the measures $\nu_{\ell}^{\V{\Lam}}$ for $\ell\in\{\ord,\dis\}$ and for all regions $\V{\Lam}$. The procedure, which we call the \emph{inductive contour sampling algorithm}, is as follows. Consider $\mu_{\ord}$. To sample a set of compatible, matching contours with ordered external contours, we first sample $\Gamma$ from $\nu_{\ord}^{\ctor}$, then for each $\V{\gamma} \in \Gamma$ we sample from $\nu_\dis^{\Int \V{\gamma}}$ and repeat inductively until there are no interiors left to sample from.  The union of all contours sampled is a set of matching and compatible contours, and these contours are distributed as the restriction of~\eqref{eq:Z-mci} to contour configurations that arise from ordered edge configurations. This set of contours can then be mapped to an edge set via the bijection of Lemma~\ref{lem:rep}, and the distribution of this edge set is $\mu_\ord$. The procedure for sampling from $\mu_{\dis}$ is analogous. For a more detailed discussion of the validity of this algorithm, see~\cite[Section~5]{helmuth2018contours}.

By using the same procedure it is possible to efficiently approximately sample from $\mu_\ord$ and $\mu_\dis$ provided one can efficiently approximately sample from the external contour measures $\nu_\dis^{\V{\Lam}}$ and $\nu_\ord^{\V{\Lam}}$. Again, we refer to~\cite[Section~5]{helmuth2018contours} for further details. 

The next lemma is an essential input for developing efficient approximate samplers for $\nu_{\ell}^{\V{\Lam}}$ as it tells us we need only consider `small' contours. For $\ell\in \{\ord,\dis\}$ let $\nu_{\ell}^{\V{\Lam},m}$ be the probability measure defined as in~\eqref{eq:nudis}--~\eqref{eq:nuord}, but restricted to $\Gamma$ with $\norm{\Gamma}<m$. The normalization factor for $\nu_{\ell}^{\V{\Lam},m}$ is thus the contour partition function restricted to $\Gamma$ with $\norm{\Gamma}<m$.

\begin{lemma}
  \label{lem:sample-trunc}
  Suppose $d\geq 2$, $q\geq q_{0}$, and $\eps>0$. Then, letting
  $N=\abs{\V{\Lam}}_{\htors}$, for $m{\geq C'} 
  \log (N /\eps)$ {with $C'$ a large enough absolute constant,}
  \begin{enumerate}
  \item  If $\beta\geq \beta_{c}$, then $\|\nu_{\ord}^{\V{\Lam},m} -  \nu_{\ord}^{\V{\Lam}}  \|_{TV} < \eps$. 
  \item If $\beta_{h}<\beta \le \beta_{c}$, then $\| \nu_{\dis}^{\V{\Lam},m} -\nu_{\dis}^{\V{\Lam}}  \|_{TV} <\eps$. 
  \end{enumerate}
  for all regions 
  $\V{\Lam}$.\footnote{The constant $C'$ 
    depends only on the constants $c$ in the bounds on $K_{\ell}(\V{\gamma}) \leq \exp(-c\beta \norm{\V{\gamma}})$. These bounds are given by Lemma~\ref{lemKestimates} and Lemma~\ref{lemSuperEstimates} for $\beta\geq \beta_{c}$, and in Appendix~\ref{sec:HT} for $\beta_{h}<\beta<\beta_{c}$.}
\end{lemma}
\begin{proof}
  This follows from the convergence of the cluster expansion for $Z_{\ell}(\V{\Lam})$ for the specified choices of $\ell$ and $\beta$.  For details see, e.g.,~\cite[Proof of Lemma~13]{helmuth2018contours}. 
\end{proof}

\begin{lemma}
\label{lemOuterSample}
Suppose $d \ge 2$ and $q \ge q_0$.  Then
\begin{enumerate}
\item  For $\beta = \beta_c$, there are efficient sampling schemes for $\nu_\ord^{\V{\Lam}}$ and $\nu_\dis^{\V{\Lam}}$.

\item For $\beta > \beta_c$ there is an efficient sampling scheme for $\nu_\ord^{\V{\Lam}}$. 

\item For $\beta_h < \beta < \beta_c$ there is an efficient sampling scheme for $\nu_\dis^{\V{\Lam}}$.   
\end{enumerate} 
In each case these algorithms apply for all regions 
$\V{\Lam}$.  
\end{lemma}
\begin{proof}
First we consider $\beta=\beta_{c}$.  By Lemma~\ref{lemOrdDiscompute} there are efficient algorithms to approximate $Z_{\dis}(\V{\Lam})$ and $Z_{\ord}(\V{\Lam})$ for all regions  $\V{\Lam}$.  With this, we can apply the approximate sampling algorithms given in~\cite[Theorem 11 and Theorem 13]{helmuth2018contours}.  We summarize the algorithm here, assuming that we want to sample a collection of ordered contours (the disordered case is identical).

By Lemma~\ref{lem:sample-trunc} it is enough to obtain an $\eps$-approximate sample from $\nu_{\ell}^{\V{\Lam},m}$ with $m = O(\log (N / \eps))$. 
List all contours of size at most $m$ in $\cC_{\ord}(\V{\Lam})$, and call this collection $\cC$.   
Order the vertices of $\V{\Lam}$ arbitrarily as $v_1, \dots, v_N$.   We will form a random collection $\Gamma=\Gamma_N$ of mutually external ordered contours step by step.  Begin with $\Gamma_0 = \emptyset$.  At step $i$, let $\cC_i$  be the subset of contours $\V{\gamma}$ in $\cC$ such that (i) $v_i \in \Int \V{\gamma}$ (ii) $\V{\gamma}$ is external to $\Gamma_{i-1}$ and (iii) $\Int \V{\gamma} \cap \{v_{1},\dots, v_{i-1}\}=\emptyset$. 
We can efficiently approximate the conditional probability of each contour in $\cC_{i}$, or of adding no contour at step $i$, by using Lemma~\ref{lemOrdDiscompute} to approximate the relevant polymer partition functions. The result of this procedure is the desired approximate sampling algorithm.

Sampling from $\nu_\ord^{\V{\Lam}}$ for $\beta > \beta_c$ also follows from the algorithm described above since we have an FPTAS for computing $Z_{\ord}(\V{\Lam})$, and similarly for $\nu_{\dis}^{\V{\Lam}}$ when $\beta_{h}<\beta<\beta_{c}$.
\end{proof}

Our strategy for efficiently approximately sampling from $\mu_{\ord}$ and $\mu_{\dis}$ requires that we can also efficiently approximately sample from $\nu_{\dis}^{\V{\Lam}}$ for small regions $\V{\Lam}$ when $\beta>\beta_{c}$ (and likewise from $\nu_{\ord}^{\V{\Lam}}$ when $\beta<\beta_c$). We cannot use the cluster expansion for this task since the disordered (resp. ordered) ground state is unstable, and so instead our approach is based on the intuition from Lemma~\ref{lemSuperEstimates2} that a disordered region will quickly `flip' to being ordered when $\beta>\beta_{c}$.

\begin{lemma}
\label{lemOuterSampleU}
Suppose $d \ge 2$ and $q \ge q_0$.  Then
\begin{enumerate}
\item For $\beta > \beta_c$ there is an $\eps$-approximate sampling algorithm for $\nu_\dis^{\V{\Lam}}$ that runs in time polynomial in $1/\eps$ and exponential in $\| \partial \V{\Lam} \|$.

\item For $\beta_h < \beta < \beta_c$ there is an $\eps$-approximate sampling algorithm for $\nu_\ord^{\V{\Lam}}$ that runs in time polynomial in $1/\eps$ and exponential in $\| \partial \V{\Lam} \|$.  
\end{enumerate} 
In each case these algorithms apply for all regions $\V{\Lam}$.  
\end{lemma}
In our sampling algorithms we can allow exponential dependence on $\| \partial \V{\Lam} \|$ since by Lemma~\ref{lem:sample-trunc} we need only consider contours $\gamma$ with $\| \gamma \| = O( \log (N/\eps) )$.  
\begin{proof}[Proof of Lemma~\ref{lemOuterSampleU}]
  Consider the case $\beta>\beta_{c}$ and suppose $\V{\Lam} = \Int \V{\gamma}$. The lemma follows from Proposition~\ref{prop:vacant} and Lemma~\ref{lemSuperEstimates2}. More precisely, set $M$ according to Lemma~\ref{lemSuperEstimates2}, and then compute $\cH^{\text{flip}}_{\dis}(\Int \V{\gamma},M)$ by Proposition~\ref{prop:vacant}. As in the proof of Lemma~\ref{lemOrdDiscompute}, compute accurate approximations to the weight of each summand in $Z^{\text{flip}}_{\dis}(\Int\V{\gamma},M)$. These approximations determine the probabilities according to which we sample $\Gamma\in \cH^{\text{flip}}_{\dis}(\Int \V{\gamma},M)$. By Lemma~\ref{lemSuperEstimates2} the result is an $\eps$-approximation to $\nu^{\Int\V{\gamma}}_{\ord}$.

For $\beta_{h}<\beta<\beta_{c}$ the proof is essentially the same given the inputs discussed in Appendix~\ref{sec:HT}.
\end{proof}

\begin{proof}[Proof of Theorems~\ref{PottsTorusCrit} and~\ref{PottsZd}, sampling]
We first consider the sampling part of Theorem~\ref{PottsZd}, which follows similarly to the proof of the approximate counting algorithm given in the previous section. Given (i) $\Lambda\subset\Z^{d}$ such that $G_{\Lambda}$ is simply connected and (ii) a choice of wired or free boundary conditions, Proposition~\ref{prop:scregion} gives a contour $\V{\gamma}$ such that the partition function associated to $\Int \V{\gamma}$ is $Z_{\Lambda}^{w}$ or $Z_{\Lam}^{f}$ {up to an efficiently computable prefactor}. Thus if $\beta=\beta_{c}$ we can use Lemma~\ref{lemOuterSample} to implement the inductive contour algorithm, but using $\eps'$-approximations to $\nu^{\V{\Lam}}_{\ord}$ and $\nu^{\V{\Lam}}_{\dis}$ in place of the true measures. If $\eps'=\eps^{2}/(9N^{2})$ where $N=\abs{\V{\Lambda}}_{\htors}$, the result is an $\eps$-approximate sample by~\cite[Lemma~12]{helmuth2018contours}. Here we are using $N$ as a crude bound for the depth of the inductive contour algorithm.

If $\beta>\beta_{c}$, then Lemma~\ref{lem:sample-trunc} tells us that it suffices to sample from $\nu^{\V{\Lam},m}_{\ord}$ with $m = O(\log (N/\eps))$.  The consequence of this fact is that we can use the algorithm described above for $\beta=\beta_{c}$, as each call for an $\eps$-approximate sample of $\nu^{\V{\Lam}}_{\dis}$ takes time $\exp (O (\log N/\eps))$ by Lemma~\ref{lemOuterSampleU} since each contour is of size at most $O(\log(N/\eps))$. 
For $\beta_{h}<\beta<\beta_{c}$ an analogous argument applies with the roles of $\ord$ and $\dis$ reversed.

For Theorem~\ref{PottsTorusCrit} the situation is similar to what we
have just discussed, except for the fact that $\mu^{\text{RC}}$ is not
an ordered or a disordered measure: it includes configurations with
ordered and disordered external contours and includes the
configurations with interfaces.  If $\beta>\beta_{c}$, however, we
have {(see Lemmas~\ref{lem:Z-split} and~\ref{lemSuperEstimates2a})} $\| \mu^{\text{RC}} - \mu_{\ord}  \|_{TV} = \exp ( - \Omega(n^{d-1}))$, and hence if $\eps$ is not too small, we can sample from $\mu_{\ord}$ as above. \emph{Mutatis mutandis} the same argument applies for $\mu_{\dis}$ if $\beta_{h}<\beta<\beta_{c}$.  On the other hand if $\eps =  \exp ( - \Omega(n^{d-1}))$, then we can use the Glauber dynamics to sample efficiently by Theorem~\ref{thmBCTglauber}.

For $\beta=\beta_{c}$ the situation is slightly different as the probability of both the ordered and disordered configurations are both of constant order, while the probability of configurations with interfaces is still $\exp ( - \Omega(n^{d-1}))$. The solution is to use the approximate counting algorithm of Lemma~\ref{lemOrdDiscompute} to approximate the relative probabilities of $\Omega_\ord$ and $\Omega_{\dis}$ under $\mu^{\text{RC}}$ and then to sample from each using the procedure above. 
Again if $\eps =  \exp ( - \Omega(n^{d-1}))$ we can use the Glauber dynamics.

Note that our sampling algorithm will not return any configurations with interfaces if $\eps \ge  4e^{-c \beta n^{d-1}}$, but such configurations have probability smaller than $\eps$.  On the other hand, if $\eps<  4e^{-c \beta n^{d-1}}$, then running Glauber dynamics may indeed return a configuration with interfaces.
\end{proof}

\section{Conclusions}
\label{sec:Conc}

In this paper we have given efficient approximate counting and
sampling algorithms for the random cluster and $q$-state Potts models
on $\Z^d$ at all inverse temperatures $\beta\geq 0$, provided
$q\geq q_{0}(d)$ and $d\geq 2$. We believe the ideas of this paper
will, however, allow for approximate counting and sampling algorithms to be
developed for a much broader class of statistical mechanics
models. The necessary conditions for the development of algorithms for
a given model is that there are only finitely many ground states, and
that there is `sufficient $\tau$-functionality'. These are the
necessary ingredients for the implementation of Pirogov--Sinai theory,
see~\cite{borgs1989unified}. Our methods allow for the presence of
unstable ground states, a significant improvement compared to the
algorithms in~\cite{helmuth2018contours}.

Our results suggest that the algorithmic tasks of counting and
sampling may be performed efficiently for a fairly broad class of
statistical mechanics models with first-order phase transitions, but
we leave a fuller investigation of this for future work. A related
interesting questions is the existence of efficient algorithms for all
$\beta\geq \beta_{c}$ in the presence of a second-order phase transition; we
are not aware of any results in this direction with the exception of
the Ising model, i.e., the $q=2$ state Potts
model~\cite{jerrum1993polynomial,guo2018}. To conclude we list some
further open questions related to this paper.

\begin{enumerate}
\item Our algorithms are restricted to $q\geq q_{0}(d)$ with
  $q_{0}(d)$ {more than exponentially large in $d$.}
  Do efficient algorithms exist that avoid this constraint? Since the
  physical phenomena behind our results are believed to hold for
  $q\geq 3$ when $d\geq 3$, there is likely room for improvement.
  
\item On the torus, we obtained an FPRAS (as opposed to an FPTAS) for
  the partition function because of the estimate on $Z_{\tunnel}$ from
  Lemma~\ref{lem:Z-split}: the contribution of $Z_{\tunnel}$ cannot be
  ignored when $\eps\leq \exp(-\Omega(n^{d-1}))$.  Fortunately, it is exactly when $\eps$ is this small that the Glauber dynamics mix in time polynomial in $1/\eps$, but of course Markov Chain Monte Carlo is a randomized algorithm. 
   A method for
  systematically accounting for the interfaces that contribute to
  $Z_{\tunnel}$ would likely enable the development of an FPTAS. We
  leave this as an open problem.
\item Our algorithms have at least two other features that could be
  improved. The first is the running time: while our algorithms are
  polynomial time, the degree of the polynomial is not small. The
  second is that our algorithms rely on \emph{a priori} knowledge of
  whether or not $\beta=\beta_{c}$.

  Both of these deficiencies have the potential to be addressed by
  Glauber-type dynamics as described in~\cite{PolymerMarkov}; see
  also~\cite[Section~7.2]{helmuth2018contours}.  Proving the
  efficiency of these proposed algorithms would be very interesting.
\item Our deterministic algorithms for $\beta>\beta_{c}$ (and
  $\beta<\beta_{c}$) have diverging running times as
  $\beta\downarrow\beta_{c}$ ($\beta\uparrow\beta_{c}$).  Are there
  \emph{deterministic} algorithms that do not suffer from this
  dependence?
 \item The algorithmic adaptation of other sophisticated contour-based
  methods, e.g.,~\cite{peled2018rigidity}, would be also be quite
  interesting, particularly for applications to problems such as counting the number of 
  proper $q$-colorings of a graph. For recent progress on approximation
  algorithms for $q$-colorings,
  see~\cite{Liu2Delta2019,bencs2018zero,JenssenAlgorithmsJ,liao2019counting}. 
\end{enumerate}

\section*{Acknowledgements}
Part of this work was done while WP and PT were visiting Microsoft
Research New England.  Part of this work was done while TH and WP were
visiting the Simons Institute for the Theory of Computing.  TH was supported by
EPSRC grant EP/P003656/1.  WP is supported in part by NSF grants DMS-1847451 and CCF-1934915.  PT is supported in part by the NSF grant DMS-1811935.  We thank Guus Regts and Ewan Davies for helpful comments on a draft of this paper.

\appendix

\section{Coupling the Potts and random cluster models}
\label{sec:ES}

Here we review the standard Edwards--Sokal coupling between the Potts
and random cluster models and indicate how one can obtain counting and
sampling algorithms for the Potts model from counting and sampling
algorithms for the random cluster model.  For more details on the
couplings between the Potts model and random-cluster measures,
see~\cite[Section~1.2.2]{DuminilCopinLectures}.

Let $G= (V,E(G))$ be a finite graph.  Then the standard Edwards--Sokal
coupling put the $q$-color Potts model at inverse temperature $\beta$
on the same probability space as the random cluster model with
parameters $q$ and $p = 1- e^{-\beta}$.  To obtain a Potts
configuration we sample a random cluster configuration $A$, then
assign one of the $q$ colors uniformly at random to each of the
connected components of the graph $G_A = (V,A)$; note that isolated vertices are
connected components.  Each vertex is then assigned the color of its
connected component.  This gives an efficient algorithm to sample from
the Potts model given a sample from the random cluster model.
Moreover, 
\begin{equation}
  Z^{\text{Potts}}_{G}(\beta) = e^{\beta |E(G)|} Z^{\text{RC}}_G(1-e^{-\beta},q)\, ,
\end{equation}
which gives us an FPTAS (FPRAS) for $Z^{\text{Potts}}$ given an
FPTAS (FPRAS) for $Z^{\text{RC}}$.

We can also couple the Potts model with monochromatic boundary
conditions to the random cluster model with wired boundary conditions.
For this, let us specialize to finite induced subgraphs
$(\Lam, E(\Lam))$ of $\Z^d$.  Define the boundary of $\Lam$ to be
$\partial \Lam \bydef \{ i \in \Lam: \exists j \in \Lam^c, (i,j) \in
E(\Z^d) \}$. Recall the definition of the random cluster model
$\mu^{f}_{\Lam}$ with wired boundary conditions from Section~\ref{sec:random-cluster-model}. Given
a color $r \in [q]$, the allowed colorings for the Potts model with
$r$-monochromatic boundary conditions on $\Lam$ are
\begin{equation}
\Omega_r(\Lam) = \left \{ \sigma \in [q]^\Lam : \sigma_v = r \, \forall \, v \in \partial \Lam   \right \}  \,.
\end{equation}
The corresponding Gibbs measure and partition function are:
\begin{align*}
\mu_\Lam^{\text{Potts},r}( \sigma)  &= \frac{ \prod_{(i,j)\in E(\Lam)} e^{-\beta \mathbf 1_{\sigma_i \ne \sigma_j}}   }{Z_\Lam^{\text{Potts},r}(\beta)   }  \, ,  \quad \quad  \sigma \in \Omega_r(\Lam)  \\
Z_\Lam^{\text{Potts},r}(\beta) &=  \sum_{  \sigma \in \Omega_r(\Lam)}  e^{-\beta \mathbf 1_{\sigma_i \ne \sigma_j}}   \, .
\end{align*}

A simple extension of the Edwards-Sokal coupling then gives the
following facts.  Given a sample $A$ from $\mu^{w}_\Lam$ one can
obtain a sample from $\mu_\Lam^{\text{Potts},r}$ by coloring all
vertices in $\partial \Lam$ or connected to $\partial \Lam$ by the
edges in $A$ with color $r$, and assigning one of the $q$ colors
uniformly at random to the remaining connected components of the graph
$(\Lam, A)$.  Moreover, we have the relation
\begin{equation}
  q Z_\Lam^{\text{Potts},r} (\beta) =   e^{-\beta | E(\Lam)|} Z^w_\Lam( (1- e^{-\beta},q)  \,.
\end{equation}
Again this shows that efficient counting and sampling algorithms for
the Potts model with monochromatic boundary conditions follow from
efficient counting and sampling algorithms for the random cluster
model with wired boundary conditions.

\section{Proofs for $\beta_{h}<\beta<\beta_{c}$}
\label{sec:HT}

\subsection{Lemma~\ref{lemOrdDiscompute} (iii)}
\label{sec:lemma-refl-iii}

The proof of Lemma~\ref{lemOrdDiscompute} in the case
$\beta_h < \beta <\beta_{c}$ is the same, \emph{mutatis mutandis}, as
for $\beta>\beta_{c}$. The necessary changes are that (i) the roles of
the ordered and disordered contours are exchanged, and (ii) some of
the ingredients from Sections~\ref{sec:count} and~\ref{sec:sample}
were stated only for $\beta>\beta_{c}$, and hence versions for
$\beta_{h}<\beta<\beta_{c}$ are necessary. We outline how to obtain
these versions here.

As explained in~\cite[Appendix~A]{borgs2012tight}, \cite[Lemma~6.3 (i)
and (ii)]{borgs2012tight} applies when~\cite[(A.1)]{borgs2012tight}
holds. In fact, the arguments apply if
\begin{equation}
  \label{eq:A1p}
  \beta \geq \max \left \{ C_{1}\log (dC), \frac{3\log q}{4d} \right \}
\end{equation}
where $C$ is the constant from~\cite[Lemma~5.8]{borgs2012tight} and
$C_{1}$ is a sufficiently large constant depending only on $d$. To
verify this it is enough to check that~\cite[(A.2)]{borgs2012tight}
holds (up to a change in the constant $8$).\footnote{Our choice of
  $3/4$ in~\eqref{eq:A1p} is somewhat arbitrary; the same conclusion
  would hold for any number strictly larger than $2/3$.} Thus for
$q_{0}$ sufficiently large \cite[Lemma~6.3 (i) and
(ii)]{borgs2012tight} apply when $\beta_{h}<\beta<\beta_{c}$. In
particular, by following the proofs from $\beta>\beta_{c}$ we obtain
that when $\beta_{h}<\beta<\beta_{c}$
\begin{enumerate}
\item the conclusions of Lemma~\ref{lemBCTsup} hold with the roles of
  $\ord$ and $\dis$ reversed. The fact that
  $a_{\ord}>0$ is contained in~\cite[Lemma~A.3]{borgs2012tight}.
\item the conclusion of Lemma~\ref{lemSuperEstimates} holds with
  $\ord$ replaced by $\dis$.
\item the conclusion of Lemma~\ref{lemSuperEstimates2} holds with
  the roles of $\ord$ and $\dis$ reversed and
  \begin{equation*}
    M\geq \frac{2}{a_{\ord}}\log \frac{8q}{\eps} +
    \frac{2}{a_{\ord}}(\kappa+4)\norm{\V\gamma}. 
  \end{equation*}
  The factor four (as opposed to three) in $M$ arises in the
  computation of the lower bound on $Z^{\text{flip}}_{\ord}(\V\Lam,M)$,
  as (in the notation of the proof of Lemma~\ref{lemSuperEstimates2})
  $\Ext \Contour$ may be of size $\norm{\V{\gamma}}$.
\end{enumerate}

Lastly, the conclusion of Proposition~\ref{prop:vacant} holds with
$\dis$ changed to $\ord$. The proof is very similar to the proof of
Proposition~\ref{prop:vacant}, but using Lemmas~\ref{lem:edge-dis}
and~\ref{lem:dis-edge-con} in place of Lemmas~\ref{lem:edge-ord}
and~\ref{lem:ord-edge-con}.

\begin{proof}[Proof of Lemma~\ref{lemOrdDiscompute} (iii)]
  Using the ingredients above, this follows exactly as in the proof of
  Lemma~\ref{lemOrdDiscompute} (ii), i.e., for $\beta>\beta_{c}$.
\end{proof}

\subsection{Theorems~\ref{PottsTorusCrit} and~\ref{PottsZd}}

These proofs are exactly as for $\beta>\beta_{c}$ provided the
conclusions of Lemma~\ref{lemSuperEstimates2a} hold with $\dis$ replaced by
$\ord$. This is straightforward to obtain by imitating the proof of
Lemma~\ref{lemSuperEstimates2a}, using (as discussed in the previous
section) that the conclusion of Lemma~\ref{lemBCTsup} hold with the
roles of $\ord$ and $\dis$ reversed.

\section{Contour computations using subgraphs of $\htors$}
\label{app:subcomp}
The next lemma shows that computations relating to contours
$\V{\gamma}$ can be implemented using only $\gamma$, the connected
subgraph of $\htors$ that corresponds to $\V{\gamma}$ by the
construction in Section~\ref{sec:cont}. 

\begin{lemma}
  \label{lem:subgcomp}
  Let $\V{\gamma}$ and $\V{\gamma}'$ be contours, and let $\gamma$ and
  $\gamma'$ be the corresponding subgraphs of $\htors$. Then given
  $\gamma$, $\gamma'$,
  \begin{enumerate}
  \item $d_{\infty}(\V{\gamma},\V{\gamma}')$ can be computed in time
    $O(\abs{V(\gamma)}\abs{V(\gamma')})$, 
  \item The set $\Int \V{\gamma} \cap \tor$ can be computed in time
    $O(\abs{V(\gamma)}^{3})$,
  \item $\norm{\V{\gamma}}$ can be computed in time $O(\abs{V(\gamma)})$.
  \end{enumerate}
\end{lemma}
\begin{proof}
  Each vertex in $\htors$ corresponds to a $(d-1)$-dimensional
  hypercube in $\ctor$. For each pair of such hypercubes we can
  compute the distance between them in constant time, which implies
  the first claim. The third claim follows similarly, since the set of
  edges passing through a given $(d-1)$-dimensional hypercube can be
  determined in constant time.

  For the second claim, we first determine the set of edges
  intersecting the $(d-1)$-dimensional hypercubes corresponding to
  $\V{\gamma}$. We can then determine $\Int \V{\gamma}\cap \tor$ in time
  $O(\norm{\V{\gamma}}^{3})$ by Lemma~\ref{lem:findext}\,.
\end{proof}


\begin{thebibliography}{10}

\bibitem{alexander2004mixing}
K.~S. Alexander.
\newblock Mixing properties and exponential decay for lattice systems in finite
  volumes.
\newblock {\em The Annals of Probability}, 32(1A):441--487, 2004.

\bibitem{barvinok2017combinatorics}
A.~Barvinok.
\newblock Combinatorics and complexity of partition functions.
\newblock {\em Algorithms and Combinatorics}, 30, 2017.

\bibitem{barvinok2017weighted}
A.~Barvinok and G.~Regts.
\newblock Weighted counting of solutions to sparse systems of equations.
\newblock {\em Combinatorics, Probability and Computing}, 28(5):696--719, 2019.

\bibitem{bencs2018zero}
F.~Bencs, E.~Davies, V.~Patel, and G.~Regts.
\newblock On zero-free regions for the anti-ferromagnetic {P}otts model on
  bounded-degree graphs.
\newblock {\em Annales de l'Institut Henri Poincare (D) Combinatorics, Physics
  and their Interactions}, 2019.

\bibitem{blanca2017random}
A.~Blanca and A.~Sinclair.
\newblock Random-cluster dynamics in $\mathbb{Z}^2$.
\newblock {\em Probability Theory and Related Fields}, 168(3-4):821--847, 2017.

\bibitem{BorgsStoc2020}
C.~Borgs, J.~Chayes, T.~Helmuth, W.~Perkins, and P.~Tetali.
\newblock Efficient sampling and counting algorithms for the {P}otts model on
  $\mathbb{Z}^d$ at all temperatures (extended abstract).
\newblock In {\em Proceedings of the 52nd Annual ACM SIGACT Symposium on Theory
  of Computing}, STOC 2020, page 738–751, New York, NY, USA, 2020.
  Association for Computing Machinery.

\bibitem{BorgsChayesKahnLovasz}
C.~Borgs, J.~Chayes, J.~Kahn, and L.~Lov{\'a}sz.
\newblock Left and right convergence of graphs with bounded degree.
\newblock {\em Random Structures \& Algorithms}, 42(1):1--28, 2013.

\bibitem{borgs2012tight}
C.~Borgs, J.~T. Chayes, and P.~Tetali.
\newblock Tight bounds for mixing of the {S}wendsen--{W}ang algorithm at the
  {P}otts transition point.
\newblock {\em Probability Theory and Related Fields}, 152(3-4):509--557, 2012.

\bibitem{borgs1989unified}
C.~Borgs and J.~Z. Imbrie.
\newblock A unified approach to phase diagrams in field theory and statistical
  mechanics.
\newblock {\em Communications in Mathematical Physics}, 123(2):305--328, 1989.

\bibitem{borgs1991finite}
C.~Borgs, R.~Koteck{\'y}, and S.~Miracle-Sol{\'e}.
\newblock Finite-size scaling for {P}otts models.
\newblock {\em Journal of Statistical Physics}, 62(3-4):529--551, 1991.

\bibitem{cannon2019counting}
S.~Cannon and W.~Perkins.
\newblock Counting independent sets in unbalanced bipartite graphs.
\newblock In {\em Proceedings of the Fourteenth Annual ACM-SIAM Symposium on
  Discrete Algorithms (SODA)}, pages 1456--1466. SIAM, 2020.

\bibitem{casel2019zeros}
K.~Casel, P.~Fischbeck, T.~Friedrich, A.~G{\"o}bel, and J.~Lagodzinski.
\newblock Zeros and approximations of {H}olant polynomials on the complex
  plane.
\newblock {\em arXiv preprint arXiv:1905.03194}, 2019.

\bibitem{PolymerMarkov}
Z.~Chen, A.~Galanis, L.~A. Goldberg, W.~Perkins, J.~Stewart, and E.~Vigoda.
\newblock Fast algorithms at low temperatures via {M}arkov chains.
\newblock {\em Random Structures \& Algorithms}, 58(2):294--321, 2021.

\bibitem{DuminilCopinLectures}
H.~Duminil-Copin.
\newblock Lectures on the {I}sing and {P}otts models on the hypercubic lattice.
\newblock In {\em PIMS-CRM Summer School in Probability}, pages 35--161.
  Springer, 2017.

\bibitem{DuminilCopinRaoufiTassion}
H.~Duminil-Copin, A.~Raoufi, and V.~Tassion.
\newblock Sharp phase transition for the random-cluster and {P}otts models via
  decision trees.
\newblock {\em Annals of Mathematics}, 189(1):75--99, 2019.

\bibitem{dyer2004relative}
M.~Dyer, L.~A. Goldberg, C.~Greenhill, and M.~Jerrum.
\newblock The relative complexity of approximate counting problems.
\newblock {\em Algorithmica}, 38(3):471--500, 2004.

\bibitem{friedli2017statistical}
S.~Friedli and Y.~Velenik.
\newblock {\em Statistical Mechanics of Lattice Systems: a Concrete
  Mathematical Introduction}.
\newblock Cambridge University Press, 2017.

\bibitem{galanis2016inapproximability}
A.~Galanis, D.~{\v{S}}tefankovi{\v{c}}, and E.~Vigoda.
\newblock Inapproximability of the partition function for the antiferromagnetic
  {I}sing and hard-core models.
\newblock {\em Combinatorics, Probability and Computing}, 25(4):500--559, 2016.

\bibitem{galanis2016ferromagnetic}
A.~Galanis, D.~Stefankovic, E.~Vigoda, and L.~Yang.
\newblock Ferromagnetic {P}otts model: Refined \#-{BIS}-hardness and related
  results.
\newblock {\em SIAM Journal on Computing}, 45(6):2004--2065, 2016.

\bibitem{gheissari2018mixing}
R.~Gheissari and E.~Lubetzky.
\newblock Mixing times of critical two-dimensional {P}otts models.
\newblock {\em Communications on Pure and Applied Mathematics},
  71(5):994--1046, 2018.

\bibitem{gheissari2016quasi}
R.~Gheissari and E.~Lubetzky.
\newblock Quasi-polynomial mixing of critical two-dimensional random cluster
  models.
\newblock {\em Random Structures \& Algorithms}, 56(2):517--556, 2020.

\bibitem{gruber1971general}
C.~Gruber and H.~Kunz.
\newblock General properties of polymer systems.
\newblock {\em Communications in Mathematical Physics}, 22(2):133--161, 1971.

\bibitem{guo2018}
H.~Guo and M.~Jerrum.
\newblock Random cluster dynamics for the {I}sing model is rapidly mixing.
\newblock {\em Ann. Appl. Probab.}, 28(2):1292--1313, 04 2018.

\bibitem{HJP}
T.~Helmuth, M.~Jenssen, and W.~Perkins.
\newblock Finite-size scaling, phase coexistence, and algorithms for the random
  cluster model on random graphs.
\newblock {\em arXiv preprint arXiv:2006.11580}, 2020.

\bibitem{helmuth2018contours}
T.~Helmuth, W.~Perkins, and G.~Regts.
\newblock Algorithmic {P}irogov-{S}inai theory.
\newblock {\em Probability Theory and Related Fields}, 176:851--895, 2020.

\bibitem{JenssenAlgorithmsJ}
M.~Jenssen, P.~Keevash, and W.~Perkins.
\newblock Algorithms for \#{BIS}-hard problems on expander graphs.
\newblock {\em SIAM Journal on Computing}, 49(4):681--710, 2020.

\bibitem{jerrum1989approximating}
M.~Jerrum and A.~Sinclair.
\newblock Approximating the permanent.
\newblock {\em SIAM Journal on Computing}, 18(6):1149--1178, 1989.

\bibitem{jerrum1993polynomial}
M.~Jerrum and A.~Sinclair.
\newblock Polynomial-time approximation algorithms for the {I}sing model.
\newblock {\em SIAM Journal on Computing}, 22(5):1087--1116, 1993.

\bibitem{Kotecky}
R.~Koteck{\`y}.
\newblock Pirogov-sinai theory.
\newblock {\em Encyclopedia of Mathematical Physics}, 4:60--65, 2006.

\bibitem{kotecky1986cluster}
R.~Koteck\'{y} and D.~Preiss.
\newblock Cluster expansion for abstract polymer models.
\newblock {\em Communications in Mathematical Physics}, 103(3):491--498, 1986.

\bibitem{kotecky1982first}
R.~Koteck{\`y} and S.~Shlosman.
\newblock First-order phase transitions in large entropy lattice models.
\newblock {\em Communications in Mathematical Physics}, 83(4):493--515, 1982.

\bibitem{laanait1991interfaces}
L.~Laanait, A.~Messager, S.~Miracle-Sol{\'e}, J.~Ruiz, and S.~Shlosman.
\newblock Interfaces in the {P}otts model {I}: {P}irogov-{S}inai theory of the
  {F}ortuin-{K}asteleyn representation.
\newblock {\em Communications in Mathematical Physics}, 140(1):81--91, 1991.

\bibitem{liao2019counting}
C.~Liao, J.~Lin, P.~Lu, and Z.~Mao.
\newblock Counting independent sets and colorings on random regular bipartite
  graphs.
\newblock In {\em Approximation, Randomization, and Combinatorial Optimization.
  Algorithms and Techniques (APPROX/RANDOM 2019)}. Schloss
  Dagstuhl-Leibniz-Zentrum fuer Informatik, 2019.

\bibitem{Liu2Delta2019}
J.~Liu, A.~Sinclair, and P.~Srivastava.
\newblock A deterministic algorithm for counting colorings with 2-{D}elta
  colors.
\newblock In {\em 2019 IEEE 60th Annual Symposium on Foundations of Computer
  Science (FOCS)}, pages 1380--1404. IEEE, 2019.

\bibitem{martinelli19942}
F.~Martinelli, E.~Olivieri, and R.~H. Schonmann.
\newblock For 2-d lattice spin systems weak mixing implies strong mixing.
\newblock {\em Communications in Mathematical Physics}, 165(1):33--47, 1994.

\bibitem{peled2018rigidity}
R.~Peled and Y.~Spinka.
\newblock Rigidity of proper colorings of $\mathbb{Z}^d$.
\newblock {\em arXiv preprint arXiv:1808.03597}, 2018.

\bibitem{pirogov1975phase}
S.~A. Pirogov and Y.~G. Sinai.
\newblock Phase diagrams of classical lattice systems.
\newblock {\em Theoretical and Mathematical Physics}, 25(3):1185--1192, 1975.

\bibitem{sly2010computational}
A.~Sly.
\newblock Computational transition at the uniqueness threshold.
\newblock In {\em Proceedings of the Fifty-first Annual IEEE Symposium on
  Foundations of Computer Science}, FOCS 2010, pages 287--296. IEEE, 2010.

\bibitem{sly2014counting}
A.~Sly and N.~Sun.
\newblock Counting in two-spin models on d-regular graphs.
\newblock {\em The Annals of Probability}, 42(6):2383--2416, 2014.

\bibitem{vstefankovivc2009adaptive}
D.~{\v{S}}tefankovi{\v{c}}, S.~Vempala, and E.~Vigoda.
\newblock Adaptive simulated annealing: A near-optimal connection between
  sampling and counting.
\newblock {\em Journal of the ACM (JACM)}, 56(3):18, 2009.

\bibitem{ullrich2013comparison}
M.~Ullrich.
\newblock Comparison of {S}wendsen-{W}ang and heat-bath dynamics.
\newblock {\em Random Structures \& Algorithms}, 42(4):520--535, 2013.

\bibitem{weitz2006counting}
D.~Weitz.
\newblock Counting independent sets up to the tree threshold.
\newblock In {\em Proceedings of the Thirty-Eighth Annual ACM {S}ymposium on
  Theory of {C}omputing}, STOC 2006, pages 140--149. ACM, 2006.

\end{thebibliography}
\end{document}